\documentclass{amsart}
\usepackage{amsthm}
\usepackage{amsmath}
\usepackage{amsfonts}
\usepackage{amssymb}
\usepackage[mathscr]{euscript}
\usepackage[usenames]{color}
\usepackage[arrow,curve,frame,color]{xy}
\usepackage[noprefix]{nomencl}
\usepackage[hidelinks]{hyperref}
\def\ball#1,#2.{B(#1,#2)} % ball: #1 is the centre, #2 is the radius
\def\biglip{\pounds} %local lipschitz constant
\def\epsi{\varepsilon}%symbol for \varepsilon
\def\metric{d} %symbol for the metric
\def\dist#1,#2.{\metric(#1,#2)} %distance between two points
\def\setdist#1,#2.{{\rm dist}(#1,#2)}%distance between two sets
\def\hdist#1,#2.{\metric_H(#1,#2)}%Hausdorff distance between two sets
\def\varlip{{\rm Var}} %variation over a ball 
\def\natural{{\mathbb N}} %natural numbers
\def\smllip{\ell} %minimal variation
\def\clball#1,#2.{\bar B(#1,#2)} % closed ball: #1 is the centre, #2 is the radius
\def\albrep#1.{{\mathcal A}_{#1}} %symbol for an Alberti representation
\def\mpush#1.{{#1}_{\sharp}} %pusch forward of a measure
\def\glip#1.{{\bf L}(#1)} %global Lipschitz constant
\def\elleone#1.{{ L}^1(#1)} %L^1 functions: #1 is the measure
\def\elleinfty#1.{{L}^\infty(#1)} %L^\infty functions: #1 is the
\def\bborel#1.{{\mathcal B}^{\infty}(#1)} %bounded borel functions
\def\lipalg#1.{{\rm Lip}_{\text{\normalfont b}}(#1)} %Lipschitz algebra of bounded functions
\def\lipfun#1.{{\rm Lip}(#1)} %Lipschitz functions
\def\lipalgb#1,#2.{{\rm Lip}_{\text{\normalfont b},#2}(#1)}% bounded
\def\onenorm#1.{\left\|#1\right\|_{\elleone\mu.}} %symbol for the L^1 norm
\def\real{{\mathbb{R}}}%real numbers
\def\mrest{{\,\evansgariepyrest}}%restriction of a measure
\def\arenseels#1.{{\rm AE}[#1]}%Arens-Eels space
\def\cone{{\mathcal C}} %symbol of a cone
\def\lebmeas{{\mathcal L}^1} %lebesgue measure
\def\pathfrag#1.{{\rm Frag}(#1)} %path fragments
\def\hmeas#1.{{\mathscr{H}^{#1}}} %hausdorff measure
\def\pullb#1.{{#1}^{\sharp}} %pullback
\def\cmass#1.{{\|#1\|}}%current mass
\def\mcurr#1,#2.{{\bf M}_{#1}(#2)} %currents with finite mass
\def\mcurr#1,#2.{{\bf M}_{#1}(#2)} %currents with finite mass                   
\def\ncurr#1,#2.{{\bf N}_{#1}(#2)} % normal currents                            
\def\pcurr#1,#2.{{\bf P}_{#1}(#2)} % precurrents                                
\def\mfunc#1,#2.{{\rm MF}_{#1}(#2)} %metric functionals of ambrosio-kirchheim   
\def\bana{{\mathscr{S}}}%banach space
\def\convgeo{{\mathscr{K}}}%the convex set used for constructing
\def\QG#1,#2.{{\rm QG}(#1,#2)}%space of quasigeodesic 
\def\tnorm#1.{{{\left\|#1\right\|_{TX}}}}%norm on the tangent bundle
\def\cotnorm#1.{{{\left\|#1\right\|_{T^*X}}}}%norm on the cotangent
\def\metuple#1,#2.{{\mathscr{D}}^{#1}({#2})}%(k+1)-tuple of Lipschitz
\newcommand{\chartfun}[1][j]{x_\alpha^j}%individual chart function
\DeclareMathOperator\spt{spt} %support of a function, measure, etc...
\DeclareMathOperator\dom{dom} %dominium of something, e.g. of a Lipschitz curve
\DeclareMathOperator\tang{Tan} %tangent cones
\DeclareMathOperator\metdiff{md} %metric differential
\DeclareMathOperator\frags{Frag} %fragments
\DeclareMathOperator\diam{diam}%diameter
\DeclareMathOperator\sgn{sgn}%signum function
\DeclareMathOperator\im{im}%image 
\DeclareMathOperator\graph{Graph}%graph of a function
\def\balt#1,#2.{{\mathcal{BA}\left(#1,#2\right)}}%bounded multilinear
\def\balta#1,#2,#3.{{\mathcal{BA}_{(#3)}\left(#1,#2\right)}}%bounded
\def\boundlin#1,#2.{{\mathcal{BL}\left(#1,#2\right)}}%bounded linear
\def\extpow#1.{\widehat\bigwedge^{#1}}%exterior projective
\def\extpowm#1,#2.{\widehat\bigwedge^{#1}_{#2}}%exterior projective
\def\albcond{{\normalfont (COND)}}%Conditions on alberti
\def\frest#1.{{\rm Red}_{#1}}%Restriction for fragments
\def\dualmod#1.{\hom\left({#1},L^\infty(\mu)\right)}%dual module for
\def\boundlin#1,#2.{{\mathcal{BL}\left(#1,#2\right)}}%bounded linear
\def\locdernorm#1,#2.{{\left|{#1}\right|_{{\rm  Der}({#2}),\text{loc}}}}%local
\def\grass#1,#2.{{\normalfont\rm Gr}({#1},{#2})}%Grassmannian
\newcommand{\evansgariepyrest}{\mathbin{\vrule height 1.3ex depth%
    0pt width 0.08ex\vrule height 0.08ex depth 0pt width 1.0ex}}
\DeclareMathOperator\haus{Haus}%space of compact subsets
\DeclareMathOperator\linspan{Span}%for linear span
\def\balt#1,#2.{{\mathcal{BA}\left(#1,#2\right)}}%bounded multilinear
\def\balta#1,#2,#3.{{\mathcal{BA}_{(#3)}\left(#1,#2\right)}}%bounded
\def\boundlin#1,#2.{{\mathcal{BL}\left(#1,#2\right)}}%bounded linear
\def\extpow#1.{{\widehat\bigwedge^{#1}}}%exterior projective power
\def\extpowm#1,#2.{\widehat\bigwedge^{#1}_{#2}}%exterior projective
\def\strip#1,#2.{{\rm St}\left({#1},{#2}\right)}%strip of width #2
\def\wder#1.{{\mathscr{X}}({#1})}%L^\infty module of derivations
\def\wform#1.{{\mathscr{E}}({#1})}%L^\infty module of forms
\def\locnorm#1,#2.{\left|{#1}\right|_{{#2},\text{\normalfont loc}}}%local norm of
\def\rational{{\mathbb Q}}%rational numbers
\def\zahlen{{\mathbb Z}}%signed integers
\def\ccgengroup#1.{{\mathbb {#1}}}%Carnot group general letter
\def\liegenalg#1.{{\mathfrak{#1}}}%lie algebra general letter
\DeclareMathOperator\Cyl{Cyl}%metric cylinder
\DeclareMathOperator\Gap{Gap}%class gap
\DeclareMathOperator\Alb{Alb}%collection of Alberti representations
\DeclareMathOperator\Der{Der}%map from Alberti representations to measures
\numberwithin{equation}{section}
\theoremstyle{plain} %to change the headings as suggested by Kleiner
\newtheorem{lem}[equation]{Lemma}

\newtheorem{thm}[equation]{Theorem}
\newtheorem{cor}[equation]{Corollary}
\theoremstyle{definition}
\newtheorem{defn}[equation]{Definition}
\newtheorem{assump}[equation]{Assumption}
\theoremstyle{remark}
\newtheorem{exa}[equation]{Example}
\newtheorem{rem}[equation]{Remark}
\makenomenclature
\newcount\sync %to syncronize numeration of Theorems
\begin{document}
\title{Derivations and Alberti representations}
\author{Andrea Schioppa}
\address{ETH, R\"amistrasse 101, 8092, Z\"urich, CH}
\email{andrea.schioppa@math.ethz.ch}
%\ead{andrea.schioppa@math.ethz.ch}
%\thanks{During the revision phase of this paper the author was supported by the ``ETH Zurich Postdoctoral Fellowship Program and the Marie Curie Actions
 % for People COFUND Program''}
%\keywords{derivation, Lipschitz algebra, measurable
 % differentiable structure}
%\subjclass[2010]{53C23, 46J15, 58C20}
%53C23=differential geometrical analysis on metric spaces
%46J15=banach algebras of differentiable or analytic functions
%58C20=differentiation theory
\begin{abstract}
  We relate generalized Lebesgue decompositions of measures in terms
  of curve fragments (``Alberti representations'') and Weaver
  derivations. This correspondence leads to a geometric
  characterization of the local norm on the Weaver cotangent bundle of
  a metric measure space $(X,\mu)$: the local norm of a form $df$
  ``sees'' how fast $f$ grows on curve fragments ``seen'' by
  $\mu$. This implies a new characterization of differentiability
  spaces in terms of the $\mu$-a.e.~equality of the local norm of $df$
  and the local Lipschitz constant of $f$. As a consequence, the
  ``Lip-lip'' inequality of Keith must be an equality. We also provide
  dimensional bounds for the module of derivations in terms of the
  Assouad dimension of $X$.
\end{abstract}
\maketitle
\tableofcontents %table of contents
\section{Introduction}
\ifnum\sync=1{
  \subsection*{SyncTable}
  Theorem \textcolor{red}{alb\_glue}: \ref{alb_glue}\\
  Theorem \textcolor{red}{alberti\_rep\_prod}:
  \ref{alberti_rep_prod}\\
  Lemma \textcolor{red}{lem:meas\_fix}: \ref{lem:meas_fix}\\
  Theorem \textcolor{red}{thm:alb\_derivation}
  \ref{thm:alb_derivation}\\
  Lemma \textcolor{red}{partialderivatives} \ref{partialderivatives}\\
  Remark \textcolor{red}{rem:derivation\_extension}
  \ref{rem:derivation_extension}\\ 
  Corollary \textcolor{red}{derbound} \ref{derbound}\\
  Theorem \textcolor{red}{onedimapprox\_multi}
  \ref{onedimapprox_multi}\\
  Lemma \textcolor{red}{lem:loc\_estimate\_dist}
  \ref{lem:loc_estimate_dist}\\
  Lemma \textcolor{red}{lem:loc\_estimate\_fnull}
  \ref{lem:loc_estimate_fnull}\\
  Corollary \textcolor{red}{cor:mu\_arb\_cone} \ref{cor:mu_arb_cone}\\
  Lemma \textcolor{red}{biLip\_dis} \ref{biLip_dis}\\
  Corollary \textcolor{red}{cor:assouad\_bound}
  \ref{cor:assouad_bound}\\
  Theorem \textcolor{red}{thm:weak*density} \ref{thm:weak*density}\\
  Lemma \textcolor{red}{borel\_rest} \ref{borel_rest}\\
  Theorem \textcolor{red}{compact\_reduction}
  \ref{compact_reduction}\\
  Lemma \textcolor{red}{lem:vector\_alberti} \ref{lem:vector_alberti}\\
  
  }\fi
  \subsection*{Overview}
This paper studies the differentiability properties of real-valued Lipschitz
functions defined on separable metric measure spaces. Two seminal
works in this field are due to Cheeger \cite{cheeger99} and Weaver
\cite{weaver00}.
\par In \cite{cheeger99} Cheeger formulated a
generalization of Rademacher's differentiability Theorem for metric measure spaces
admitting a Poincar\'e inequality %\textcolor{red}
{(this is an analytic condition that
has been intoduced in \cite{heinonen98} and has proven useful to
generalize notions of 
calculus on metric measure spaces; knowing about the Poincar\'e
inequality is \emph{not} a prerequisite for understanding this paper)}; a metric measure space satisfying
the \emph{conclusion} of Cheeger's result is often called a
\emph{(Lipschitz) differentiability space} or is said to have a
\emph{(measurable / strong measurable) differentiable structure}. Applications of
differentiability spaces include the study of Sobolev and
quasiconformal maps in metric measure spaces~\cite{balogh_stepanov,keith04bis} and the study of metric
embeddings \cite{cheeger_kleiner_radon,cheeger99}. Recently Bate \cite{bate-diff}
approached this subject from a different angle by showing that 
differentiability spaces have a rich \emph{$1$-rectifiable structure}, which can be
described in terms of Fubini-like representations of the measure which are called
\emph{Alberti representations} or \emph{$1$-rectifiable
  representations} \cite{acp_proceedings}.
\par Even though there are many examples of differentiability spaces,
the notion of differentiable structure is rather restrictive. For
example, consider $\real^2$ with the metric:
\begin{equation}
  \label{eq:examet}
  d((x_1,y_1),(x_2,y_2))=|x_1-x_2|+|y_1-y_2|^{1/2};
\end{equation}
the existence of nowhere differentiable H\"older functions (for
example the classical Weierstrass function, see~\cite[Exa.~11.3]{falconer_fractal}) in the
$y$-direction can be used to show that $(\real^2,d,\mu)$ is never a
differentiability space for any choice of $\mu$.  
However, for many
measures, e.g.~for the Lebesgue measure, there is a good notion of
differentiation, or vector field, in the $x$-direction.
\par This example
can be better understood using Weaver's approach
\cite{weaver00}  to differentiability,
which is motivated by the study of Lipschitz algebras. This approach is,
roughly speaking, based on the idea of defining \emph{measurable
  vector fields} (called \emph{derivations}) as operators acting on
Lipschitz functions. Even though Weaver's
approach is more flexible than Cheeger's, there are fewer works on
this topic \cite{gong_rigidity,derivdiff} and, apart from specific
examples, it seemed unclear whether it would be possible to
obtain a geometric description of derivations for a general Radon
measure $\mu$ on a metric space $X$.
\par The main achievement of this paper is to provide a general
approach to differentiability that can be applied to \emph{any} Radon measure
defined on a complete separable metric space; this approach unifies
Weaver's theory with the study of Alberti representations and gives a
geometric description of measurable vector fields and $1$-forms on
metric measure spaces.
\par Even though we build on ideas introduced in
\cite{acp_proceedings,bate-diff}, we have to overcome significant obstacles, most notably the fact that we \emph{do not} assume the
existence of a differentiable structure. On the analytic side, this
prevents the use of the porosity techniques used in
\cite[Sec.~9]{bate-diff}, which can be applied only to
(asymptotically) doubling measures. On the algebraic side, we cannot
use finite-dimensionality arguments as in \cite{cheeger99} as the
measurable tangent and cotangent bundles might be infinite-dimensional.
\par Applications of this theory include a sharp bound on the number
of generators of the module of derivations (morally the ``dimension''
of the measurable tangent bundle) in terms of the Assouad dimension of
the support of the underlying measure, and the proof that
Keith's \emph{Lip-lip inequality} \cite{keith04} self-improves to an
\emph{equality}. This last application corresponds to one of the
deepest results in \cite[Sec.~5,6]{cheeger99}; as we do not assume a
Poincar\'e inequality, a conceptually new argument is
required.
\par This theory can also be applied to describe the structure of
\emph{metric currents} \cite{curr_alb}; in particular, metric currents
carry a measurable tangent bundle and their action on Lipschitz
functions can be described using vector fields on this tangent bundle
and $1$-forms on the dual bundle. The study of metric currents
illustrates also how one can use derivations to better understand
Alberti representations and, vice versa, how one can use Alberti
representations to better understand derivations. For example, in
\cite{curr_alb} it is shown that in quasiconvex metric spaces
$1$-dimensional metric currents are \emph{flat} in the sense that they
are limits of normal currents in the mass norm. This application is
based on associating a derivation to a $1$-dimensional current, and
then using Alberti representations to produce the approximating
sequence of normal currents. On the other hand, in \cite{curr_alb} it
is also shown on how to obtain Alberti representations in the
directions of measurable vector fields. The argument is based on a renorming
argument that shows that the Weaver norm on the tangent bundle can be
taken to be strictly convex.
\par A more recent application of this theory is the proof that at
generic points blow-ups / tangents of differentiability spaces are
still differentiability spaces, see~\cite{dim_blow}. Interestingly,
that proof goes by contrapositive, and this theory comes to rescue as
it works without the assumption of a differentiable structure.
\subsection*{Alberti representations}
In this subsection we give an informal account of Alberti representations.
 This tool has
played an important r\^ole in the proof of the rank-one property of BV
functions \cite{alberti_rank_one}, in understanding the structure of
measures which satisfy the conclusion of the classical Rademacher's
Theorem \cite{acp_proceedings,alberti_marchese}, and in describing the
structure of measures in differentiability spaces \cite{bate-diff}.
We start with a description in vague terms and refer the reader to
\cite{bate-diff} and Subsection \ref{subsec:alberti} for further
details. 
\par An \textbf{Alberti representation} of a Radon measure
$\mu$ is a generalized Lebesgue decomposition of $\mu$ in terms
of rectifiable measures supported on path fragments;
a \textbf{path fragment} in $X$ is a bi-Lipschitz
map $\gamma\colon K\to X$ where $K\subset\real$ is compact with positive
Lebesgue measure. Denoting the set of fragments in $X$ by
$\frags(X)$,%
\nomenclature[frags]{$\frags(X)$}{parametrized biLipchitz curve fragments in $X$}%
\ an Alberti representation of $\mu$ takes the form:
\begin{equation}
\mu=\int_{\frags(X)}\nu_\gamma\,dP(\gamma)
\end{equation}
where $P$ is a regular Borel probability measure on $\frags(X)$ and
$\nu_\gamma\ll\hmeas 1._\gamma$; here $\hmeas 1._\gamma$ denotes the
$1$-dimensional Hausdorff measure on the image of $\gamma$. Note that in general
it is necessary to work with path fragments instead of Lipschitz paths
because the space $X$ on which $\mu$ is defined might lack any
rectifiable curve.
\begin{exa}
A simple example of an Alberti representation is offered by Fubini's
Theorem. Let $\hmeas n.$ denote the Lebesgue measure in $\real^n$ and,
for $y\in[0,1]^{n-1}$, let 
\begin{equation}
    \begin{aligned}
  \psi(y)&:[0,1]\to[0,1]^n\\
  t&\mapsto y+te_{n},
\end{aligned}
\end{equation}
$e_n$ denoting the last vector in the standard orthonormal basis of
$\real^n$. Then $P=\mpush\psi.\hmeas n-1.\mrest[0,1]^{n-1}$ is a regular
Borel probability measure on $\frags([0,1]^n)$; if we let
$\nu_\gamma=\hmeas 1._\gamma$, by Fubini's Theorem
\begin{equation}
  \hmeas n.\mrest[0,1]^n=\int_{\frags([0,1]^n)}\nu_\gamma\,dP(\gamma)
\end{equation} and so $(P,\nu)$ is an Alberti representation of
$\hmeas n.\mrest[0,1]^n$.
\end{exa}
\subsection*{Weaver derivations}
In this subsection we give an informal account of Weaver derivations,
hereafter simply called derivations, and refer the reader to
\cite{weaver_book99,weaver00} and Subsection \ref{subsec:der_modules}
for more details. Note that we need a less general setting than that
of Weaver because we deal with Radon measures on separable metric
spaces.
\par We will denote by $\lipfun X.$%
\nomenclature[lipfun]{$\lipfun X.$}{real-valued Lipschitz functions defined on $X$}%
\ the space of real-valued Lipschitz
functions defined on $X$ and by $\lipalg X.$%
\nomenclature[lipfun]{$\lipalg X.$}{bounded real-valued Lipschitz functions defined on
$X$}%
\ the real algebra of
real-valued bounded Lipschitz functions. The algebra $\lipalg X.$ is a
dual Banach space\footnote{there are two definitions of the
  predual in use \cite{arens_eels56,deleeuw61}, but they coincide: see \cite[pg.~40]{weaver_book99}}
  and therefore has a
well-defined weak* topology: a sequence $f_n\to f$ in the weak*
topology if and only if the global Lipschitz constants of the $f_n$
are uniformly bounded and $f_n\to f$ pointwise.
\par \textbf{Derivations} can be regarded as \textbf{measurable vector fields} in
metric spaces allowing one to take \emph{partial derivatives} of Lipschitz
functions once a background measure $\mu$ has been established. More
precisely, a \textbf{derivation}
is a weak* continuous bounded linear map $D:\lipalg X.\to
L^\infty(\mu)$ which satisfies the
product rule $D(fg)=fDg+gDf$.
The set of derivations forms an $L^\infty(\mu)$-module $\wder\mu.$.
\begin{exa}
Many examples of derivations can be found in
\cite[Sec.~5]{weaver00}. Considering $\real^n$, the partial derivatives
$\frac{\partial}{\partial x^i}$ induce derivations with respect to
the Lebesgue measure $\hmeas n.$%
\nomenclature[measure]{$\hmeas m.$}{$m$-dimensional Hausdorff measure}%
\nomenclature[measure]{$\mu\mrest A$}{restriction of the measure $\mu$ on $A$}%
\nomenclature[derivation]{$\wder\mu.$}{module of derivations}%
\ because of Rademacher's Theorem; this example generalizes to
Lipschitz manifolds \cite[Subsec.~5.B]{weaver00}. Similarly, for an 
$m$-rectifiable set $M\subset\real^n$, $\wder{\hmeas m.}\mrest M.$ can
be identified with the module of bounded measurable sections of
approximate tangent spaces \cite[Thm.~38]{weaver00}.
\end{exa}
\subsection*{Matching Alberti representations and Derivations}
We now start to describe the results obtained in this paper. Since
most results require quite a bit of preparatory material to be
properly stated, we state them here in a less formal way and provide a reference
to the corresponding more precise formulation later in the paper. Note that
we will assume metric spaces to be \textbf{separable}.
\par The first step is to use
Alberti representations to construct derivations.  We introduce the
notion of \textbf{Lipschitz} and \textbf{bi-Lipschitz} Alberti
representations: an Alberti representation $\albrep.=(P,\nu)$ is
called \textbf{$C$-Lipschitz (resp.~$(C,D)$-bi-Lipschitz)} if $P$ gives
full-measure to the set of $C$-Lipschitz (resp.~$(C,D)$-bi-Lipschitz)
fragments. In Subsection \ref{subsec:derivations_alberti} we show that
(Theorem \ref{thm:alb_derivation}):
\begin{itemize}
\item To a Lipschitz Alberti representation $\albrep .$ of $\mu$, one
  can associate a derivation $D_{\albrep.}\in\wder\mu.$%
\nomenclature[alberti]{$D_{\albrep.}$}{derivation associated to the
  Alberti representation $\albrep.$}%
\ by using duality and by taking an average
of derivatives along fragments \eqref{eq:derivation_alberti}.
\item Denoting by $\Alb_{{\rm sub}}(\mu)$%
\nomenclature[alberti]{$\Alb_{{\rm sub}}(\mu)$}{family of Lipschitz Alberti
representations of measures of the form  $\mu\mrest S$}%
\ the set of Lipschitz Alberti
representations of some measure of the form  $\mu\mrest S$ for $S\subset X$ Borel,
we obtain a map \begin{equation}\Der:\Alb_{{\rm sub}}(\mu)\to\wder\mu..
\end{equation}%
\nomenclature[alberti]{$\Der$}{map associating to an Alberti representation $\albrep.$  a
 derivation $D_{\albrep.}$}%
\end{itemize}
\par Note that the previous construction produces a wealth of
derivations; in fact, Subsection \ref{subsec:alberti}
provides a standard criterion (Theorem \ref{alberti_rep_prod}) which
allows to produce Alberti representations which are
$(1,1+\epsi)$-bi-Lipschitz: this is an improvement on the treatment in
\cite{bate-diff} and will play an important r\^ole in the rest of the paper.
We also point out that the choice of $\Alb_{{\rm sub}}(\mu)$ reflects
the fact that $\wder\mu.$ depends only on the measure class of $\mu$:
i.e.~if $\mu_1\ll\mu_2$ and $\mu_2\ll\mu_1$, $\wder\mu_1.=\wder\mu_2.$.
\par We next relate the notion of algebraic independence in
$\wder\mu.$ to a notion of independence for Alberti representations
introduced in \cite{bate-diff}\footnote{We replace the constant cones
  in \cite{bate-diff} by Borel cones}: if $f:X\to\real^n$ is Lipschitz
and $\cone$%
\ is an \textbf{$n$-dimensional cone field}, i.e.~a Borel map on $X$ which
takes values in the set of cones in $\real^n$ (see Definition
\ref{defn:cone}), an Alberti representation $\albrep.=(P,\nu)$ is in the \textbf{$f$-direction of
  $\cone$} if, for $P$-a.e.~fragment $\gamma$ and
$\lebmeas\mrest\dom\gamma$\nobreakdash-\hspace{0pt}a.e.~point $t$, $(f\circ\gamma)'(t)\in\cone(\gamma(t))$; cone
fields $\{\cone_i\}_{i=1}^k$ are \textbf{independent} if for each $x\in
X$, choosing $v_i\in\cone_i(x)\setminus\{0\}$, the $\{v_i\}_{i=1}^k$
are linearly independent.  In Subsection~\ref{subsec:derivations_alberti} we show that (Theorem
\ref{thm:directional_cone}):
\begin{itemize}
\item If $\albrep.$ is in the $f$-direction of $\cone$, then
  $D_{\albrep.}f\in\cone$.
\item  Alberti representations $\{\albrep
  i.\}_{i=1}^k$ in the $f$-directions of independent cone fields $\{\cone_i\}_{i=1}^k$ generate
  independent derivations $\{D_{\albrep i.}\}_{i=1}^k$.
\end{itemize}
\par Furthermore, our construction relates to a notion of speed for
Alberti representations introduced in \cite{bate-diff}\footnote{We
  allow $\delta$ to be a Borel function, while in \cite{bate-diff} it
  is a constant}: if $f:X\to\real$ is
  Lipschitz we say that an Alberti representation
$\albrep.=(P,\nu)$ \textbf{has $f$-speed $\ge\delta$} if, for $P$-a.e.~$\gamma$ and
$\lebmeas\mrest\dom\gamma$-a.e.~point $t$,
$(f\circ\gamma)'(t)\ge\delta\metdiff\gamma(t)$, where $\metdiff\gamma$%
\nomenclature[frags]{$\metdiff$}{metric differential}%
\ denotes the the metric differential of $\gamma$ (Definition
\ref{def:met_diff}). Using an averaging process
\eqref{eq:alberti_speed}, we associate an \textbf{effective speed}
$\sigma_{\albrep.}$ to ${\albrep.}$ and show that:
\begin{itemize}
\item If $\albrep.$ has $f$-speed $\ge\delta$, then
  $D_{\albrep.}f\ge\sigma_{\albrep.}\delta$ (Theorem \ref{thm:directional_speed}).
\end{itemize}
\par We then address questions related to the injectivity and surjectivity of
$\Der$. The map $\Der$ is very far from being injective: as an
illustration of this fact, we show that (Lemma \ref{lem:domain_reduction}):
\begin{itemize}
\item Given a nondegenerate compact interval $I\subset\real$, an
  Alberti representation $\albrep.$ can be
  replaced by a new representation $\albrep.'$, whose probability is
  concentrated on fragments with domain inside $I$, and such that
  $D_{\albrep.}=D_{\albrep.'}$.
\item Properties like the Lipschitz/bi-Lipschitz constant, the speed
  and the direction are preserved by replacing $\albrep.$ by $\albrep.'$.
\end{itemize}
\par To study the surjectivity of $\Der$, we use derivations to
produce Alberti representations, the intuition being that independent
derivations can be used to produce Alberti representations in the directions of
independent cone fields. The starting point is the observation that the
independence of derivations $\{D_i\}_{i=1}^k\subset\wder\mu\mrest U.$, up to
taking a Borel partition of $U$ and linear combinations of the $D_i$,
is detected by \textbf{pseudodual} Lipschitz functions
$\{g_i\}_{i=1}^k\subset\lipalg X.$ such that
$D_ig_j=\delta_{i,j}\chi_U$%
\nomenclature[measure]{$\chi_U$}{indicator function of $U$}%
\ (by Corollary \ref{cor:pseudoduality}). 
 To deal with Borel
partitions, one is led to introduce the \textbf{restriction of
  $\albrep.=(P,\nu)$ to a Borel set $U$}: $\albrep.\mrest
U=(P,\nu\mrest U)$ \cite[Lem.~2.4]{bate-diff}.%
\nomenclature[alberti]{$\albrep.\mrest U$}{restriction of $\albrep.$ to $U$}%
\ In Subsection \ref{subsec:derivations_to_alberti}
we show:
\begin{itemize}
\item If the $\{D_i\}_{i=1}^k\subset\wder\mu\mrest U.$ have
  pseudodual Lipschitz functions $\{g_i\}_{i=1}^k\subset\lipalg
  X.$, letting $g=(g_i)_{i=1}^k$, for any constant $k$-dimensional cone field $\cone$, it is possible
  to obtain a $(1,1+\epsi)$-bi-Lipschitz Alberti representation of
  $\mu\mrest U$ in the $g$-direction of $\cone(w,\alpha)$ with \emph{almost
    optimal} \eqref{eq:speed_almost_optimal} $\langle
  w,g\rangle$-speed (Theorem \ref{derivation_alberti}).
\item If the $\{D_i\}_{i=1}^k\subset\wder\mu\mrest U.$ are
  independent, passing to a Borel partition $U=\bigcup_\alpha U_\alpha$,
  there are Lipschitz functions $f_\alpha$ such that, for each
  $k$-dimensional cone field
  $\cone$, there is an Alberti representation $\albrep.$ of $\mu$ with
  $\albrep.\mrest U_\alpha$ in the $f_\alpha$-direction of $\cone$ (Corollary
  \ref{der-alb}).
\item If $f:X\to\real^k$ and $\mu$ admits an Alberti representation in
  the $f$-direction of $k$ independent cone fields, then for each
  $k$-dimensional cone field
  $\cone$, the measure $\mu$ admits an Alberti representation in the $f$-direction
  of $\cone$ (Corollary \ref{cor:mu_arb_cone}).
\end{itemize}
Corollary \ref{cor:mu_arb_cone} is saying that there cannot be
\emph{gaps} in the directions accessible by curve
fragments. Surprisingly, in the proof we manage to avoid ``harder''
arguments (e.g.~involving porosity techniques) and  rely on some ``soft'' functional
analysis.
\par In Subsection \ref{subsec:derivations_to_alberti} we finally show that:
\begin{itemize}
\item If $\wder\mu.$ is finitely generated, then $\Der$ is surjective
  (Theorem \ref{thm:fin_gen_surjectivity}).
\item In general, $\Der(\Alb_{{\rm sub}}(\mu))$ is weak* dense
  in $\wder\mu.$ (Theorem \ref{thm:weak*density}).
\end{itemize}
\par 
The \textbf{dual module} $\wform\mu.$%
\nomenclature[derivations]{$\wform\mu.$}{module of forms}%
\ of $\wder\mu.$ can be regarded as
the $L^\infty(\mu)$-module of differential forms because each
$f\in\lipalg X.$ gives rise to a form $df\in\wform\mu.$. The modules
$\wder\mu.$ and $\wform\mu.$ admit \textbf{local norms}
$\locnorm\,\cdot\,,{\wder\mu.}.$ and
$\locnorm\,\cdot\,,{\wform\mu.}.$,%
\nomenclature[derivations]{$\locnorm\,\cdot\,,{\wder\mu.}.$}{local norm on
  $\wder\mu.$}%
\nomenclature[derivations2]{$\locnorm\,\cdot\,,{\wform\mu.}.$}{local norm on $\wform\mu.$}%
\ which can be thought of as families
$\{\|\cdot\|_x\}_{x\in X}$ of pointwise norms that one can use
reconstruct the global norms by taking the essential supremum.  In
Subsection \ref{subsec:weaver_norm} we obtain a geometric
characterization of $\locnorm\,\cdot\,,{\wform\mu.}.$, which plays a
central r\^ole in our characterization of differentiability spaces:
\begin{itemize}
\item For $U$ Borel, $f\in\lipalg X.$ and $\alpha>0$, if $\locnorm
  df,{\wform\mu\mrest U.}.\approx\alpha$, then there is an Alberti
  representation $\albrep.=(P,\nu)$ of $\mu\mrest U$ with $P$
  concentrated on the fragments where $(f\circ\gamma)'\approx
  \alpha\metdiff\gamma$ (Theorem
  \ref{thm:weaver_fnorm_char}).
\end{itemize}
\par The correspondence that we have illustrated between derivations and Alberti
representations applies to any Radon measure defined on a separable
metric space. Here are some relevant examples:
\begin{itemize}
\item Spaces $(X_{\rm lack},\mu_{\rm lack})$ which either lack any rectifiable curve
or where $\mu_{\rm lack}$ does not admit any Alberti representation:
in this case $\wder\mu_{\rm lack}.=\{0\}$.
\item Spaces which are $k$-rectifiable $(X_{\text{$k$-rect}},\hmeas k.)$.
\item Products $(X_{\rm lack}\times X_{\text{$k$-rect}}, \mu_{\rm
    lack}\times \hmeas k.)$ and quotients of such products, for
  example Laakso spaces and the non-doubling Laakso-like spaces of
  \cite{weaver00} and \cite{bate_speight}.
\item Differentiability spaces. In this case we obtain a quantitative
  characterization, see Theorem~\ref{thm:diff_char2}.
\item Carnot groups equipped with a Radon measure $\mu$: in this case $\wder\mu.$ is always finitely
  generated and the number of generators is at most the dimension of
  the horizontal distribution.
\item Spaces which have rectifiable fragments in infinitely many
  directions, for example the Hilbert cubes in $l^p$ and $c_0$
  considered in \cite{weaver00}.
\item The supports of metric currents (see \cite{curr_alb}).
\end{itemize}
\subsection*{Structure of differentiability spaces}
We now describe an application of the correspondence between
derivation and Alberti representation to the theory of
differentiability spaces. A differentiability space is a metric
measure space where a generalized version of Rademacher's Theorem on
differentiability of Lipschitz functions holds; this generalization
relies on the idea that the space of Lipschitz functions looks, in a
suitable sense, finite-dimensional at small scales, see Subsection
\ref{subsec:diff_spaces} for more details; the least upper bound on
the dimension is called the \textbf{differentiability dimension}.
\par We now recall notions of \textbf{finite dimensionality} for
measures on metric spaces:
\begin{itemize}
\item A measure $\mu$ on $X$ \textbf{is doubling (with constant $C$)} if,
for all pairs $(x,r)\in X\times (0,\diam X]$ ($\diam X=\infty$ is
allowed, but then we require $r\in(0,\infty)$),
\begin{equation}\label{eq:meas_doubling} \mu\left(\ball
  x,r/2.\right)\ge C\mu\left(\ball x,r.\right).
\end{equation}
Note that it does not matter if balls are open or closed; however, in this
paper the notation $\ball x,r.$ denotes an open ball.
\item If \eqref{eq:meas_doubling} holds for $\mu$-a.e.~$x$ for
$r\in(0,R_x]$, where the real number $R_x>0$ is allowed to depend on $x$, the measure $\mu$ is called \textbf{asymptotically
  doubling (with constant $C$)}.
\item If there are disjoint Borel sets $X_\alpha$ with $\mu(X\setminus
  \bigcup_\alpha X_\alpha)=0$ and the measure $\mu\mrest X_\alpha$ is doubling on $X_\alpha$,
 $\mu$ is called \textbf{$\sigma$-asymptotically doubling}.
\end{itemize}
\par We now recall the definitions of \textbf{infinitesimal Lipschitz
  constants} and differentiability for Lipschitz functions.
For a real-valued Lipschitz function $f$, the variation of $f$ at $x$ at scale $r$ is
$\sup_{y\in\ball x,r.}|f(x)-f(y)|/{r}$; the lower and upper
limits of the variation as $r\searrow0$ are denoted by $\smllip f(x)$ and $\biglip
f(x)$.%\footnote{Sometimes called the ``small Lip'' lip and the ``big
 % Lip'' Lip in the literature. Other terms used are pointwise lower
  %and upper Lipschitz constants.}.%
\nomenclature[loclip]{$\smllip f(x)$}{pointwise lower Lipschitz constant}%
\nomenclature[loclip]{$\biglip f(x)$}{pointwise upper Lipschitz constant}%
\par We will now recall two sufficient conditions for a metric
measure space to be a differentiability space. A $(C,\tau, p)$-PI
space is a doubling metric measure space with constant $C$ supporting
a $p$-Poincar\'e inequality (\cite{heinonen_analysis}) with constant
$\tau$. Note that knowledge of the Poincar\'e inequality is
\textbf{not} required for understanding this paper. For the reader
already familiar with the Poincar\'e inequality we point out that
sometimes the ball on which the
upper gradient is integrated is allowed to be enlarged by a constant
factor $\Lambda$ (compare~\cite[(4.3)]{cheeger99}): in this case the constant
$\tau$ takes into account also the effect of $\Lambda$, i.e.~with the
notation of \cite[(4.3)]{cheeger99}, we would take $\tau=\max(C,\Lambda)$.
In \cite[Thm.~4.38]{cheeger99} Cheeger showed that:
  \begin{thm}\label{thm:cheeger}
    If $(X,\mu)$ is a $(C,\tau,p)$-PI-space, then it is a differentiability space
    with differentiability dimension $\le N(C,\tau)$. Moreover, for each
    Lipschitz function $f$, the equality $\smllip f = \biglip f$ holds $\mu$-a.e.
  \end{thm}
In \cite[Thm.~2.3.1]{keith04} Keith was able to relax the assumptions of Cheeger:
\begin{thm}\label{thm:keith}
  If $(X,\mu)$ is a metric measure space with a $C$-doubling measure
  $\mu$ and there is a constant
  $\tau>0$ such that, for each real-valued Lipschitz function $f$,
  \begin{equation}\label{eq_Lip_lip_ineq}
    \tau\smllip f(x)\ge\biglip f(x)\quad\text{\normalfont(for $\mu$-a.e.~$x$),}
  \end{equation}
then $(X,\mu)$ is a differentiability space
    with differentiability dimension $\le N(C,\tau)$.
\end{thm}
The inequality \eqref{eq_Lip_lip_ineq} is sometimes called the
\textbf{Lip-lip inequality}. In the literature the quantity $\smllip f$ is sometimes
denoted by ${\rm lip}f$ and the quantity $\biglip f$ is sometimes
denoted by ${\rm Lip}f$.
% \par Note that Theorems \ref{thm:cheeger} and \ref{thm:keith}
% can be proved under the relaxed assumption that
% $\mu$ is asymptotically doubling. Moreover, in
\par The structure of differentiability spaces is currently not very
well-understood. In particular, there is no simple geometric
characterization. However, there are recent results providing necessary
conditions for the existence of a differentiable structure. These
results, however, do not require a uniform bound on the
differentiability dimension and we thus introduce the term
\textbf{$\sigma$-differentiability space} to denote a complete separable metric
measure space $(X,\mu)$  which is a countable union of
differentiability spaces  $(X_\alpha,
\mu\mrest X_\alpha)$ where each $X_\alpha$ is Borel in $X$, and such
that for any Lipschitz function $f:X\to\real$ and for $\mu$-a.e.~$x\in X_\alpha$
the infinitesimal Lipschitz constant at $x$, $\biglip f(x)$, computed in $X$ and that computed in
$X_\alpha$, $\biglip_{X_\alpha} f(x)$, agree. This last additional technical
requirement about the equality of $\biglip f(x)$ and
$\biglip_{X_\alpha} f(x)$ might seem unnatural, but is actually needed
to have a well-defined notion of derivative at $x$. For the moment,
the reader might think intuitively that $\biglip_{X_\alpha} f(x)$ measures the size
of the gradient of $f$ in $X_\alpha$ and $\biglip f(x)$ measures the
size of the gradient of $f$ in $X$. For details we refer the reader to
Section~\ref{sec:app_diff_spaces}, in particular to Remark~\ref{porosity_remark}.
In \cite{bate_speight} Bate and Speight showed:
\begin{thm}\label{thm:bate_speight}
  If $(X,\mu)$ is a differentiability space, then $\mu$ is
  $\sigma$-asymptotically doubling.
\end{thm}
\begin{exa}
  Note that one cannot conclude that $\mu$ is asymptotically doubling,
  i.e.~that the local doubling constant is uniformly bounded.  For
  example, consider a compact metric measure space with $\mu$ finite
  which is obtained by gluing together, along some geodesics,
  countably many (rescalings of) Laakso spaces with Hausdorff
  dimensions tending to $\infty$; this construction produces a
  differentiability space but the constant $C$ in
  \eqref{eq:meas_doubling} is not uniformly bounded.
\end{exa}
Later
\cite{bate-diff} Bate provided the following characterization:
\begin{thm}\label{thm:bate-diff_char}
    A metric measure space $(X,\mu)$ is a
  $\sigma$-differentiability space if and only if:
\begin{enumerate} 
\item The measure $\mu$ is $\sigma$-aymptotically doubling.
\item There is a Borel map $\tau:X\to(0,\infty)$ such that, for each
  real-valued Lipschitz function $f$, the measure $\mu$ admits an
  Alberti representation with \hbox{$f$-speed} $\ge\biglip f/\tau$.
\end{enumerate}
\end{thm}
\par Even though the existence of Alberti representations provides new
information about the structure of differentiability spaces, it seems that using (2) in
Theorem~\ref{thm:bate-diff_char} to verify that a metric measure space is a
differentiability space is at least as impractical as verifying a
Lip-lip inequality.
Note, however, that Theorem \ref{thm:bate-diff_char} implies that a weaker form
of the Lip-lip inequality, where $\tau$ is allowed to be a Borel
function, must hold in a $\sigma$-differentiability
space. This leads to the natural question of whether there are
differentiability spaces where the Lip-lip inequality holds for some
constant $\tau>1$, but not for $\tau=1$. In particular, in \cite{keith04} Keith does not
provide any examples of differentiability spaces that cannot be
realized as positive measure subsets of a countable union of
PI-spaces. In the light of Theorem \ref{thm:cheeger}, such spaces
would easily arise if the Lip-lip inequality could hold as a strict
inequality. We make negative progress in this direction by showing
that the Lip-lip condition is more rigid than it
seems. In particular we show that in a 
$\sigma$-differentiability space, the equality $\biglip f=\smllip f$ holds $\mu$-a.e. Our
proof relies on the geometric characterization of
$\locnorm\,\cdot\,,{\wform\mu.}.$; another proof of $\biglip f=\smllip
f$ has appeared in a more recent work on metric differentiation
\cite{cks_metric_diff}. We provide the following characterization of
differentiability spaces:
\begin{thm}\label{thm:diff_char2}
%\textcolor{red}
{The metric measure space  $(X,\mu)$ is a $\sigma$-differentiability
  space if and only if one of the following
  equivalent conditions holds:}
  \begin{enumerate}
  \item For each $f\in\lipfun X.$
    \begin{equation}\label{eq:diff=Lip_gl}
      \locnorm df,{\wform\mu.}.(x)=\biglip f(x)\quad\text{\normalfont(for $\mu$-a.e.~$x$)}.
    \end{equation}
  \item For each $f\in\lipfun X.$, denoting by $S_f$ the Borel set
    \begin{equation}
      \left\{x\in X:\biglip f(x)>0\right\},
    \end{equation}
    for all $\epsi,\sigma\in(0,1)$ the measure $\mu\mrest S_f$ admits a
    $(1,1+\epsi)$-bi-Lipschitz Alberti representation with $f$-speed
    $\ge\sigma\biglip f$.
  \item For each $f\in\lipfun X.$,
    \begin{equation}
      \label{eq:diff=smllip=biglip}
      \biglip f(x)=\smllip f(x)\quad\text{\normalfont(for $\mu$-a.e.~$x$)}.
    \end{equation}
  \end{enumerate}
\end{thm}
It is worth noting that the condition that $(X,\mu)$ is
$\sigma$-asymptotically doubling is already contained
in~(\ref{eq:diff=Lip_gl}) or~(\ref{eq:diff=smllip=biglip}): this is
explained in the proof of Theorem~\ref{thm:diff_char2}.
\par There is also a connection between differentiability and
derivations. We show that:
\begin{thm}\label{thm:schioppa}
  A metric measure space $(X,\mu)$ is a differentiability space with
  dimension $\le N$ if and only if  $\mu$ is
    $\sigma$-asymptotically doubling and
  there are a conformal factor $\lambda\in L^\infty(\mu)$ and derivations $\{D_i\}_{i=1}^N\subset{\rm
    Der}(\mu)$ such that,
for each $f\in\lipalg X.$,
  \begin{equation}\label{revder}
    \lambda(x)\max_i|D_if(x)|\ge\biglip f(x)\quad\text{for $\mu$-a.e.~$x$.}
  \end{equation}
\end{thm}
In \cite{derivdiff}, motivated by \cite{gong11-revised}, the author
showed that \eqref{revder} gives sufficient conditions for the
existence of a differentiable structure\footnote{In \cite{derivdiff}
  the measure $\mu$ was assumed doubling. However, taking a Borel
  partition, it suffices to assume that $\mu$ is
  $\sigma$-asymptotically doubling.}; in \cite{derivdiff} the author
also proved a partial converse: if $(X,\mu)$ is a differentiability
space with $\mu$ doubling and if the partial derivative
operators are derivations, then \eqref{revder} holds. Thus
Theorem~\ref{thm:schioppa} 
follows from \cite{derivdiff} because we provide
two different proofs that the partial derivative operators are
derivations: the first proof uses \cite{bate-diff} and can be found in
Subsection \ref{subsec:derivation_differentiability}; the second
proof uses Theorem \ref{thm:diff_char2} and can be found at the end of
Subsection  \ref{subsec:char_diff}. The possibility of providing
independent and different proofs is closely related to the topic of
metric differentiation discussed in \cite{cks_metric_diff}.
\par An important consequence of Theorem \ref{thm:schioppa} is that in
Theorems \ref{thm:cheeger} and \ref{thm:keith} one can replace the
bound $N(C,\tau)$ by the Assouad dimension, removing the dependence on
$\tau$ (Corollary \ref{cor:assouad_bound}).
\subsection*{Technical tools}\label{subsec:technical-tools}
In this subsection we give an overview of four technical tools used
in this paper.
\par The first tool is an approximation scheme for Lipschitz functions
in the weak* topology. % One can extend the notion of weak* sequential
% convergence from $\lipalg X.$ to $\lipfun X.$ (note: we are not
% defining a Banach space structure or a weak* topology on $\lipfun X.$)
% by saying that 
% $f_n\xrightarrow{\text{w*}} f$ if $f_n\to f$ pointwise there is an
% uniform bound on the Lipschitz constants of the $f_n$.
The intuition is that if a set $S$ is
$\frags(X,f,\delta)$-null (Definitions \ref{def:frag_nullity}
and \ref{def:classes_frags}),
 i.e.~does not contain fragments where
$(f\circ\gamma)'(t)\ge\delta\metdiff\gamma(t)$, then $f\in\lipalg X.$ can be
approximated by Lipschitz functions which
have Lipschitz constant at most $\delta$ in sufficiently small balls
centred on $S'$, where $S'\subset S$ has full measure in $S$ . We prove an
approximation scheme, Theorem \ref{onedimapprox_multi}, which takes
into account also the direction of the fragments. We state here a
simplified version which is sufficient for the results on
differentiability spaces.
\begin{thm}\label{onedimapprox}
  Let $X$ be a compact metric space, $f:X\to\real$ $L$-Lipschitz and
  $S\subset X$ compact. Let $\mu$ be a Radon measure on $X$. Assume that
  $S$ is \hbox{$\frags(X,f,\delta)$-null}. Then there are
  $\max(L,\delta)$-Lipschitz functions $g_k\xrightarrow{\text{w*}} f$
  with $g_k$ $\mu$-a.e.~locally
  $\delta$-Lipschitz on $S$.
\end{thm}
\par
The motivation to prove Theorem \ref{onedimapprox_multi} came
from reading the first version of \cite[Subsec.~6.1]{bate-diff}: the author observed that
Bate's construction can be used to produce an approximation scheme for
Lipschitz functions with $\biglip f=0$ (\emph{flat}) on $S$. In the
author's opinion, \cite[Subsec.~6.1]{bate-diff} is an adaptation to metric spaces
of a construction sketched in
\cite[Defn.~1.14]{acp_proceedings}. However, the original approximation scheme
based on \cite[Subsec.~6.1]{bate-diff} could be used only to prove
part of the results presented here: the major obstacle is that the
approximation works only if one is allowed to take the limit for
$\delta\to0$.  The problem stems from the presence of a potential term
$\delta\hmeas 1.$ in the $Q$-Lagrangian introduced by Bate. In a
subsequent version of \cite{bate-diff}, there is a refined Lagrangian 
that can also be used to prove Theorem \ref{onedimapprox}. In Subsection
\ref{subsec:an-appr-scheme} we sketch the original approach that we
used to prove 
Theorem \ref{onedimapprox}. This approach is based 
on a combinatorial result stated (without proof)
in \cite[Thm.~2.4]{acp_proceedings}; the
author found a proof in A.~Marchese's
PhD~thesis~\cite{alberti_marchese}.
\par It should be noted that the passage from Euclidean spaces to
general metric spaces does not require essentially new ideas: it is
based on an \emph{abstract nonsense construction} of a cylinder which is a
geodesic metric space containing the graph of the function to be
approximated. The starting
point, as in \cite[Subsec.~6.1]{bate-diff}, is a Kuratowski embedding
in $l^\infty$.
\par The second technical tool is a decomposition of $\wder\mu.$ into
free modules. The problem stems from the fact that $L^\infty(\mu)$ is
not an integral domain and thus the notion of linear independence of
derivations behaves quite differently than in a vector space. In
\cite{derivdiff} the author introduced the following concept of finite
dimensionality:
\begin{defn}
  The module $\wder\mu\mrest U.$ is said to have \textbf{index $\le
    N$} if any linearly independent (over $L^\infty(\mu\mrest U)$) set
  of derivations in it has cardinality at most $N$.  If, for any
  $\mu$-measurable $U$, the module $\wder\mu\mrest U.$ has index $\le N$, we say
  that $\wder\mu.$ has \textbf{index locally bounded by $N$}.
\end{defn}
\par Under this assumption the author \cite{derivdiff} obtained the
following decomposition result.
\begin{thm}\label{thm:free_dec}
  Suppose that $\wder\mu.$ has index locally bounded by $N$. Then
  there is a Borel partition $X=\bigcup_{i=0}^N X_i$ such that, if
  $\mu(X_i)>0$, then $\wder\mu\mrest X_i.$ is free of rank $i$. A
  basis of $\wder\mu\mrest X_i.$ will be called a \textbf{local basis of
  derivations}.
\end{thm}
\par The third technical tool is to relate Alberti representations and
\textbf{blow-ups} of metric spaces. A \textbf{blow-up\footnote{Sometimes
    called a tangent space / cone} of a metric space $X$ at a point
  $p$} is a (complete) pointed metric space $(Y,q)$ which is a pointed
Gromov-Hausdorff limit of a sequence $(\frac{1}{t_n}X,p)$ where $t_n\searrow0$:
the notation $\frac{1}{t_n}X$ means that the metric on $X$ is rescaled by
$1/t_n$; the class of blow-ups of $X$ at $p$ is denoted by
$\tang(X,p)$. A \textbf{blow-up of a Lipschitz function $f\colon X\to\real^Q$
  at a point $p$} is is a triple $(Y,q,g)$ with
$(\frac{1}{t_n}X,p)\to(Y,q)\in\tang(X,p)$, $g\colon X\to\real^Q$ Lipschitz and such
that the rescalings $(f-f(p))/t_n:\frac{1}{t_n}X\to\real^Q$ converge to $g$;
the class of blow-ups of $f$ at $p$ is denoted by $\tang(X,p,f)$.
The general existence of blow-ups requires the notion of
\textbf{ultralimits}: we simplified the treatment assuming that $X$ has
\textbf{finite Assouad dimension} (\cite{tyson_mackay_conf}) because
this assumption is not restrictive in the theory of differentiability
spaces.
\par In Subsection \ref{subsec:dimens-bounds-tang} we show that:
\begin{itemize}
\item If $f:X\to\real^N$ and $\mu$ admits Alberti representations in the $f$-directions of
  independent cone fields, for $\mu$-a.e.~$p$ all blow-ups
  $(Y,q,g)\in\tang(X,p,f)$ are such that $g:Y\to\real^N$ is surjective (Theorem \ref{alberti_blow_up}).
\item If $X$ has Assouad dimension $D$, then $\wder\mu.$ has index
  locally bounded by $D$ (Corollary \ref{derbound}).
\end{itemize}
\par Note that Corollary \ref{derbound} improves
\cite[Lem.~1.10]{gong11-revised} by giving an explicit bound \emph{equal} to
the Assouad dimension.  
\par The fourth tool is a construction of independent Lipschitz
functions, Theorem~\ref{lip_ind}. Lipschitz functions
$\{\psi_1,\ldots,\psi_n\}$ are \textbf{independent} on a set $S$ if,
for each $x\in S$, the map $\real^n\ni
(a_i)\mapsto\biglip(\sum_{i=1}^na_i\psi_i)(x)$ is a norm. The property
of being a differentiability space can be reformulated as a
\textbf{finite dimensionality statement}: there is a uniform upper bound on the
number of Lipschitz functions which are independent on a positive
measure set.
\par In the case of Euclidean spaces, instead of constructing
independent functions on a set $S$, it is more natural to construct
Lipschitz functions which are not differentiable on $S$.  In the case
of $\real$, there is a classical construction of Zahorski
(\cite{zahorski_line}); for $\real^n$, a generalized construction is
announced in \cite[Thm.~1.15]{acp_proceedings}; in the case in which
one relaxes the conclusion to nondifferentiability $\mu$-a.e.~on
$S$ (i.e.~``nondifferentiability for measures''),
the construction is simplified as it can %\textcolor{red}
{be handled}
independently on different parts of $S$ using a truncation principle
(Lemma~\ref{lip_trunc}). Recently, Alberti and Marchese
\cite{alberti_marchese} strengthened this approach showing that the
set of Lipschitz functions which are nondifferentiable on a large
part of $S$ is \textbf{comeagre} in a suitable metric space of Lipschitz
functions.
\par In \cite[Sec.~4]{bate-diff} Bate produces a construction of
nondifferentiable Lipschitz functions for measures on metric spaces; the
approach is similar to the Euclidean case but requires the tool of
\textbf{structured charts} introduced in \cite{bate_speight}. Theorem~\ref{lip_ind} is a reformulation of this result in the language of
independent functions; the author thinks this is useful because: [1]
it is technically simpler avoiding a discussion of structured charts;
[2] it fits more naturally with the approaches of Cheeger and Keith.
In the first version of this paper the author raised the question of
whether it is possible to do a construction of
independent functions \emph{for sets} in metric spaces. This question
is natural as in~\cite{acp_proceedings} the authors announce the
construction, for each set $S$ in a suitable class of null sets,
of a Lipschitz map nowhere differentiable on $S$. However, the answer in the metric setting turns out to be
negative as Cheeger's Differentiation Theorem~\cite{cheeger99} does \emph{not} single out a
canonical measure class, see~\cite{sing_poinc}.
\par Finally, when the first version of this paper appeared in
November 2013 all known examples of differentiability spaces were
covered by the theory in~\cite{cheeger99}, as all examples were countable unions
of positive measure subsets of spaces admitting Poincar\'e
inequalities. Therefore, some colleagues have criticized
Theorem~\ref{thm:diff_char2} for a lack of applicability. In a
forthcoming paper we respond to this criticism by constructing a new
class of differentiability spaces which are not covered by the theory
in~\cite{cheeger99} and for which, in some sense, the Lip-lip
equality~(\ref{eq:diff=smllip=biglip}) is an optimal characterization.
\subsection*{Acknowledgements} The author wants to thank
% \textcolor{red}
{his advisor,
Bruce Kleiner, }for reading this work and providing many stimulating
questions. In particular, some of these questions motivated the author
to prove Theorem \ref{thm:weak*density}.
\par The author  wishes to thank David Bate for comments on the first
version of this preprint and for pointing out the refined Lagrangian
in the final version of \cite{bate-diff}.
\par Finally, the author is grateful to the anonymous referee for
reading the paper very carefully and raising many good questions to
clarify the exposition.
\par During the revision phase of this paper the author was supported
by the ``ETH Zurich Postdoctoral Fellowship Program and the Marie
Curie Actions for People COFUND Program''.
\renewcommand{\nomname}{List of Symbols}
\printnomenclature[2cm]
\section{Preliminaries}\label{sec:prelim}
\subsection{Alberti representations}\label{subsec:alberti}
The goal of this subsection is to define Alberti representations
precisely and prove Theorem \ref{alberti_rep_prod}, which can be
regarded as a standard criterion to produce Alberti
representations. The treatment has been a expanded a bit to address what seem
to be a couple of little gaps in \cite{bate-diff}:
\begin{enumerate}
\item In \cite[Lem.~5.2]{bate-diff} Alberti representations are
  produced with $P$ a probability measure on $1$-rectifiable measures,
  instead of fragments (point addressed in Lemma \ref{biLip_dis}).
\item In \cite[Sec.~6]{bate-diff} Alberti representations are produced
  with $P$ defined on $\frags(\bana)$ where $\bana$ is a Banach space
  containing $X$ (point addressed in Theorem
  \ref{compact_reduction}).
\end{enumerate}
 Throughout this section,
unless otherwise specified, $X$
will denote a \textbf{%\textcolor{red}
{complete separable and locally
    compact metric space.}} %\textcolor{red}
{Note that we make the assumption on local
compactness to have a convenient definition of topology for Radon
measures on $X$. This assumption does not entail a significant loss of
generality; in fact,  if $X$ is a separable metric space
and $\mu$ is a Radon measure on $X$, as
$\mu$ is locally finite and inner regular, one might reduce the study
of Alberti representations of $\mu$ to studying those of $\mu\mrest T$,
where $T$ is a subset of $X$ which has full $\mu$-measure and is
$\sigma$-compact, i.e.~a countable union of compact sets.}
We start by defining (path) fragments.
\begin{defn}
Given a metric space $Y$,
let $\haus(Y)$ denote the topological space of
%\textcolor{red}
{nonempty compact subsets of $Y$} with the
Vietoris topology, which is induced by the Hausdorff distance. The
following properties are easy to see: if $Y$
is compact / separable / complete / locally compact, then $\haus(Y)$
is compact / separable / complete / locally compact. We introduce the set of (path) \textbf{fragments}:
  \begin{equation}
  \frags(X)=\left\{\gamma:K\to X:\text{$\gamma$ bi-Lipschitz,
    $K\subset\real$ compact, $\lebmeas(K)>0$}\right\};
\end{equation}
which is identified with a subspace of $\haus(\real\times X)$ via the
map $\gamma\mapsto\graph\gamma$.  Given a nondegenerate compact
interval $I\subset\real$, we denote by $\frags(I; X)$ the subset of
fragments $\gamma$ with $\dom\gamma\subset I$.
\end{defn}
\par In order to define Alberti representations precisely, we need to
recall some facts from measure theory. 
\begin{assump}\label{ass_mea}[Assumption on the topology on Radon measures]
Let $Z$ denote a locally
compact metric space, $M(Z)$%
\nomenclature[measure]{$M(Z)$}{finite Radon measures on $Z$}%
\nomenclature[measure]{$P(Z)$}{finite Borel probability measures on $Z$}%
\ the Banach space of finite (signed) Borel
measures on $Z$ and $P(Z)\subset M(Z)$ the subset of probability
measures.  It might be useful to recall that finite Borel measures on
metric spaces are regular and that a finite Borel measure on a
Polish space is Radon
\cite[Thm.~7.1.7]{bogachev_measure}. The Banach space $M(Z)$ is also a
dual Banach space; in the definition of $M(Z)$-valued Borel maps \textbf{we
will consider on $M(Z)$ the weak* topology}. In particular, given a
metric space $Y$, a map $\psi:Y\to M(Z)$ is Borel if and only
if for each $g\in C_c(Z)$ (set of real-valued continuous functions
  with compact support) the map $y\mapsto\int_Z g\,d\psi(y)$ is Borel (compare
\cite[Rem.~1.1]{alberti_rank_one}).
\end{assump}
\par We will use the following result
(\cite[Defn.~1.2]{alberti_rank_one}):
\begin{lem}\label{integration_meas}
  Let $Y$ be a separable locally compact topological space, and
  $\lambda$ a locally finite Borel measure on $Y$. Let $\psi\colon Y\to
  M(Z)$ be Borel and, denoting by $\|\psi(y)\|_{M(Z)}$ the norm of
  $\psi(y)$, assume that
  \begin{equation}
    \int_Y\|\psi(y)\|_{M(Z)}\,d\lambda(y)<\infty;
  \end{equation}
  then the integral
  \begin{equation}
    \mu=\int_Y\psi(y)\,d\lambda(y)
  \end{equation} exists and defines an element of $M(Z)$. More
  precisely, for a Borel set $A\subset Z$,
  \begin{equation}
    \mu(A)=\int_Y\psi(y)(A)\,d\lambda(y).
  \end{equation}
\end{lem}
We can now state precisely the definition of an Alberti
representation; note that condition (4) has been added for technical
convenience (compare the proof of Lemma \ref{lem:meas_fix}).
\begin{defn}\label{defn:Alberti_rep}
  Let $\mu$ be a Radon measure on the metric space $X$; an \textbf{Alberti
  representation} of $\mu$ is a pair $(P,\nu)$:
  \begin{enumerate}
  \item The measure $P$ is in $P(\frags(X))$.
  \item The map $\nu:\frags(X)\to M(X)$ is Borel and
  $\nu_\gamma\ll\hmeas 1._\gamma$, the one-dimensional Hausdorff
  measure associated to the image of $\gamma$.
\item  The measure $\mu$ can be represented as
  $\mu=\int_{\frags(X)}\nu_\gamma\,dP(\gamma)$.
\item For each Borel set $A\subset X$ and for all real numbers $a\le b$,
  the map
  $\gamma\mapsto\nu_\gamma\left(A\cap\gamma(\dom\gamma\cap[a,b])\right)$
  is Borel.
  \end{enumerate}
\end{defn}
\par We now define precisely the notions of Euclidean cones and
metric differential employed in the introduction
to present the notions of direction and speed for Alberti representations.
\begin{defn}\label{defn:cone}
  Let $\alpha\in(0,\pi/2)$, $w\in{\mathbb S}^{q-1}$; the \textbf{open
    cone} $\cone(w,\alpha)\subset\real^q$ with axis $w$ and opening
  angle $\alpha$ is:
\begin{equation}\label{eq:defcone}
    \cone(w,\alpha)=\{u\in\real^q: \tan\alpha\langle
    w,u\rangle>\|\pi_w^\perp u\|_2\},
  \end{equation}%
\nomenclature[frags]{$\cone(w,\alpha)$}{cone / cone field of direction $w$
  and angle $\alpha$}%
%\textcolor{red}
{where $\pi_w^\perp$ denotes the orthogonal projection on
  {%\textcolor{red}
the subspace}
of $\real^q$ orthogonal to $w$.}
The set of open cones is given the topology of ${\mathbb
  S}^{q-1}\times(0,\pi/2)$.
\end{defn}
\begin{rem}\label{rem:cone_diam}Note that if $u\in\cone(w,\alpha)\cap{\mathbb S}^{q-1}$,
\begin{equation}
  \|u-w\|_2\le(1-\cos\alpha)+\sin\alpha.
\end{equation}
\end{rem}
\par We recall the definition of \textbf{metric differential}, which is
an adaptation of \cite[Defn.~4.1.2]{ambrosio_metric} (compare
\cite[Sec.~3]{ambrosio-rectifiability}):
\begin{defn}\label{def:met_diff}
  Given $\gamma\in\frags(X)$, the metric
  differential $\metdiff\gamma(t)$ of $\gamma$ at $t\in\dom\gamma$ is the limit
  \begin{equation}
    \lim_{\dom\gamma\ni t'\to t}\frac{\dist\gamma(t'),\gamma(t).}{|t'-t|}
  \end{equation}
  whenever it exists\footnote{We convene that if $t$ is an isolated
    point of $\dom\gamma$, the limit does not exist}.
\end{defn}
\par For an Alberti representation $\albrep.$ will abbreviate a set of conditions on the
Lipschitz / bi-Lipschitz constant, speed and direction by \albcond. For
example, \albcond\ might amount to requiring that $\albrep.$ is in the
direction of a given cone field $\cone$ and has speed $\ge\delta$.
\par If $X$ is a metric space and $C\subset X$ with $\mu$ a Radon
measure on $C$, there are a priori two different notions of Alberti
representations $(P,\nu)$ depending on whether $P\in P(\frags(X))$ or
$P\in P(\frags(C))$; in the former case we will say that the Alberti
representation is defined on $\frags(X)$. We will prove Theorem
\ref{compact_reduction}, which produces an Alberti representation
defined on $\frags(C)$ given an Alberti representation defined on
$\frags(X)$ and preserves a set \albcond. Recall that, unless
otherwise specified, $X$ is assumed to be complete separable and
locally compact.
\begin{thm}\label{compact_reduction}
  Let $C$ be a compact subset of $X$. Suppose that the Radon measure
  $\mu$, with support contained in $C$, admits an Alberti
  representation $(P,\nu)$ defined on $\frags(X)$ and satisfying
  \albcond. Then $\mu$ admits an Alberti representation $(P',\nu')$
  defined on $\frags(C)$ and satisfying \albcond.
\end{thm}
\par The proof of Theorem \ref{compact_reduction} requires some
preparation. We start introducing some subsets of fragments.
\begin{defn}
Let $C\subset X$ and $n\in\natural$. We define:
\begin{align}
  \frags(X,C)&=\left\{\gamma\in\frags(X): \text{$\gamma^{-1}(C)$ has
    positive Lebesgue measure}\right\}.\\
\begin{split}
    \frags_n(X)&=\bigl\{\gamma\in\frags(X):\text{$\dom\gamma\subset[-n,n]$,
      $\gamma$ is
      $\left(\frac{1}{n},n\right)$-bi-Lipschitz,}\\ &\phantom{=\bigl\{\gamma\in\frags(X):\text{$\dom\gamma\subset[-n,n]$}}
    \text{and $\lebmeas(\dom\gamma)\ge\frac{1}{n}$}\bigr\};
\end{split}\\
\begin{split}
    \frags_n(X,C)&=\bigl\{\gamma\in\frags_n(X): \text{
      $\lebmeas(\gamma^{-1}(C))\ge\frac{1}{n}$}\bigr\}.
\end{split}
\end{align}
\end{defn}
\begin{lem}\label{lem:frag_compact} Let $X$ be a metric space and $C$
  a closed subset of $X$.
  The sets $\frags_n(X)$ and $\frags_n(X,C)$ are closed in
  $\haus(\real\times X)$. If $C\subset X$ is compact, then $\frags_n(C)$  is
 compact.
\end{lem}
\begin{proof}
  The subset of $\haus(\real\times X)$ consisting of all the graphs of
  $\left(\frac{1}{n},n\right)$-bi-Lipschitz maps
  $\gamma:K\subset\real\to X$ is closed. Also the set of those compact
  sets $K\subset\real\times X$, whose projection $\pi_\real(K)$ on
  $\real$ satisfies $\pi_\real(K)\subset[-n,n]$, is closed. Consider a
  sequence of compact sets $\{K_k\}\subset\haus(\real)$ and assume
  that for each $k$ one has $\lebmeas(K_k)\ge\frac{1}{n}$; then, if the
  $K_k$ converge to $K$, we have $\lebmeas(K)\ge\frac{1}{n}$. We thus
  conclude that the set $\frags_n(X)$ is closed in $\haus(\real\times
  X)$. The compactness of $\frags_n(C)$ follows from
  Ascoli-Arzel\`a. We now show that $\frags_n(X,C)$ is a closed subset
  of $\frags_n(X)$; suppose that the sequence of fragments
  $\{\gamma_k\}\subset\frags_n(X,C)$ converges to $\gamma$; by passing
  to a subsequence we can also assume that the compact sets
  $\gamma_k^{-1}(C)$ converge to a compact set
  $K\subset\dom\gamma$. We then have $\lebmeas(K)\ge\frac{1}{n}$ and
  $\gamma(K)\subset C$, showing that $\frags_n(X,C)$ is closed.
\end{proof}
An immediate consequence of Lemma \ref{lem:frag_compact} is:
\begin{lem}\label{frag_meas} Let $X$ be a metric space and $C$ a
  compact subset of $X$.
  Then the sets $\frags(X)$ and $\frags(X,C)$ are of class $F_\sigma$ in
  $\haus(\real\times X)$.  If $C\subset X$ is compact, then $\frags(C)$
  is of class $K_\sigma$ in $\haus(\real\times X)$.
\end{lem}
\par The key tool to prove Theorem \ref{compact_reduction} is the
following {%\textcolor{red}
{lemma}}.
\begin{lem}\label{borel_rest}
  If $X$ is a complete separable metric space and if $C\subset X$ is compact,
  the restriction map
  \begin{equation}
    \begin{aligned}
      \frest C.:\frags(X,C)&\to\frags(C)\\
      \gamma&\mapsto\gamma|\gamma^{-1}(C),
    \end{aligned}
  \end{equation}
is Borel.
\end{lem}
\begin{proof}
  Note that $\frest C.$ maps $\frags_n(X,C)$ to $\frags_n(C)$ so by
   Lemma \ref{frag_meas} it suffices to show that the restriction $\frest
  C,n.=\frest C.|\frags_n(X,C)$ is Borel. By
  \cite[Lem.~6.2.5]{bogachev_measure} it suffices to show that
  for each $\psi\in C\left(\frags_n(C)\right)$ (the set of real-valued
    continuous functions), the map $\psi\circ\frest C,n.$ is Borel.
  \par If
  $(t,x)\in\real\times C$, let $d_{(t,x)}(K)$ denote the distance from
  the compact set $K\subset\real\times C$ to the point $(t,x)$.  Note that $d_{(t,x)}$ is
  Lipschitz: %\textcolor{red}
{given another compact set $K'\subset\real\times C$}, for each $\epsi>0$ if $(t_1,x_1)\in K$ is a closest point
  to $(x,t)$, there is a point $(t_2,x_2)\in K'$ with
  \begin{equation}
    \dist {(t_1,x_1)},{(t_2,x_2)}.\le\hdist K,K'.+\epsi,
  \end{equation}
  which implies
  \begin{equation}
    d_{(t,x)}(K')\le d_{(t,x)}(K)+\hdist K,K'..
  \end{equation}
  Note that these functions separate points in $\haus(\real\times
  C)$. As $\frags_n(C)$ is compact by Lemma~\ref{lem:frag_compact}, by the Stone-Weierstrass Theorem
  \cite[Thm.~7.32]{rudin-principles} the unital subalgebra of
  $C\left(\frags_n(C)\right)$ generated by the functions
  $\{d_{(t,x)}\}$ is dense. In particular, any ${\psi\in
  C\left(\frags_n(C)\right)}$ is a uniform limit of polynomials in the
  $\{d_{(t,x)}\}$. Therefore, it suffices to show that
  $d_{(t,x)}\circ\frest C,n.$ is Borel. 
  \par We show that $d_{(t,x)}\circ\frest C,n.$ is lower
  semicontinuous: assume that $\gamma_m\to\gamma$ in $\frags_n(X,C)$ and that
  along a subsequence $m_k$ one has
  \begin{equation}
    d_{(t,x)}(\frest C,n.(\gamma_{m_k}))\le\epsi.
  \end{equation}
  It is then possible to find points $t_{m_k}\in\dom\gamma_{m_k}$ with
  ${x_{m_k}:=\gamma_{m_k}(t_{m_k})\in C}$ and $\dist {(t_{m_k},x_{m_k})},
  {(t,x)}.\le\epsi$. By passing to a subsequence we can assume that
  $t_{m_k}\to\tilde t\in\dom\gamma$, $x_{m_k}\to\tilde x\in C$ and $\tilde
  x=\gamma(\tilde t)$. As $\dist {(\tilde t,\tilde
    x)},{(t,x)}.\le\epsi$,
  \begin{equation}
    d_{(t,x)}(\frest
    C,n.(\gamma))\le\liminf_{m\to\infty}d_{(t,x)}(\frest C,n.(\gamma_m)).
  \end{equation}
\end{proof} 
\par The second step in the proof of Theorem \ref{compact_reduction}
is a simplification of the description of Alberti representations
which relies on the following map $\Psi$.
\begin{lem}
  Let $X$ be a complete separable and locally compact metric space; consider the map
  \begin{equation}
    \begin{aligned}
      \Psi:\frags(X)&\to M(X)\\
      \gamma&\mapsto\mpush\gamma.\left(\lebmeas\mrest\dom\gamma\right)=\Psi_\gamma;
    \end{aligned}
  \end{equation}
then $\Psi$ is Borel and, if $A\subset X$ is Borel and
$[a,b]\subset\real$, the map
\begin{equation}\label{bairexx}
  \gamma\mapsto\Psi_\gamma\left(A\cap\gamma(\dom\gamma\cap[a,b])\right)
\end{equation}
is Borel.
\end{lem}
\begin{proof}
  %\textcolor{red}
{From Assumption~\ref{ass_mea}, in order to prove
    that $\Psi$ is Borel, it suffices to show that if $g\in
  C_c(X)$, the map
    \begin{equation}
    \Psi_{g}:\gamma\mapsto\int_Xg\,d\Psi_\gamma=\int_{\real}g\circ\gamma(t)\chi_{\dom\gamma}(t)\,dt
  \end{equation}
  is Borel. Having established that $\Psi_g$ is Borel, the fact that
  the map defined in~(\ref{bairexx}) is Borel follows by choosing a
  sequence $\{g_k\}\subset C_c(X)$ which converges pointwise to
  $\chi_A$.} 
\par %\textcolor{red}
{We now turn to the proof that $\Psi_g$ is Borel by reducing it to
showing that other maps, easier to analyze, are Borel.
For each fragment $\gamma_0$, as $\dom\gamma_0$ is compact, one can find an open
  neighbourhood $O_{\gamma_0}$ of $\gamma_0$ and $a,b\in\real$ such that,
  for each $\gamma\in O_{\gamma_0}$, one has
  $\dom\gamma\subset[a,b]$. Note that $\Psi_g$, when
  restricted to $O_{\gamma_0}$, agrees with the map
  \begin{equation}
    \Psi_{g,a,b}:\gamma\mapsto\int_a^bg\circ\gamma(t)\chi_{\dom\gamma}(t)\,dt.
  \end{equation}
On the other hand, as $X$ is assumed separable, so is $\frags(X)$,
and thus we can cover $\frags(X)$ with countably many open sets
$\{O_{\gamma}\}$. It thus suffices to show that the maps $\{\Psi_{g,a,b}\}_{g,a,b}$, as
$g$ varies in $C_c(X)$ and $a$ and $b$ vary in $\real$ (with the restriction
$a<b$), are Borel.} Without loss of generality we can assume that $g$ is
 nonnegative, in which case we will show that $\Psi_{g,a,b}$
is upper semicontinuous. Assuming that $\gamma_n\to\gamma$ in
$\frags(X)$, we have:
\begin{equation}\label{eq:uppersem}
  \limsup_{n\to\infty}\left(g\circ\gamma_n\,\chi_{\dom\gamma_n}\right)\le
  g\circ\gamma\, \chi_{\dom\gamma};
\end{equation}
in fact, either a point $t$ belongs to infinitely many of the
$\dom\gamma_n$, in which case $t\in\dom\gamma$ and
$\gamma_n(t)\to\gamma(t)$, or eventually the point $t$ does not belong
to $\dom\gamma_n$, in which case the left hand side of
\eqref{eq:uppersem} is $0$.
By the reverse Fatou Lemma:
\begin{equation}
  \limsup_{n\to\infty}\int_a^bg\circ\gamma_n\,\chi_{\dom\gamma_n}\,dt
  \le \int_a^bg\circ\gamma\, \chi_{\dom\gamma}\,dt,
\end{equation}
which is the upper semicontinuity of $\Psi_{g,a,b}$.
\end{proof}
\par We now obtain a simplification of the description of Alberti
representations: one can assume that $\nu_\gamma=g\,\Psi_\gamma$,
where $g$ is a Borel function on $X$.
\begin{lem}
  Let $X$ be a complete separable and locally compact metric space with a Radon measure $\mu$
  admitting an Alberti representation $(P,\nu)$ satisfying
  \albcond. Then there is an Alberti representation
  $\albrep.'=(P',\nu')$ which satifies \albcond\ and with
  $\nu'_\gamma=g\,\Psi_\gamma$, where $g$ is Borel on $X$.
\end{lem}
\begin{proof}
  By Lemma \ref{frag_meas} we can find pairwise disjoint Borel sets
  $F_n\subset\frags_n(X)$ with $\frags(X)=\bigcup_nF_n$. Note that if
  $\gamma\in F_n$, then we have the bound $\|\Psi_\gamma\|\le
  2n$. Note that $\frags_n(X)$ is locally compact because $X$ is
  assumed to be locally compact; therefore, by Lemma~\ref{integration_meas} the integral
  \begin{equation}
    \tilde\mu=\sum_n\int_{F_n}\frac{1}{2n}\Psi_\gamma\,dP(\gamma)
  \end{equation}
  defines a finite Radon measure on $X$ with total mass at most $1$. If $\tilde\mu(A)=0$, then for
  $P$-a.e.~$\gamma$, $\Psi_\gamma(A)=0$ which implies
  $\nu_\gamma(A)=0$ as $\hmeas 1._\gamma\ll\Psi_\gamma$. Introducing
  the measure
  \begin{equation}
    \tilde P = \sum_n\frac{1}{2n}P\mrest F_n,
  \end{equation}
we obtain a finite Borel measure on $\frags(X)$ and we have
\begin{equation}
  \tilde\mu=\int_{\frags(X)}\Psi_\gamma\,d\tilde P(\gamma).
\end{equation} As $\tilde\mu\gg\mu$, by the
Radon-Nikodym Theorem \cite[Thm.~6.10]{rudin-real}, there is a Borel
function $\tilde g$ on $X$ with $d\mu=\tilde g\, d\tilde\mu$. Then
\begin{equation}
  \mu=\int_{\frags(X)}\tilde g\,\Psi_\gamma\,d\tilde P
\end{equation}
and the result follows letting
\begin{align}
  g&=\tilde g \tilde P (\frags(X));\\
  \nu'_\gamma&=g\Psi_\gamma;\\
  P'&=\frac{1}{\tilde P(\frags(X))}\tilde P.
\end{align}
From the way in which $\tilde P$ was obtained from $P$, we conclude
that if the original Alberti representation $(P,\nu)$ satisfied
\albcond, so does the new one $(P',\nu')$.
\end{proof}
\par We can now prove Theorem \ref{compact_reduction}.
\begin{proof}[Proof of Theorem \ref{compact_reduction}]
  By the previous lemma we can assume that
  $\nu_\gamma=g\,\Psi_\gamma$. Note that as
  \begin{equation}\label{mu_on_C}
    \mu\left(X\setminus C\right)=0,
  \end{equation}
for $P$-a.e.~$\gamma\in\frags(X)\setminus\frags(X,C)$, 
\begin{equation}\label{gnull}
  g\,\Psi_\gamma=0.
\end{equation}
%\textcolor{red}
{Note also that as $\mu(X\setminus C)=0$, we may assume that $g$ vanishes on
$X\setminus C$.}
In particular, replacing $P$ by 
\begin{equation}
  \frac{P\mrest\frags(X,C)}{P\left(\frags(X,C)\right)},
\end{equation}
and $g$ by
\begin{equation}
gP\left(\frags(X,C)\right),
\end{equation}
we can assume that 
\begin{equation}
  P\left(\frags(X)\setminus\frags(X,C)\right)=0, 
\end{equation}
and obtain for a Borel set $A$ the equations:
\begin{equation}
  \begin{aligned}
    \mu(A)&=\int_{\frags(X,C)}dP(\gamma)\int_Xg\chi_A\,d\Psi_\gamma &
    &\\
    &=\int_{\frags(X,C)}dP(\gamma)\int_Xg\chi_A\chi_C\,d\Psi_\gamma
    &&\text{(by \eqref{mu_on_C})}\\
    &=\int_{\frags(X,C)}dP(\gamma)\int_Xg\chi_A\,d\Psi(\frest
    C.(\gamma)) & &\\
    &=\int_{\frags(C)}d\mpush{\frest
      C.}.P(\gamma)\int_Xg\chi_A\,d\Psi_\gamma &&\text{(by pushing forward)}
  \end{aligned}
\end{equation}
where in the last step we used that $\frest C.$ is Borel by Lemma
\ref{borel_rest}. Letting $P'=\mpush{\frest
      C.}f.$ and $\nu'_\gamma=g\,\Psi_\gamma$, we get that $(P',\nu')$
    is an
Alberti representation of $\mu$ defined on $\frags(C)$ (using again
Lemma \ref{integration_meas}). From the way
in which $P'$ was obtained from $P$, if $(P,\nu)$ satisfies \albcond,
so does $(P',\nu')$.
\end{proof}
\par The following is a gluing principle for Alberti representations.
\begin{thm}\label{alb_glue}
  Let $\mu$ be a Radon measure on $Z$ and $U\subset Z$ Borel. Assume
  that there are disjoint Borel sets $U_n\subset U$ isometrically
  embedded in some complete separable and locally compact metric spaces $Y_n$ ($U_n\hookrightarrow
  Y_n$) with $\mu\mrest U_n$ admitting an Alberti representation
  defined on $\frags(Y_n)$ and satisfying \albcond. Then $\mu\mrest U$
  admits an Alberti representation defined on $\frags(Z)$ and
  satisfying \albcond.
\end{thm}
\begin{proof}
  Using that $\mu$ is Radon we can find disjoint compact subsets
  $\{C_\alpha\}$ such that for each $\alpha$ there is an $n_\alpha$ with
  $C_\alpha\subset U_{n_\alpha}$ and 
  \begin{equation}
    \mu\left(U\setminus\bigcup_\alpha C_\alpha\right)=0.
  \end{equation}
As $\mu\mrest U_{n_\alpha}$ admits an Alberti representation
$(P_{n_\alpha},\nu_{n_\alpha})$ defined on $\frags(Y_{n_\alpha})$ and
satisfiying \albcond, so does $\mu\mrest C_\alpha$ by letting
$P'_\alpha=P_{n_\alpha}$ and $\nu'_\alpha=\nu_{n_\alpha}\mrest
C_\alpha$. Then by Theorem \ref{compact_reduction}, $\mu\mrest
C_\alpha$ admits an Alberti representation $(P_\alpha,\nu_\alpha)$
defined on $\frags(C_\alpha)$ and satisfying \albcond. Observing that the
$\frags(C_\alpha)$ are disjoint Bo\-rel sub\-sets of $\frags(Z)$, we obtain
an Alberti representation $(P,\nu)$ of $\mu\mrest U$ defined on
$\frags(Z)$ and satisfing \albcond\ by letting
\begin{align}\label{axx1}
  P&=\sum_{\alpha=1}^\infty 2^{-\alpha}P_\alpha,\\
  \nu&=\sum_{\alpha=1}^\infty 2^{\alpha}\nu_\alpha.
\end{align}
%\textcolor{red}
{Note that in~(\ref{axx1}) we used that each $P_\alpha$ is a measure on
$\frags(Z)$ because it is a measure on $\frags(C_\alpha)\subset\frags(Z)$. }
\end{proof}
\par%\textcolor{red}
{We now introduce the key concept to build Alberti
representations: this is the notion of a set $S$ intersecting the image of each fragment $\gamma$ belonging to a
family of fragments ${\mathcal G}$ in a set of vanishing $1$-dimensional
Hausdorff measure. }
\begin{defn}\label{def:frag_nullity}
  If ${\mathcal G}\subset\frags(X)$, a Borel $S\subset X$ is said to
  be \textbf{${\mathcal G}$-null} if for each $\gamma\in{\mathcal G}$, $\hmeas
  1._\gamma(S)=0$. In case ${\mathcal G}$ is the family of fragments
  satisfying a set of conditions \albcond, we say that $S$ is \textbf{\albcond-null}.
\end{defn}
\par We now prove Lemma~\ref{biLip_dis}, which produces bi-Lipschitz
Alberti representations in Banach spaces. This is a step towards
Theorem~\ref{alberti_rep_prod}, which is the main tool we use to
produce Alberti representations.  Throughout the remainder of
this subsection the symbol $\convgeo$ denotes a \textbf{closed convex compact} subset of
$\bana^*$, where $\bana$ is a \textbf{separable Banach
  space}. %\textcolor{red}
{Working
inside the Banach space $\bana^*$ is
justified by the fact that any separable metric space $X$ can be
isometrically embedded in $l^\infty$ via a Kuratowski embedding, and
then one might take $\bana=l^1$ and $\bana^*=l^\infty$}. The
closed convex hull $\convgeo$ of $X$ in $\bana^*$ is compact by
\cite[Thm.~3.25]{rudin-functional}.
\begin{defn}
  For positive $\epsi$ and $\tau$, let $\QG\epsi,\tau.\subset\frags(\bana^*)$
  denote the compact subset
\begin{multline}
    \QG\epsi,\tau.=\left\{\gamma:[0,\tau]\to\convgeo:\forall
    t,s\in[0,\tau],\right. \\\left. |t-s|\le
    \|\gamma(t)-\gamma(s)\|_{\bana^*}\right.\left.\le(1+\epsi) |t-s|\right\}.
\end{multline}
\end{defn}
\begin{lem}
  Let $\Psi:\QG\epsi,\tau.\to P(\convgeo)$ be defined by 
  \begin{equation}
    \Psi_\gamma=\frac{1}{\tau}\mpush\gamma.\lebmeas\mrest[0,\tau];
  \end{equation} then $\Psi$ is continuous.
\end{lem}
\begin{proof}
  As the topological spaces $\QG\epsi,\tau.$ and $P(\convgeo)$ are metric spaces, it suffices to
  check sequential continuity. Assume that $\gamma_n\to\gamma$ in
  $\QG\epsi,\tau.$ and that $f$ is a real-valued continuous function
  with domain
  $\convgeo$. Then, as $f\circ\gamma_n\to f\circ\gamma$ uniformly,
  \begin{equation}
    \lim_{n\to\infty}\frac{1}{\tau}\int_0^\tau
    f\circ\gamma_n\,d\lebmeas= \frac{1}{\tau}\int_0^\tau f\circ\gamma\,d\lebmeas.
  \end{equation}
\end{proof}
\begin{lem}\label{biLip_dis}
Let $\mathcal{G}\subset\QG\epsi,\tau.$ be compact and
$\bar{H}\mathcal{G}$ denote the closed convex hull of
$\Psi(\mathcal{G})$ in $M(\convgeo)$. Let $\mu$ be a Radon measure on
$\convgeo$.  Then we can write
\begin{equation}
  \mu=\mu_{\mathcal{G}}+\mu\mrest F_{\mathcal{G}},
\end{equation} where $\rho_{\mathcal{G}}\in\bar{H}\mathcal{G}$ and
$\mu_{\mathcal{G}}\ll\rho_{\mathcal{G}}$; 
furthermore, there are a
Borel regular probability measure $P_{\mathcal{G}}$ on $\mathcal{G}$
and a Borel function $f_{\mathcal{G}}$ on $\convgeo$ such that
\begin{align}
  \rho_{\mathcal{G}}&=\int_{\mathcal{G}}\Psi_\gamma\,dP_{\mathcal{G}}\\
\label{gksrep}  \mu_{\mathcal{G}}&=\int_{\mathcal{G}}f_{\mathcal{G}}\Psi_\gamma\,dP_{\mathcal{G}}.
\end{align}
The set $F_{\mathcal{G}}$ is a $\mathcal{G}$-null $F_\sigma$.
\end{lem}
\begin{proof}
  The space $M(\convgeo)$ in the weak* topology is a Fr\'echet space
  and by \cite[Thm.~3.25]{rudin-functional} the closed convex hull of a
  compact subset is compact, so $\bar{H}\mathcal{G}$ is compact. By
  Rainwater's Lemma (\cite[Thm.~9.4.4]{rudin-complex-ball},
  \cite{rainwater-note})
  %\textcolor{red}
{we can find a set  $F_{\mathcal{G}}$ and  measures $\rho_{\mathcal{G}}$,
  $\mu_{\mathcal{G}}$ such that $\rho_{\mathcal{G}}\in
  \bar{H}\mathcal{G}$ and the following hold:}
  \begin{enumerate}
  \item The measure $\mu$ decomposes as $\mu=\mu_{\mathcal{G}}+\mu\mrest
    F_{\mathcal{G}}$.
    \item The measure $\mu_{\mathcal{G}}$ satisfies
      $\mu_{\mathcal{G}}\ll\rho_{\mathcal{G}}$.
      \item The set $F_{\mathcal{G}}$ is of class $F_\sigma$ and is $\Psi(\mathcal{G})$-null in the sense
        that, for each $\gamma\in \mathcal{G}$, $\Psi_\gamma(F_{\mathcal{G}})=0$.
  \end{enumerate}
  Note that (3) implies that $F_{\mathcal{G}}$ is ${\mathcal{G}}$-null
  because for $\gamma\in{\mathcal{G}}$, $\Psi_\gamma$ and $\hmeas
  1._\gamma$ are absolutely continuous one with respect to the other.
     By \cite[Thm.~3.28]{rudin-functional}
  there is a regular Borel probability measure $\pi_{\mathcal{G}}$
  supported on $\Psi(\mathcal{G})$ and such that
  \begin{equation}
    \rho_{\mathcal{G}}=\int_{\Psi({\mathcal{G}})}\sigma\,d\pi_{\mathcal{G}}(\sigma).
  \end{equation}
As $\Psi$ is continuous and ${\mathcal{G}}$ is compact, by
\cite[Thm.~6.9.7]{bogachev_measure}
there is a Borel $\tilde{\mathcal{G}}\subset{\mathcal{G}}$ with
$\Psi(\tilde{\mathcal{G}})=\Psi({\mathcal{G}})$ and such that $\Psi$ is
injective on $\tilde {\mathcal{G}}$ with a Borel inverse 
\begin{equation}
\tilde\Psi^{-1}:\Psi(\tilde{\mathcal{G}})\to{\mathcal{G}};
\end{equation}
in particular, 
\begin{equation}
  \begin{split}
    \rho_{\mathcal{G}}&=\int_{\Psi({\mathcal{G}})}\sigma\,d\pi_{\mathcal{G}}(\sigma)\\
    &=\int_{\Psi(\tilde{\mathcal{G}})}\Psi(\tilde\Psi^{-1}(\sigma))\,d\pi_{\mathcal{G}}(\sigma)\\
    &=\int_{{\mathcal{G}}}\Psi_\gamma\,d\underbrace{\left(\mpush\tilde\Psi^{-1}.\pi_{\mathcal{G}}\right)}_{P_{\mathcal{G}}}(\gamma).
  \end{split}
\end{equation}
If $f_{\mathcal{G}}$ is a Borel
representative of the Radon-Nikodym derivative of $\mu_{\mathcal{G}}$ with
respect to $\rho_{\mathcal{G}}$,
\begin{equation}
  \mu_{\mathcal{G}}=\int_{{\mathcal{G}}}f_{\mathcal{G}}\Psi_\gamma\,dP_{\mathcal{G}}.
\end{equation}
\end{proof}
The following theorem is the main tool we will use to produce Alberti
representations.
\begin{thm}\label{alberti_rep_prod}
  Let $X$ be a complete separable and locally compact metric space and $\mu$ a Radon measure on
  $X$. %\textcolor{red}
{Let $f \colon X\to\real^q$ and $g\colon X\to\real$ be Lipschitz functions,
  and let $v \colon X\to{\mathbb S}^{q-1}$, $\alpha \colon X\to(0,\pi/2)$ and
  $\delta \colon X\to(0,\infty)$ be Borel functions.}
  Let $\Omega$ denote the set of fragments in the $f$-direction
  of $\cone(v,\alpha)$ with \hbox{$g$-speed $>\delta$} (note that we
    require a strict inequality). Then the following are equivalent:
  \begin{enumerate}
  \item The measure $\mu$ admits an Alberti representation in the $f$-direction of
    $\cone(v,\alpha)$ with $g$-speed $>\delta$.
    \item  For each $K\subset X$ compact and $\Omega$-null, $\mu(K)=0$.
      \item For each $\epsi>0$, $\mu$ admits a $(1,1+\epsi)$-bi-Lipschitz
        Alberti representation in the $f$-direction of
    $\cone(v,\alpha)$ with \hbox{$g$-speed $>\delta$}.
  \end{enumerate}
\end{thm}
\begin{proof}
  We first show that (1) implies (2). Let $\albrep.=(P,\nu)$ be an
  Alberti representation with $P$ giving
  full measure to a Borel set $\Xi$ of fragments in the $f$-direction of
  $\cone(v,\alpha)$ with $g$-speed $>\delta$. If $K\subset X$ is
  $\Omega$-null, it is $\Xi$-null, so $\nu_\gamma(K)=0$ for
  $P$-a.e.~$\gamma$ and hence
  \begin{equation}
    \mu(K)=\int_\Xi\nu_\gamma(K)\,dP(\gamma)=0.
  \end{equation}
\par We now show that (2) implies (3) first assuming that
$X$ is compact and $v$, $\alpha$ and $\delta$ are constant.
%\textcolor{red}
{We first rescale $f \colon X\to\real^q$ and $g$ to be $1$-Lipschitz and
  embed $X$ isometrically in $l^\infty=l^\infty(\natural)$ (e.g.~using a Kuratowski
  embedding). Having thus identified $X$ with a subset of $l^\infty$
  we then  embed $\graph
(f,g)$ isometrically in $l^\infty\times\real^q\times\real=\bana^*$ with norm 
\begin{equation}
  \|(\xi,v,t)\|_{l^\infty\times\real^q\times\real}=\max(\|\xi\|_{l^\infty},\|v\|_2,|t|),
\end{equation}
and take the convex hull $\convgeo$ of the resulting set. We point out
that we use the notation $\bana^*$ as $\bana^*$ is the dual of a
separable Banach space $\bana$, where $\bana$ is the $l^1$-direct sum of $l^1(\natural)$, $\real^q$ (with the
Euclidean norm) and $\real$.
Note 
now that $f$ extends to a linear function $f \colon \bana^*\to\real^q$. In
fact, the function $f$, regarded as a function defined on $\graph(f,g)\subset\bana^*$, becomes just the canonical
projection $\pi_{\real^q} \colon \bana^*\to\real^q$ and thus the linear extension of
$f$ can be taken to be $\pi_{\real^q}$.}
\par Having chosen strictly decreasing sequences $\{\tau_n\}$, $\{\eta_n\}$
converging to $0$, we let
\begin{multline}
  {\mathcal{G}}_n=\bigl\{\gamma\in\QG\epsi,\tau_n. \colon \text{$\forall
    t,s\in[0,\tau_n]$,}\\\text{$\sgn(t-s)\left(f\circ\gamma(t)-f\circ\gamma(s)\right)\in\bar\cone(v,\alpha-\eta_n)$}\\
\text{and $\sgn(t-s)\left(
  g\circ\gamma(t)-g\circ\gamma(s)\right)\ge(\delta+\eta_n)|t-s|$}\bigr\},
\end{multline}
which is a compact subset of $\QG\epsi,\tau_n.$. We then apply Lemma
\ref{biLip_dis} repeatedly; having obtained
\begin{equation}
  \mu=\mu_{{\mathcal{G}}_1}+\mu\mrest
    F_{{\mathcal{G}}_1},
\end{equation}
we apply Lemma \ref{biLip_dis} to $\mu\mrest
    F_{{\mathcal{G}}_1}$ to obtain a decomposition
    \begin{equation}
      \mu=\mu_{{\mathcal{G}}_1}+\mu_{{\mathcal{G}}_2}+\mu\mrest F_{{{\mathcal{G}}_{1}\cup\mathcal{G}_2}}
    \end{equation}
where $\mu_{{\mathcal{G}}_1}$ and $\mu_{{\mathcal{G}}_2}$ are
concentrated on disjoint $G_\delta$ sets, and where $F_{{{\mathcal{G}}_{1}\cup\mathcal{G}_2}}$ is an intersection of
two $F_\sigma$ sets and is ${({{\mathcal{G}}_1}\cup{{\mathcal{G}}_2)}}$-null. Iterating,
\begin{equation}\label{iterxx}
  \mu=\sum_{n=1}^\infty\mu_{{\mathcal{G}}_n}+\mu\mrest F,
\end{equation}
where $F$ is an $F_{\sigma\delta}$ which is ${{\mathcal{G}}_n}$-null
for every $n$. 
\par We now show that $\mu(F)=0$ using condition (2). The
following observation \textbf{(Ob1)} will be used repeatedly: if
$\gamma\in\frags(\convgeo)$ and the $\{K_\alpha\}\subset\dom\gamma$ are
(countably many) compact sets with
\begin{equation}
  \lebmeas\left(\dom\gamma\setminus\bigcup_\alpha K_\alpha\right)=0
\end{equation}
and $\hmeas 1._{\gamma|K_\alpha}(F)=0$, then $\hmeas
1._\gamma(F)=0$. 
\par In the following we are going to uniformize the properties of the
fragment $\gamma$ by subdividing it; concretely, if $\gamma$ is a fragment in the $f$-direction of
$\cone(v,\alpha)$ with \hbox{$g$-speed $>\delta$}  as $(f\circ\gamma)'$,
$\metdiff\gamma$ and $g$ are measurable, using \textbf{(Ob1)} and Egorov
and Lusin's Theorems \cite[Thms.~7.1.12 and 7.1.13]{bogachev_measure}
and subdividing $\gamma$ into smaller subfragments %\textcolor{red}
{(i.e. we write
$\dom\gamma = \bigcup_iT_i\cup T_{\infty}$ where the $\{T_i\}$ are
pairwise disjoint, for $i<\infty$ each $T_i$ is compact, and $T_\infty$ has zero Lebesgue measure, and then focus on
each $\gamma_i=\gamma|T_i$ ($i\ne\infty$) separately, but keep writing
$\gamma$ for $\gamma_i$ to avoid burdening the notation),} we can assume that for all $t,s\in\dom\gamma$:
\begin{align}
  \|f\circ\gamma(t)-f\circ\gamma(s)-w(t-s)\|_2&\le\rho|t-s|;\label{conical_inclusion}\\
  (l-\rho)|t-s|\le\|\gamma(t)-\gamma(s)\|_{\bana^*}&\le(L+\rho)|t-s|;\label{bound_metdiff}\\
  \sgn(t-s)\left(g\circ\gamma(t)-g\circ\gamma(s)\right)&\ge(\delta+\eta-\rho)\|\gamma(t)-\gamma(s)\|_{\bana^*},\label{speed_bound}
\end{align}
where $\eta>0$, $w\in\cone(v,\alpha-\eta)$, $\frac{L}{l}<1+\epsi$ and $\rho$
is small enough and to be chosen later. Let
$I_\gamma=[\min\dom\gamma,\max\dom\gamma]$ denote the smallest
interval containing the domain of $\gamma$.
%\textcolor{red}
{Moreover, after further subdividing $\gamma$, and by applying Lebesgue's Differentiation Theorem
we can also assume that for all $t,s\in I_\gamma$ with at least one of
them in $\dom\gamma$,}
\begin{equation}\label{dom_dens}
  \lebmeas\left(\dom\gamma\cap[t,s]\right)\ge(1-\rho)|t-s|.
\end{equation}
Let $\{(u_n,v_n)\}$ be the components of
$I_\gamma\setminus\dom\gamma$. On each component $(u_n,v_n)$ we extend
$\gamma$ using the fact that $\convgeo$ is convex
\begin{equation}
  \gamma(t)=\frac{t-u_n}{v_n-u_n}\gamma(v_n)+\frac{v_n-t}{v_n-u_n}\gamma(u_n);
\end{equation}
if $(s,t)\subset (u_n,v_n)$,
\begin{equation}
  \gamma(t)-\gamma(s)=\frac{t-s}{v_n-u_n}(\gamma(v_n)-\gamma(u_n))
\end{equation}
so that \eqref{bound_metdiff} remains true; in particular, the bound
on the metric derivative implies that the extension is
$(L+\rho)$-Lipschitz. Similarly, as %\textcolor{red}
{$f$ is linear
(after embedding its graph in $\bana^*$)}:
\begin{align}
  f\circ\gamma(t)-f\circ\gamma(s)&=\frac{t-s}{v_n-u_n}(f\circ\gamma(v_n)-f\circ\gamma(u_n))\\
    g\circ\gamma(t)-g\circ\gamma(s)&=\frac{t-s}{v_n-u_n}(g\circ\gamma(v_n)-g\circ\gamma(u_n))
\end{align} so that \eqref{conical_inclusion} and
\eqref{speed_bound} remain true. If $t,s\in
I_\gamma\setminus\dom\gamma$ but in different components, there are
$s',t'\in\dom\gamma$:
\begin{equation}
  s\le s'\le t'\le t
\end{equation}
and each of  $(s,s')$ and $(t',t)$ is contained in a component of
$I_\gamma\setminus\dom\gamma$. %\textcolor{red}
{In particular as $t'\in\dom\gamma$ we
have from~(\ref{dom_dens}) that:
\begin{equation}
  \label{eq:fmfx1}
  (1-\rho)|s-t'|\le\lebmeas(\dom\gamma\cap[s,t']);
\end{equation}
but we also have $(s,s')\cap\dom\gamma=\emptyset$, and so get
\begin{equation}
  \label{eq:fmfx2}
  \lebmeas(\dom\gamma\cap[s,t']) \le |s-t'|-|s-s'|,
\end{equation}
from which we infer:
\begin{equation}
  \label{eq:fmfx3}
  |s-s'|\le |s-t'|-(1-\rho)|s-t'|\le \rho|s-t'|.
\end{equation}
This argument establishes the first equation~(\ref{fffr1}) of the
following inequalities, and
the second one follows by a similar argument:
\begin{align}
  |s-s'|&\le\rho|s-t|;\label{fffr1}\\
  |t'-t|&\le\rho|s-t|.
\end{align}}
Therefore, \eqref{conical_inclusion}, \eqref{bound_metdiff} and
\eqref{speed_bound} generalize to
  \begin{equation}
  \|f\circ\gamma(t)-f\circ\gamma(s)-w(t-s)\|_2\le2\rho(L+\rho+\|v\|)|t-s|;\label{adj_conical_inclusion}
  \end{equation}
  \begin{multline}\label{adj_bound_metdiff}
  \left[(l-\rho)(1-2\rho)-2(L+\rho)\rho\right]|t-s|\le\|\gamma(t)-\gamma(s)\|_{\bana^*}\le\\
  \left[L+\rho(1+2(L+\rho))\right]|t-s|;
  \end{multline}
  \begin{multline}\label{adj_speed_bound}
\sgn(t-s)\left(g\circ\gamma(t)-g\circ\gamma(s)\right)\ge\Biggl(\delta+\eta-\rho
  \\ -{}2\rho
\frac{L+\delta+\eta}{(l-\rho)(1-2\rho)-2(L+\rho)\rho}
\Biggr)\|\gamma(t)-\gamma(s)\|_{\bana^*}.
  \end{multline}
If $\eta'\in(0,\eta)$, for $\rho$ sufficiently small we have:
\begin{align}
  \ball w,2\rho(L+\rho+\|v\|).&\in\cone(v,\alpha-\eta');\\
  \frac{L+\rho(1+2(L+\rho))}{(l-\rho)(1-2\rho)-2(L+\rho)\rho}&\le1+\epsi;\\
\left(\delta+\eta-\rho-2\rho
\frac{L+\delta+\eta}{(l-\rho)(1-2\rho)-2(L+\rho)\rho}
\right)&\ge\delta+\eta'.
\end{align}
We can now reparametrize $\gamma$ to be $(1,1+\epsi)$-bi-Lipschitz,
subdivide the domain and choose $n$ sufficiently large so that
$\gamma\in{{\mathcal{G}}_n}$. Thus,
$\hmeas 1._\gamma(F)=0$. This implies that
\begin{equation}
    \mu=\sum_{n=1}^\infty\mu_{{\mathcal{G}}_n}\quad\text{(on account
      of condition (2))}.
\end{equation}
\par As the measures $\mu_{{\mathcal{G}}_n}$ are concentrated on
pairwise disjoint sets, by Theorem \ref{alb_glue} we obtain a
representation $\albrep.'=(P',\nu')$ of $\mu$ in the $f$-direction of
$\cone(v,\alpha)$ with $g$\nobreakdash-\hspace{0pt}speed $>\delta$ with $P'\in
P(\frags(\convgeo))$; Theorem \ref{compact_reduction} gives a
representation $\albrep.=(P,\nu)$ with $P\in P(\frags(X))$.
\par The case in which $X$ is not compact and $v$, $\alpha$ and
$\delta$ are not constant is treated by using Egorov and Lusin's
Theorems to find Borel functions $(v_n,\alpha_n,\delta_n)$ such that:
\begin{itemize}
\item There are disjoint compacts $\{K_{n,j}\}_j$ with
  $\mu\left(X\setminus\bigcup_jK_{n,j}\right)=0$ and the
  $(v_n,\alpha_n,\delta_n)$ are constant on each $K_{n,j}$.
\item The functions $\delta_n\searrow\delta$ pointwise.
\item %\textcolor{red}
{%\color{red} 
For each $x\in X$ one has
  $\cone(v_n(x),\alpha_n(x))\subset\cone(v(x),\alpha(x))$ and as
  $n\nearrow\infty$ the cones $\{\cone(v_n(x),\alpha_n(x))\}_n$ converge to
  $\cone(v(x),\alpha(x))$ (i.e.~the Hausdorff distance between the
  intersections %\textcolor{red}
{$\cone(v_n(x),\alpha_n(x))\cap{\mathbb
      S}^{q-1}$ and $\cone(v(x),\alpha(x))\cap{\mathbb S}^{q-1}$
    converges to $0$ as $n\nearrow\infty$}).}
\end{itemize}
\par One then applies the construction for constant cone fields and speeds on the
$K_{n,j}$ recursively to obtain $\mu=\sum_{n,j}\mu_{n,j}$ where the
measures $\{\mu_{n,j}\}_{n,j}$ have pairwise disjoint supports and each
$\mu_{n,j}$ has an Alberti representation in the $f$-direction of
$\cone(v_n,\alpha_n)$ with $g$-speed $>\delta_n$.  These
representations are then glued via Theorem \ref{alb_glue} to give a
representation of $\mu$. 
\par That (3) implies (1) is immediate from the definitions.
\end{proof}
\subsection{Derivations and $L^\infty$-modules}\label{subsec:der_modules}
In this subsection we recall some facts about derivations and $L^\infty$-modules.
 Recall that the Banach space $L^\infty(\mu)$ is a real Banach
algebra, and, in particular, a ring. We will denote by
$L^\infty_+(\mu)$ the set of nonnegative elements, i.e.~those $f$
which satisfy:
\begin{equation}
  \mu\left\{x\in X:f(x)<0\right\}=0.
\end{equation}
The map absolute value 
\begin{align}
  L^\infty(\mu)&\xrightarrow{|\cdot|}L^\infty_+(\mu)\\
  f&\mapsto|f|\notag
\end{align}
is sort of an $L^\infty$-valued seminorm:
\begin{enumerate}
\item For all $c_1,c_2\in\real$, and for all $f_1,f_2\in L^\infty(\mu)$,
  \begin{equation}
    |c_1f_1+c_2f_2|\le|c_1||f_1|+|c_2||f_2|.
  \end{equation}
\item For all $f_1,f_2\in L^\infty(\mu)$,
  \begin{equation}
    |f_1f_2|=|f_1||f_2|.
  \end{equation}
\end{enumerate}
note also that $L^\infty_+(\mu)=\left\{f:f=|f|\right\}$. In particular,
$\|f\|_{L^\infty(\mu)} =\|\,|f|\,\|_{L^\infty(\mu)}$. 
\par %\textcolor{red}
{An \textbf{$L^\infty(\mu)$-module} $M$ is a Banach space $M$ which
is algebraically an \hbox{$L^\infty(\mu)$-mo}\-dule with the
additional requirement that the $L^\infty(\mu)$-action is bounded;
i.e., if $\lambda\in L^\infty(\mu)$ and $m\in M$, then $\|\lambda
m\|_M\le \|\lambda\|_{L^\infty(\mu)}\|m\|_M$.} An algebraic submodule
$M'\subset M$ is called a \textbf{submodule} if it is closed. Algebraic
concepts like direct sum, linear independence, free module and basis
of a free module extend immediately. Given $S\subset M$ the linear
span of $S$ will be denoted by $\linspan_{L^\infty(\mu)}(S)$.  The
\textbf{dual $\dualmod M.$} of an $L^\infty(\mu)$-module $M$ is the
$L^\infty(\mu)$-module of bounded module homomorphisms
\begin{equation}
  f\colon M\to L^\infty(\mu).
\end{equation}
\par Among $L^\infty(\mu)$-modules a special r\^ole is played by
\textbf{$L^\infty(\mu)$-normed modules}.
\begin{defn}\label{defn:local_norm}
  An $L^\infty(\mu)$-module $M$ is said to be an
  \textbf{$L^\infty(\mu)$-normed module} if there is a map
  \begin{equation}
    |\cdot|_{M,{\text{loc}}}:M\to L^\infty_+(\mu) 
  \end{equation}
such that:
\begin{enumerate}
\item The function $|\cdot|_{M,{\text{loc}}}$ is a ``seminorm'': for
  all $c_1,c_2\in\real$, and all $m_1,m_2\in M$,
  \begin{equation}
        |c_1m_1+c_2m_2|_{M,{\text{loc}}}\le|c_1||m_1|_{M,{\text{loc}}}+|c_2||m_2|_{M,{\text{loc}}}.
  \end{equation}
\item For each $\lambda\in L^\infty(\mu)$ and each $m\in M$,
  \begin{equation}
    |\lambda m|_{M,{\text{loc}}}=|\lambda|\,|m|_{M,{\text{loc}}}.
  \end{equation}
\item The local seminorm $|\cdot|_{M,{\text{loc}}}$ can be used to
  reconstruct the norm of any $m\in M$:
  \begin{equation}
    \|m\|_M=\|\,|m|_{M,{\text{loc}}}\,\|_{L^\infty(\mu)}.
  \end{equation}
\end{enumerate}
\end{defn}
\begin{rem}
Note that a submodule of an $L^\infty(\mu)$-normed module is an
\hbox{$L^\infty(\mu)$-nor}\-med module. The ring $L^\infty(\mu)$ has many
idempotents $\chi_U\in L^\infty_+(\mu)$ (characteristic function of the $\mu$-measurable
set $U$), which we will call projections as they give direct sum
decompositions
\begin{equation}
  L^\infty(\mu)=\chi_U\,L^\infty(\mu)\oplus (1-\chi_U)\,L^\infty(\mu),
\end{equation}
$\lambda\,L^\infty(\mu)$ denoting the ideal / submodule generated by
$\lambda$.
\end{rem}
\par A simple and important criterion
(\cite[Thm.~9]{phillips_cstarnorm}) for an $L^\infty(\mu)$-module $M$ to be
an $L^\infty(\mu)$-normed module is the following:
\begin{lem}
  An $L^\infty(\mu)$-module $M$ is an $L^\infty(\mu)$-normed module if
  and only if for each $U$ $\mu$-measurable and each $m\in M$:
  \begin{equation}
    \|m\|_M=\max\left(\|\chi_Um\|_M,\|(1-\chi_U)m\|_M\right);
  \end{equation}
in particular if the $U_\alpha$ are disjoint and
$\mu\left(X\setminus\bigcup U_\alpha\right)=0$, 
\begin{equation}
  \|m\|_M=\sup_\alpha\|\chi_{U_\alpha}m\|_M.
\end{equation}
\end{lem}
\begin{exa}
While $L^\infty(\mu)$ is an $L^\infty(\mu)$-normed module, the
previous lemma implies that for $p\in[1,\infty)$, the $L^p(\mu)$ are
  not $L^\infty(\mu)$-normed modules.
\end{exa}
\begin{lem}
    If $M$ is an $L^\infty(\mu)$-module, $\dualmod M.$ is an $L^\infty(\mu)$-normed module.
  \end{lem}
  \begin{proof}
    Given $\xi\in\dualmod M.$ and $\epsi>0$, choose $m\in M$ with
    $\|m\|_M=1$ and
    \begin{equation}
      \|\xi\|_{\dualmod M.}\le \|\xi(m)\|_{L^\infty(\mu)}+\epsi.
    \end{equation}
As $L^\infty(\mu)$ is an $L^\infty(\mu)$-normed module, for each $U$
$\mu$-measurable,
\begin{equation}
\begin{split}
  \|\xi(m)\|_{L^\infty(\mu)}&=\max\left(\|\chi_U\xi(m)\|_{L^\infty(\mu)},\|(1-\chi_U)\xi(m)\|_{L^\infty(\mu)}\right)\\
  &\le\max\left(\|\chi_U \xi\|_{\dualmod
    M.},\|(1-\chi_U)\xi\|_{\dualmod M.}\right).
\end{split}
\end{equation}
  \end{proof}
\begin{rem}Actually, the previous argument shows that $\boundlin X,N.$ (bounded
linear maps from the Banach space $X$ to $N$) is an
$L^\infty(\mu)$-normed module whenever $N$ is too.
\end{rem}
\par For $L^\infty(\mu)$-normed modules there is an analogue of the
Hahn-Banach Theorem \cite[Thm.~5]{weaver00}:
\begin{lem}\label{lem:hahn-banach}
  Let $M'$ be an algebraic submodule of the $L^\infty(\mu)$-normed
  module $M$ and $\Phi'\in\dualmod M'.$ with norm $\le c$. Then there
  is a $\Phi\in\dualmod M.$ extending $\Phi'$ with norm $\le c$.
\end{lem}
  We now introduce derivations, but first recall some facts about
Lipschitz functions.  We will denote by $\lipfun X.$ the space of
real-valued Lipschitz functions defined on $X$ and by $\lipalg X.$ the
real algebra of real-valued bounded Lipschitz functions defined on
$X$. For $f\in\lipfun X.$ its global Lipschitz constant will be
denoted by $\glip f.$.  For $f\in\lipalg X.$ we define the norm
  \begin{equation}
    \|f\|_{\lipalg X.} = \max(\|f\|_\infty,\glip f.).
  \end{equation}
  This gives $(\lipalg X.,\|\cdot\|_{\lipalg X.})$ the structure of a
  Banach algebra, compare \cite[Sec.~4.1]{weaver_book99}. An important
  property of $\lipalg X.$ is that it is a dual Banach. 
  For more information we refer the reader to
  \cite[Chap.~2]{weaver_book99}. For the scope of the present work,
  the most useful topology on $\lipalg X.$ is the weak*
  topology. %\textcolor{red}
{As
  we are assuming $X$ separable, the weak* topology will be metrizable on bounded
  subsets of $\lipalg X.$: one way to see this is to use the
  description of the predual in~\cite{deleeuw61} and apply the
  Krein--\v Smulian Theorem (see the argument on~\cite[pg.~34]{weaver_book99}), or one can use
  the description of the predual in~\cite{arens_eels56} (see
  \cite[Sec.~2.2]{weaver_book99} for details);} moreover, it
  turns out that a sequence $f_n\xrightarrow{\text{w*}}f$ if and only
  if $f_n\to f$ pointwise and if $\sup_n\glip
  f_n. <\infty$. {Note also that this kind of convergence implies uniform convergence on compact
  subsets}.
\par In the sequel, especially in Section
\ref{sec:derivations_alberti}, we will sometimes need to consider
functions which are Lipschitz with respect to a different (pseudo)-distance
$d'$. If $f\in\lipfun X.$ is $L$-Lipschitz with respect to $d'$, we will say that
$f$ is $(1,d')$-Lipschitz.

  \begin{defn}
    A \textbf{derivation $D\colon\lipalg X.\to L^\infty(\mu)$} is a weak*
    continuous, bounded linear map satisfying the product rule:
    \begin{equation}
      D(fg)=fDg+gDf.
    \end{equation}
  \end{defn}
Note that the notion of derivation depends on the measure $\mu$ and
that the product rule implies that $Df=0$ if $f$ is constant. The
collection of all derivations $\wder\mu.$ is an $L^\infty(\mu)$-normed
module. 
 Moreover, any $D\in
\wder\mu\mrest U.$ gives rise to an element of $\wder\mu.$ by
extending $Df$ to be $0$ on the complement of $U$. In this way,
$\wder\mu\mrest U.$ can be naturally identified with the submodule
$\chi_U\, \wder\mu.$ of $\wder\mu.$. Derivations are local
in the following sense \cite[Lem.~27]{weaver00}:
\begin{lem}\label{lem:locality_derivations}
  If $U$ is $\mu$-measurable and if $f,g\in\lipalg X.$ agree on $U$,
  then for each derivation $D\in\wder\mu.$, $\chi_UDf=\chi_UDg$.
\end{lem}
\begin{rem}
  \label{rem:derivation_extension}
This locality property allows to extend derivations to the space of Lipschitz
functions. 
Considering countably many pairwise disjoint compact sets $K_\alpha$ with
uniformly bounded diameter and points
$c_\alpha\in K_\alpha$,
with $\mu\left(X\setminus\bigcup_\alpha K_\alpha\right)=0$, we define
%\textcolor{red}
{for $f\in\lipfun X.$
\begin{equation}
  Df=\sum_\alpha\chi_{K_\alpha}D\left(\min\left(\max\left(f-f(c_\alpha),-\max_{K_\alpha}|f-f(c_\alpha)|\right),
\max_{K_\alpha}|f
-f(c_\alpha)|\right)\right),
\end{equation} which agrees} with the previous definition for
$f\in\lipalg X.$, using locality (Lemma~\ref{lem:locality_derivations}). Observe that we used that the norm of
$f-f(c_\alpha)$ in $\lipalg K_\alpha.$ is bounded in terms of $\glip
f.$ and the diameter of $K_\alpha$. Again using locality, it is
possible to show that the extension does not depend on the choice of
the $K_\alpha$ or the point $c_\alpha$. 
\end{rem}
\begin{defn}
  The dual $L^\infty(\mu)$-normed module of $\wder\mu.$ will be
  denoted by $\wform\mu.$ and called \textbf{module of forms}. For
$f\in\lipalg X.$ (or $\lipfun X.$) let $df\in\wform\mu.$ be defined by
\begin{equation}
  \forall D\in\wder\mu.,\quad \langle df,D\rangle=Df;
\end{equation}
the map 
\begin{equation}
  \begin{aligned}
    d:\lipalg X.&\to\wform \mu.\\
    f&\mapsto df
  \end{aligned}
\end{equation}
is linear, satisfies the product rule
\begin{equation}
  d(fg)=fdg+gdf
\end{equation}
and is weak* continuous.
\end{defn}
The following result is also used repeatedly in the following. Proofs
can be found in \cite{gong11-revised,derivdiff}.
\begin{cor}\label{cor:pseudoduality}
  Let $\{D_i\}_{i=1}^N\subset\wder\mu\mrest U.$ be linearly
  independent, where $U$ is Borel. Then there are disjoint Borel
  subsets $V_\alpha\subset U$ with $\mu(U\setminus\bigcup_\alpha
  V_\alpha)=0$ such that, for each $\alpha$, there are $1$-Lipschitz
  functions $\{g_{i,\alpha}\}_{i=1}^N$ and derivations
  $\{D_{\alpha,i}\}_{i=1}^N$ in the linear span of
  $\{\chi_{V_\alpha}D_i\}_{i=1}^N$, satisfying
%\textcolor{red}
{\begin{equation}\label{eq:pseudoduality}
D_{\alpha,i}g_{\alpha,j}=\delta_{ij}\chi_{V_\alpha}.
\end{equation}}
\end{cor}
\begin{rem}
  The space $\boundlin\lipalg X.,L^\infty(\mu).$ of bounded linear
  maps \begin{equation}\lipalg X.\to L^\infty(\mu)
  \end{equation}
  is a dual Banach space. In fact, let
  $B_1$ denote the
  unit ball in $\lipalg X.$ and consider the Banach space
\begin{equation}
  l^1\left(B_1; L^1(\mu)\right)=\left\{f:B_1\to L^1(\mu):\sum_{s\in B_1}\left\|f(s)\right\|_{L^1(\mu)}<\infty\right\};
\end{equation}
then $\boundlin\lipalg X.,L^\infty(\mu).$ can be identified with a
subspace of the dual of the generalized $l^1$-space:
$l^1\left(B_1;L^1(\mu)\right)$; on $\wder\mu.\subset\boundlin\lipalg
X.,L^\infty(\mu).$ we can consider the corresponding weak* topology;
in concrete terms, this is the coarsest topology making continuous all
the seminorms $\{\nu_\xi\}_{l^1\left(B_1; L^1(\mu)\right)}$:
  \begin{equation}
    \nu_{\xi}(D)=\left|\sum_{f\in B_1}\int \xi(f)\,Df\,d\mu\right|;
  \end{equation}
a basis for the weak* topology consists of the sets
\begin{equation}\label{omegafxx}
  \Omega(\xi_1,\ldots,\xi_k; D_0; \epsi)=\left\{D\in\wder\mu.:
    \nu_{\xi_i}(D-D_0)<\epsi\,(\forall i\in\left\{1,\ldots,k\right\})\right\}
\end{equation}
where $\left\{\xi_1,\ldots,\xi_k\right\}\subset l^1\left(B_1;
  L^1(\mu)\right)$, $D_0\in\wder\mu.$, $\epsi>0$.
\end{rem}
\subsection{Differentiability spaces}\label{subsec:diff_spaces}
In this subsection we recall some facts about differentiability
spaces.  We first recall the definition of the infinitesimal Lipschitz
constants of $f\in\lipfun X.$.
\begin{defn}
  For $f\in\lipfun X.$ 
we define the \textbf{lower and upper variations of $f$ 
at $x$ below scale $r$}  by
\begin{align}
 \biglip f(x,r)&=\sup_{s\le r}\sup_{y\in B(x,s)}\frac{|f(x)-f(y)|}{s}\\
 \smllip f(x,r)&=\inf_{s\le r}\sup_{y\in B(x,s)}\frac{|f(x)-f(y)|}{s}.
\end{align}
The \textbf{(``big'' and ``small'') infinitesimal Lipschitz constants of $f$ at
$x$} are defined by
\begin{align}
  \biglip f(x)&=\inf_{r\ge0}\biglip f(x,r)\\
  \smllip f(x)&=\sup_{r\ge0}\smllip f(x,r).
\end{align}
The functions $\biglip f(\cdot,r)$, $\smllip f(\cdot, r)$, 
$\biglip f$ and $\smllip f$ are Borel. Other notations in use for
$\biglip f$ (resp.~$\smllip f$) are ${\rm Lip}f$ (resp.~${\rm lip}f$).
\end{defn}
\par The key notion for generalizing the classical Rademacher's Theorem to
metric spaces is that of infinitesimal independence.
\begin{defn}\label{def:infi_indep}
  Let $\{f_i\}_{i=1}^n\subset\lipfun X.$, then
  \begin{equation}
    \begin{aligned}
      \Phi_{x,\{f_i\}_{i=1}^n}:\real^n&\to[0,\infty)\\
        (a_i)_{i=1}^n&\mapsto\biglip\left(\sum_{i=1}^na_if_i\right)(x)
    \end{aligned}
  \end{equation}
defines a seminorm on $\real^n$; the functions $\{f_i\}_{i=1}^n$ are
called \textbf{infinitesimally independent on a $\mu$-measurable $A\subset X$} if
$\Phi_{x,\{f_i\}_{i=1}^n}$ is a norm for $\mu$-a.e.~$x\in A$.  
\end{defn}
\par We can now define differentiability spaces.
\begin{defn}
  A metric measure space $(X,\mu)$ is called a \textbf{differentiability space}
  if there is a uniform bound $N$ on the number of Lipschitz functions
  that can be infinitesimally independent on a positive measure
  set. In this case
  there are (countably many) $\mu$-measurable $\{X_\alpha\}$ with $X=\bigcup_\alpha X_\alpha$
  such that:
  \begin{enumerate}
  \item For each $\alpha$ there is
    $\{x_\alpha^j\}_{j=1}^{N_\alpha}\subset\lipfun X.$ such that for each
    $f\in\lipfun X.$ there are unique $\left\{\frac{\partial
      f}{\partial x_\alpha^j}\right\}_{j=1}^{N_\alpha} \subset
    L^\infty(\mu\mrest X_\alpha)$ such that
    \begin{equation}\label{diff_linearization}
      \biglip\left(f-\sum_{j=1}^n\frac{\partial
      f}{\partial x_\alpha^j}(x)x_\alpha^j\right)(x)=0\quad\text{for
        $\mu$-a.e.~$x\in X_\alpha$}.
    \end{equation}
\item The $N_\alpha$ are uniformly bounded by $N$, and the minimal
  value of $N$ is
  called \textbf{the differentiability dimension}.
  \end{enumerate}
\par Each $(X_\alpha,\{x_\alpha^j\}_{j=1}^{N_\alpha})$ is called a
\textbf{chart}, $N_\alpha$ is called the \textbf{dimension} of the
chart, the $\{x_\alpha^j\}_{j=1}^{N_\alpha}$ are called \textbf{chart functions}, and
the maps:
\begin{equation}
  \begin{aligned}
    \frac{\partial
      }{\partial x_\alpha^j}:\lipfun X.&\to L^\infty(\mu\mrest
    X_\alpha)\\
    f&\mapsto \frac{\partial
      f}{\partial x_\alpha^j}
  \end{aligned}
\end{equation}
are called \textbf{partial derivative operators}; the representatives $\frac{\partial
      f}{\partial x_\alpha^j}$ can be taken to be Borel if $X_\alpha$
is Borel.
\par %\textcolor{red}
{Note that differentiability spaces are usually assumed to be
Polish or even complete and separable. Separability is quite a natural
requirement, where completeness is usually assumed for convenience:
see \cite[Subsec.~2, pg.~12]{cks_metric_diff} for more details. 
}
\par%\textcolor{red}
{Another notion that we will use is that of
  $\sigma$-differentiability spaces, which have been defined between
  Theorem~\ref{thm:keith} and Theorem~\ref{thm:bate_speight}.}
\end{defn}
% We now recall the definitions of doubling and asymptotically doubling measures.
% \begin{defn}
%   A measure $\mu$ is called doubling at $x$ if $\exists r_0(x)>0,
%   C_\mu(x)<\infty$ such that for $r_0(x)\ge r>0$
%   \begin{equation}
% 0<    \mu\left(\ball x,2r.\right)\le C_\mu(x)\mu\left(\ball x,r.\right);
%   \end{equation}
% $\mu$ is called asymptotically doubling at $x$ if $\exists
%   C_\mu(x)<\infty$:
%   \begin{equation}\label{eq:asym_doubling}
%     \limsup_{r\to0}\frac{\mu\left(\ball x,2r.\right)}{\mu\left(\ball
%       x,r.\right)}\le C_\mu(x);
%   \end{equation}
% $\mu$ is called doubling if $r_0(x)=\diam X$ and $C_\mu(x)=C_\mu$,
%   independent of the point (also called a doubling constant of $\mu$);
%   similarly, if \eqref{eq:asym_doubling} holds for a uniform
%   $C_\mu(x)=C_\mu$, $\mu$ is called asymptotically doubling. Finally
%   $\mu$ is called $\sigma$-asymptotically doubling if there are
%   disjoint Borel sets $\{X_\alpha\}$ with $\mu(X_\alpha)>0$,
%   $\mu\left(X\setminus\bigcup_\alpha X_\alpha\right)=0$ and $\mu\mrest
%   X_\alpha$ doubling on $X_\alpha$.
% \end{defn}
\section{Derivations and Alberti representations}\label{sec:derivations_alberti}
\subsection{From Alberti representations to Derivations}\label{subsec:derivations_alberti}
In this subsection we associate derivations to Alberti
representations; the fundamental construction is provided by
Theorem \ref{thm:alb_derivation}; its proof uses the following lemma
to check the measurability of certain functions. %\textcolor{red}
{In the following
$\bborel X.$ will denote the set of real-valued bounded Borel functions defined on
$X$; we will often use without explicit mention the
elementary fact that for each $f\in L^\infty(\mu)$ one can find a
representative in $\bborel X.$.}
\begin{lem}\label{lem:meas_fix}
  Suppose that $X$ is Polish and $\albrep.=(P,\nu)$ is an Alberti
  representation. Then for each $(f,g)\in\lipfun X.\times\bborel X.$
  the map:
  \begin{equation}
    \label{eq:meas_fix_s1}
    \begin{aligned}
      H_{f,g}:\frags(X)&\to\real\\
      \gamma&\mapsto\int_{\dom\gamma}(f\circ\gamma)'(t)\,(g\circ\gamma)(t)\,d(\mpush\gamma^{-1}.\nu_\gamma)(t)
    \end{aligned}
  \end{equation}
  is Borel.
\end{lem}
\begin{proof}
  %\textcolor{red}
{We start by showing that $H_{f,1}$ is Borel. Without loss of
  generality we can assume that $f$ is $1$-Lipschitz and by
  Lemma~\ref{lem:frag_compact}
we have just to show}
  that the restriction $H_{f,1}|\frags_n(X)$ is Borel. The first step
  is to use a measure-theoretic argument to extend ${f\circ\gamma}$ to
  have domain over an inverval in such a way that the extension
  is ``measurable'' in $\gamma$: this will be used to define a Borel
  map in~(\ref{eq:meas_fix_p6}).
Recall that the
  set
  \begin{equation}
    \label{eq:meas_fix_p0}
    \lipalgb \real,n.=\left\{\psi\in\lipalg\real.: \glip\psi.\le n\right\}
  \end{equation}
  is closed in $\lipalg \real.$ and is a Polish space if we consider
  the weak* topology. Note that the set
  \begin{equation}
    \label{eq:meas_fix_p1}
    \text{\normalfont
      Ext}_n=\left\{(\gamma,\psi)\in\frags_n(X)\times\lipalgb
      \real,n.:\text{$\psi$ extends $f\circ\gamma$}\right\}
  \end{equation}
  is closed in $\frags_n(X)\times\lipalgb \real,n.$. For each
  $\gamma$, by taking a McShane extension of ${f\circ\gamma}$, one
  concludes that the section $(\text{\normalfont Ext}_n)_\gamma$ is
  non-empty. Moreover, by Ascoli-Arzel\`a each section
  $(\text{\normalfont Ext}_n)_\gamma$ is compact. By the Lusin-Novikov
  Uniformization Theorem \cite[Thm.~18.10]{kechris_desc}, the set
  $\text{\normalfont Ext}_n$ admits a Borel uniformization and thus
  there is a Borel uniformizing function:
  \begin{equation}
    \label{eq:meas_fix_p2}
    F_n:\frags_n(X)\to\lipalgb \real,n.
  \end{equation}
  with $(\gamma,F_n(\gamma))\in\text{\normalfont Ext}_n$. To avoid a
  cumbersome notation for the value of $F_n(\gamma)$ at a point $t$, we will
  also write $F_{n,\gamma}$ to denote $F_n(\gamma)$.
  We now show that the function $G_{f,1}:\frags_n(X)\to\real$ defined by:
  \begin{equation}
    \label{eq:meas_fix_p3}
    G_{f,1}(\gamma)=\int_\real F_{n,\gamma}(t)\,d(\mpush\gamma^{-1}.\nu_\gamma)(t)
  \end{equation} is Borel.
  For $k\in\natural$ we let $\Delta^k$ denote the collection of closed
  dyadic intervals in $\real$ of the form
  $\left[\frac{m}{2^k},\frac{m+1}{2^k}\right]$ for
  $m\in\zahlen$. Given an interval $I\in\Delta^k$ we denote by
  $a_I,b_I$ its extremes so that $[a_I,b_I]=I$. Consider the maps
  $G_{f,1;k}:\frags_n(X)\to\real$ defined by:
  \begin{equation}
    \label{eq:meas_fix_p4}
    G_{f,1;k}(\gamma)=\sum_{I\in\Delta^k}\int_I F_{n,\gamma}(a_I)\,d(\mpush\gamma^{-1}.\nu_\gamma)(t);
  \end{equation}
  as the measure $\mpush\gamma^{-1}.\nu_\gamma$ is finite and
  $F_{n,\gamma}$ is Lipschitz,
  $\lim_{k\to\infty}G_{f,1;k}(\gamma)=G_{f,1}(\gamma)$. Note also that
  \begin{equation}
    \label{eq:meas_fix_p5}
    \int_IF_{n,\gamma}(a_I)\,d(\mpush\gamma^{-1}.\nu_\gamma)(t)=F_{n,\gamma}(a_I)\,\nu_\gamma\left(\gamma(\dom\gamma\cap[a_I,b_I])\right);
  \end{equation} as the map $\gamma\mapsto F_{n,\gamma}(a_I)$ is Borel
  because it is the composition of $F_{n,\gamma}$ with the evaluation at the
  point $a_I$, and as the map
  $\gamma\mapsto\nu_\gamma\left(\gamma(\dom\gamma\cap[a_I,b_I])\right)$
  is Borel by condition (4) in the Definition \ref{defn:Alberti_rep}
  of Alberti representations, we conclude that the maps $G_{f,1; k}$
  and $G_{f,1}$ are Borel. We now fix $k\in\natural$ and note that
  also the map
  \begin{equation}
    \label{eq:meas_fix_p6}
    \gamma\mapsto\int_\real F_{n,\gamma}\left(t+\frac{1}{k}\right)\,d(\mpush\gamma^{-1}.\nu_\gamma)(t)
  \end{equation}
  is Borel. We thus conclude that the function $H_{f,1;
    k}:\frags_n(X)\to\real$ defined by:
  \begin{equation}
    \label{eq:meas_fix_p7}
    H_{f,1;k}(\gamma)=\int_\real \frac{F_{n,\gamma}\left(t+\frac{1}{k}\right)-F_{n,\gamma}(t)}{1/k}\,d(\mpush\gamma^{-1}.\nu_\gamma)(t),
  \end{equation}
  is Borel. Note that for $\lebmeas$-a.e.~$t$ we have
  $\lim_{k\to\infty}\frac{F_{n,\gamma}\left(t+\frac{1}{k}\right)-F_{n,\gamma}(t)}{1/k}=F'_{n,\gamma}(t)$,
  which agrees with $(f\circ\gamma)'(t)$ for
  $\mpush\gamma^{-1}.\nu_\gamma$-a.e.~$t$. As
  $H_{f,1}(\gamma)=\lim_{k\to\infty}H_{f,1;k}(\gamma)$, we conclude
  that $H_{f,1}$ is Borel. For a Borel $A\subset X$, the previous
  argument can be applied to the Alberti representation $(P,\nu\mrest
  A)$ of $\mu\mrest A$ to conclude that $H_{f,\chi_A}$ is Borel. As each
  $g\in\bborel X.$ is a pointwise limit of uniformly bounded simple
  functions, we conclude that $H_{f,g}$ is Borel.
\end{proof}
\begin{thm}\label{thm:alb_derivation}
  If $\albrep.=(P,\nu)$ is a $C$-Lipschitz Alberti representation of
  the Radon measure $\mu$, the
  formula
  \begin{equation}\label{eq:derivation_alberti}
    \begin{split}
      \int_X gD_{\albrep.}f\,d\mu&=\int_{\frags(X)}dP(\gamma)\int_{\dom\gamma}
      (f\circ
      \gamma)'(t)g\circ\gamma(t)\,d(\mpush\gamma^{-1}.\nu_\gamma)(t)\\
      &=\int_{\frags(X)}dP(\gamma)\int_X
      g\partial_\gamma f\,d\nu_\gamma\quad(g\in L^1(\mu)\cap\bborel X.),
    \end{split}
  \end{equation}
where
\begin{equation}
    \partial_\gamma f(x)=
    \begin{cases}
      (f\circ\gamma)'(\gamma^{-1}(x))&\text{if $x\in\im\gamma$ and the
        derivative exists}\\
      0&\text{otherwise},
    \end{cases}
\end{equation}
  defines a derivation $D_{\albrep.}\in\wder\mu.$ with
  $\|D_{\albrep.}\|_{\wder\mu.}\le C$.
  \par Denoting by $\Alb(\mu)$ the set of Alberti representations of $\mu$
  which are Lipschitz and letting \begin{equation}
    \Alb_{{\text{\normalfont
          sub}}}(\mu)=\bigcup\left\{\Alb(\mu\mrest S):\text{\normalfont $S\subset X$ Borel}\right\},
  \end{equation} we obtain a map:
  \begin{equation}
    \begin{aligned}
      \Der:\Alb_{\text{\normalfont sub}}(\mu)&\to\wder\mu.\\
      \albrep.&\mapsto D_{\albrep.}.
    \end{aligned}
  \end{equation}
\end{thm}
\begin{proof}
%\textcolor{red}
{We first explain why $D_{\albrep.}$ exists;
  in the first line of~(\ref{eq:derivation_alberti}) we are using the 
  right hand side to define the left hand side. Concretely,
  let $\albrep.=(P,\nu)$ be a \hbox{$C$-Lipschitz} Alberti representation of
  $\mu$; considering a \hbox{$C$-Lipschitz} fragment $\gamma$, from the estimate:
  \begin{equation}
    \left|(f\circ\gamma)'(t)\right|\le \glip f.\metdiff\gamma(t)
    \le C\glip f.,
  \end{equation} 
  we obtain that the absolute value of the right hand side
  of~(\ref{eq:derivation_alberti}) is at most:
  \begin{equation}
   C\glip f.\onenorm g.;
  \end{equation} in particular, fixing $f$ and letting $g$ vary,  the
  right hand side of~(\ref{eq:derivation_alberti}) defines a linear
  functional on $L^1(\mu)$. By the duality between $L^1(\mu)$ and
  $L^\infty(\mu)$ we conclude that $D_{\albrep.}f$ is uniquely determined by
  requiring~(\ref{eq:derivation_alberti}) to hold. Moreover, we conclude
  that the map $D_{\albrep.}:\lipalg X.\to\elleinfty\mu.$ is a linear operator with
  norm bounded by $C$. }
  \par That $D_{\albrep.}$ satisfies the product rule follows from a
  direct computation.
  \par We show that $D_{\albrep.}$ is weak* continuous; let
  $f_n\xrightarrow{\text{w*}} f$ in $\lipalg X.$, and consider a bounded
  Borel function $g\in
  L^1(\mu)$. Now, $\frac{d}{dt}$ is a
  derivation of $\real$ with respect to the Lebesgue measure; as
  $\mpush\gamma^{-1}.\nu_\gamma$ is absolutely continuous with respect
  to the Lebesgue measure, the function
  \begin{equation}
    \frac{d\mpush\gamma^{-1}.\nu_\gamma}{d{\rm Leb}}g
  \end{equation}
  is integrable, and so
  \begin{equation}
    \lim_{n\to\infty}  \int_{\dom\gamma}
    (f_n\circ \gamma)'(t)(g\circ\gamma(t))\,d(\mpush\gamma^{-1}.\nu_\gamma)(t)
    = \int_{\dom\gamma}
    (f\circ \gamma)'(t)(g\circ\gamma(t))\,d(\mpush\gamma^{-1}.\nu_\gamma)(t).
  \end{equation}
  Assume that $\glip f_n.,\glip f.\le L$ and let
  \begin{align}
    h_n(\gamma)&=\int_{\dom\gamma}
    (f_n\circ \gamma)'(t)(g\circ\gamma(t))\,d(\mpush\gamma^{-1}.\nu_\gamma)(t),\\
    h(\gamma)&=\int_{\dom\gamma}
    (f\circ \gamma)'(t)(g\circ\gamma(t))\,d(\mpush\gamma^{-1}.\nu_\gamma)(t),\\
    H(\gamma)&=CL\int_{\dom\gamma}(|g|\circ \gamma(t))\,d(\mpush\gamma^{-1}.\nu_\gamma)(t);
  \end{align}
  note that $h_n\to h$ pointwise, $|h_n|,|h|\le H$. The map $H$ is
  Borel by Definition \ref{defn:Alberti_rep} of Alberti
  representations, and the maps $h_n, h$ are Borel by Lemma
  \ref{lem:meas_fix}; we thus conclude by the Lebesgue's Dominated
  Convergence Theorem that
  \begin{multline}
    \lim_{n\to\infty}\int_{\frags(X)}dP(\gamma)\int_{\dom\gamma}
    (f_n\circ \gamma)'(t)g\circ\gamma(t)\,d(\mpush\gamma^{-1}.\nu_\gamma)(t)\\
    =\int_{\frags(X)}dP(\gamma)\int_{\dom\gamma}
    (f\circ \gamma)'(t)g\circ\gamma(t)\,d(\mpush\gamma^{-1}.\nu_\gamma)(t),
  \end{multline}
  showing weak* continuity.
  \par The definition of $\Der$ is well-posed because $\wder\mu\mrest S.$
  can be canonically identified with $\chi_S\wder\mu.$.
\end{proof}
\par We now describe what happens if $\albrep.$ is in the $f$-direction of
some cone field.
\begin{thm}\label{thm:directional_cone}
  Suppose $\albrep.=(P,\nu)\in\Alb(\mu)$ is in the $f$-direction of
  $\cone(v,\alpha)$; then
  \begin{equation}
    \label{eq:directional_cone}
    D_{\albrep.}f(x)\in\cone\left(v(x),\alpha(x)\right)\quad\text{(for $\mu$-a.e.~$x$).}
  \end{equation}
  In particular, if the Alberti representations $\{\albrep
  i.\}_{i=1}^k$ of $\mu$ are in the $f$-directions of independent
  cone fields $\{\cone(v_i,\alpha_i)\}_{i=1}^k$, the derivations
  $\{D_{\albrep i.}\}_{i=1}^k$ are independent.
\end{thm}
\begin{proof}
  Note that if $U\subset X$ is Borel, $D_{\albrep.\mrest
    U}=\chi_UD_{\albrep.}$; therefore we can assume that $\mu$ is
  finite with $\mu(X)>0$ and $\tan\alpha\in L^1(\mu)$. If
  $\varphi\in\bborel X.$%
\nomenclature[Borel]{$\bborel X.$}{bounded real-valued Borel functions
  defined on $X$}%
  \ is nonnegative with $\int_X\varphi\,d\mu>0$,
  \begin{equation}
    \label{eq:directional_cone1}
    \begin{split}
      \int_X\varphi\langle
      v,D_{\albrep.}f\rangle\tan\alpha\,d\mu&=\int_{\frags(X)}dP(\gamma)\int
      \varphi \langle v, \partial_\gamma f\rangle\tan\alpha\,d\nu_\gamma\\
      &>\int_{\frags(X)}dP(\gamma)\int \varphi\|\pi_v^\perp \partial_\gamma f\|_2\,d\nu_\gamma,
    \end{split}
  \end{equation}
  %\textcolor{red}
{where we used that $\mu(X)>0$, that $\int_X\varphi\,d\mu>0$ and~(\ref{eq:defcone})}. On the
  other hand, for each $\epsi>0$
  there are countably many $(K_j,w_j)$ such that
  \begin{enumerate}
  \item The $K_j\subset X$ are disjoint compact sets and $\mu\left(X\setminus\bigcup_jK_j\right)=0$.
  \item Each $w_j$ is a unit vector field orthogonal to $v$ on $K_j$.
  \item The inequality $\|\pi_v^\perp D_{\albrep.}f\|_2\le \langle
    w_j,D_{\albrep.}f\rangle+\epsi$ holds on $K_j$.
  \end{enumerate}
  Thus
  \begin{equation}
    \label{eq:directional_cone2}
    \begin{split}
      \int_X \varphi \|\pi_v^\perp D_{\albrep.}f\|_2\,d\mu &\le\sum_j\int_{K_j}
      \varphi\langle w_j,D_{\albrep.}f\rangle\,d\mu+\max\varphi\,\epsi\,\mu(X)\\
      &=\sum_j\int_{\frags(X)}dP(\gamma)\int_{K_j}\varphi\langle
      w_j,\partial_\gamma f\rangle\,d\nu_\gamma+\max\varphi\,\epsi\,\mu(X)\\
      &\le\int_{\frags(X)}dP(\gamma)\int_X\varphi \|\pi_v^\perp \partial_\gamma
      f\|_2\,d\nu_\gamma +\max\varphi\,\epsi\,\mu(X);
    \end{split}
  \end{equation}
  letting $\epsi\searrow0$ and combining \eqref{eq:directional_cone1}
  and \eqref{eq:directional_cone2}
  \begin{equation}\label{eq:directional_cone3}
    \int_X \varphi\left(\tan\alpha\langle v,D_{\albrep.}f\rangle-\|\pi_v^\perp D_{\albrep.}f\|_2\right)\,d\mu>0,
  \end{equation} which implies the result.
\end{proof}
We now introduce the notion of \textbf{effective speed} of an Alberti
representation.
\begin{defn}\label{defn:alberti_speed}
  Let $\albrep.=(P,\nu)\in\Alb(\mu)$ be an Alberti representation and
  define the measure
  \begin{equation}\label{eq:alberti_speed}
    \Sigma_{\albrep.}=\int_{\frags(X)} \metdiff\gamma\circ\gamma^{-1}\,\nu_\gamma\,dP(\gamma);
  \end{equation} an \textbf{effective speed of $\albrep.$} is a
  Borel representative $\sigma_{\albrep.}$ of the Radon-Nikodym derivative of
  $\Sigma_{\albrep.}$ with respect to $\mu$.
\end{defn}
We now describe what happens if one knows a lower bound of the
$f$-speed of $\albrep.$.
\begin{thm}\label{thm:directional_speed}
  If $\albrep.\in\Alb(\mu)$ has $f$-speed $\ge\delta$, then
  \begin{equation}
    \label{eq:directional_speed}
    D_{\albrep.}f(x)\ge\delta(x)\sigma_{\albrep.}(x)\quad\text{(for $\mu$-a.e.~$x$).}
  \end{equation}
\end{thm}
\begin{proof}
  Let $\varphi\in L^1(\mu)$ nonnegative; then
  \begin{equation}
    \begin{split}
      \int \varphi D_{\albrep.}f\,d\mu &= \int dP(\gamma) \int
      \varphi \partial_\gamma f\,d\nu_\gamma\\
      &\ge\int dP(\gamma) \int
      \varphi\delta\,d\left(\metdiff\gamma\circ\gamma^{-1}\,\nu_\gamma\right)\\
      &=\int \varphi\delta\sigma_{\albrep.}\,d\mu,
    \end{split}
  \end{equation} from which the result follows.
\end{proof}
As we deal with parametrized fragments, it is useful to know how
Alberti representations are affected by affine reparametrizations of
the fragments.
\begin{lem}\label{lem:alberti_affine_action}
  For $a\in\real\setminus\{0\}$ and $b\in\real$ let
  $\tau_{a,b}:\real\to\real$ denote the homeomorphism
  \begin{equation}
\tau_{a,b} (x)=ax+b;
\end{equation}
then 
\begin{equation}
  \begin{aligned}
    \pullb\tau_{a,b} .:\frags(X) &\to\frags(X)\\
    \gamma&\mapsto\gamma\circ\tau_{a,b},
  \end{aligned}
\end{equation}
is a homeomorphism.\par 
If  $\mu\mrest A$ admits an Alberti representation $\albrep .=(P,\nu)$ 
then \begin{equation}\pullb\tau_{a,b}.\albrep.=\left(\mpush{(\pullb\tau_{a,b}.)}.
  P,\nu\circ(\pullb\tau_{a,b} .)^{-1}\right)
\end{equation}
is an Alberti representation and, moreover, if
\begin{enumerate}
\item The Alberti representation $\albrep.$ is $C$-Lipschitz ($(C,D)$-bi-Lipschitz) then
  $\pullb\tau_{a,b}.\albrep.$ is $C|a|$-Lipschitz
  ($(C|a|,\allowbreak D|a|\-)$-bi-Lipschitz).
  \item If $\albrep.$ has $f$-speed $\ge\delta$ then
    $\pullb\tau_{a,b}.\albrep.$ has $\sgn a\cdot f$-speed $\ge\delta$.
    \item If $\albrep.$ is in the $f$-direction of $\cone(v,\alpha)$,
      then $\pullb\tau_{a,b}.\albrep.$ is in the $f$-direction of
      $\cone(\sgn a\cdot v,\alpha)$.
\end{enumerate}
Finally, if $\albrep.$ is Lipschitz,
  $D_{\pullb\tau_{a,b} .\albrep.}=aD_{\albrep.}$.
\end{lem}
\begin{proof}
  Note that $\pullb\tau_{a,b}.$ is $\max(|a|,\frac{1}{|a|})$-Lipschitz
  and with inverse %\textcolor{red}
{$\pullb\tau_{1/a,-b/a}.$}; if $B$ is Borel, from
  $\gamma\mapsto\nu_\gamma(B)$ being Borel, it follows that
  $\gamma\mapsto\nu({(\pullb\tau_{a,b}.)^{-1}\gamma})(B)$ is Borel.
  Then
\begin{equation}
  \mu\mrest A(B)=\int_{\frags(X)}\nu(\gamma)(B)\,dP(\gamma)
  =\int_{\frags(X)}\nu({(\pullb\tau_{a,b}.)^{-1}\gamma})(B)\,d\mpush{(\pullb\tau_{a,b}.)}.P(\gamma),
\end{equation} which shows that $\pullb\tau_{a,b}.\albrep.$ is an
Alberti representation of $\mu$. Points (1)--(3) %\textcolor{red}
{in
  the statement of this lemma} follow observing that
$\mpush{(\pullb\tau_{a,b}.)}.P$ is supported on
$\pullb\tau_{a,b}.(\spt P)$.
Finally,
\begin{equation}
\begin{split}
  \int_X D_{\pullb\tau_{a,b} .\albrep.}f\,g\,d\mu &= \int_{\frags(X)}
  d\mpush{(\pullb\tau_{a,b} .)}.P(\gamma) \int_{\dom\gamma}(f\circ\gamma)'(t)
  \,(g\circ\gamma)(t)\\ &\times d\left(\mpush\gamma^{-1}.\nu\left({(\pullb\tau_{a,b} .)^{-1}\gamma}\right)\right)(t)\\ &=\int_{\frags(X)}
  dP(\tilde\gamma) \int_{\dom\gamma}(f\circ\tilde\gamma\circ\tau_{a,b} )'(t)
  \,(g\circ\tilde\gamma\circ\tau_{a,b} )(t)\\ &\times d\left(\mpush(\tilde\gamma\circ\tau_{a,b} )^{-1}.\nu({
    \tilde\gamma})\right)(t),
\end{split}
\end{equation}
where $\gamma=\pullb\tau_{a,b}.\tilde\gamma$; then
\begin{equation}
  \begin{split}
    \int D_{\pullb\tau_{a,b} .\albrep.}f\,g\,d\mu &= \int_{\frags(X)}
    dP(\tilde\gamma)
    \int_{\dom\gamma}a(f\circ\tilde\gamma)'\circ\tau_{a,b} (t)
    \,(g\circ\tilde\gamma\circ\tau_{a,b} )(t)\\ &\times\,d\left(\mpush\tau_{a,b} ^{-1}.\mpush\tilde\gamma^{-1}.\nu({
      \tilde\gamma})\right)(t)\\ &=\int_{\frags(X)} dP(\tilde\gamma)
    \int_{\dom\tilde\gamma}a(f\circ\tilde\gamma)'(t)
    \,(g\circ\tilde\gamma)(t)d\left(\mpush\tilde\gamma^{-1}.\nu({\tilde\gamma})\right)(t)\\ &=\int aD_{\albrep.}f\,g\,d\mu.
  \end{split}
\end{equation}
\end{proof}
To illustrate the lack of injectivity of $\Der$ we provide a useful
reparametrization result which allows to use fragments with domain
contained in a prescribed interval. Recall that given a nondegenerate compact
interval $I\subset\real$, we denote by $\frags(I; X)$ the subset of
fragments $\gamma$ with $\dom\gamma\subset I$.
\begin{lem}\label{lem:domain_reduction}
  Let $\albrep.\in\Alb(\mu\mrest S)$ where $S$ is a Borel subset of
  $X$, and let $I$ be a nonempty compact interval; then there is $
  \albrep.'=(P',\nu')\in\Alb(\mu\mrest S)$ with
  \begin{equation}P'\left(\frags(X)\setminus\frags(I; X)\right)=0
  \end{equation}
  and such that
  \begin{enumerate}
  \item We have the identity $D_{\albrep.'}=D_{\albrep.}$.
    \item If $\albrep.$ satisfies a set of conditions \albcond, so
      does $\albrep.'$.
  \end{enumerate}
\end{lem}
\begin{proof}
  We can assume that $\albrep.=(P,\nu)\in\Alb(\mu)$ and that $X$ is
  compact. %\textcolor{red}
{The overall strategy of the proof is to decompose the
    fragments used in the Alberti representation $\albrep.$ into
    smaller pieces whose domains can be translated to be subsets of
    $I$.}
  \par Recall that $\haus(\real)$ denotes the set of nonempty compact subsets of
  $\real$ with the Vietoris topology. The maps:
  \begin{equation}
    \begin{aligned}
      \max,\min:\haus(\real)&\to\real\\
      K&\mapsto\max_{x\in K} x,\min_{x\in K} x
    \end{aligned}
  \end{equation}
are continuous. Therefore, the map
\begin{equation}
  \begin{aligned}
  \Delta_1:\haus(\real)&\to\haus(\real)\\
  K&\mapsto K\cap[\min(K), \min(K)+\lebmeas(I)]
\end{aligned}
\end{equation}
is continuous. For $n>1$ let $O_n$ denote the open set
\begin{equation}
O_n=\left\{K\in\haus(\real): \max(K)-\min(K)>(n-1)\lebmeas(I)\right\},
\end{equation}
and $\Delta_n$ the continuous map
\begin{equation}
  \begin{aligned}
    \Delta_n:O_n&\to\haus(\real)\\
    K&\mapsto\Delta_1\left(\overline{K\setminus\bigcup_{k=1}^{n-1}\Delta_k(K)}\right).
  \end{aligned}
\end{equation}
Note that the set%\textcolor{red}
{
\begin{equation}
  F_i=\left\{\gamma\in\frags(X):\text{$\dom\gamma\in O_i$ and $\lebmeas\left(\Delta_i(\dom\gamma)\right)>0$}\right\}
\end{equation}
is Borel and the map $R_i:F_i\to\frags(I; X)$ 
\begin{equation}
  \gamma\mapsto\pullb\tau_{1,-\min(I)+\min(\Delta_i(\dom\gamma))}.\gamma|_{\Delta_i(\dom\gamma)}
\end{equation}
is Borel. }
\par The following discussion applies only for those $i$ for
which $P(F_i)>0$.
 We 
apply the %\textcolor{red}
{Disintegration Theorem \cite[452O]{fremlin4} (and also 
\cite[452G(c)]{fremlin4})} for
\begin{equation}
  R_i:\left(F_i,\frac{P\mrest F_i}{P(F_i)}\right)\to\left(\frags(I; X),\underbrace{\frac{\mpush R_i.P\mrest F_i}{P(F_i)}}_{P_i}\right),
\end{equation} which is, in the terminology of \cite{fremlin4}, an
\textbf{inverse-measure-preserving} function. By \cite[434K(b)]{fremlin4}
$\haus(\real\times X)$, being Polish, is a \textbf{Radon measure space} and, as
$F_i$ is a Borel subset of $\haus(\real\times X)$, it is also a Radon
measure space by \cite[434F(c)]{fremlin4}. On the other hand,
$\left(\frags(I; X), P_i\right)$ is \textbf{strictly
localizable} in the terminology of \cite{fremlin4} because the measure
$P_i$ is finite. The Disintegration Theorem yields Radon probability
measures $\{\pi_{\gamma}\}_{\gamma\in \frags(I; X)}$ on $\frags(X)$ such that
\begin{align}
  \frac{P\mrest F_i}{P(F_i)}&=\int_{\frags(I; X)}\pi_{\gamma}\,dP_i(\gamma)\\
  \pi_\gamma\left(R_i^{-1}\left(\{\gamma\}\right)\right)&=1\quad\text{(for $P_i$-a.e.~$\gamma$)}.
\end{align}
Consider the weakly measurable maps $\tilde\nu_i:\frags(X)\to M(X)$
\begin{equation}
  \tilde\nu_i(\gamma)=
  \begin{cases}
    0&\text{if $\gamma\not\in F_i$,}\\
    {P(F_i)}\nu_\gamma\chi_{\gamma\left(\Delta_i(\dom\gamma)\right)}&\text{otherwise;}
  \end{cases}
\end{equation}
and the measures
\begin{equation}
  \mu_i=\frac{1}{P(F_i)}\int_{F_i}\tilde\nu_i(\gamma)\,dP(\gamma);
\end{equation}
if we let $\nu_i:\frags(I; X)\to M(X)$ be given by
\begin{equation}
  \nu_i(\gamma)=\int_{\frags(X)}\tilde\nu_i(\tilde\gamma)\,d\pi_\gamma(\tilde\gamma),
\end{equation} $\albrep i.=(P_i,\nu_i)$ is an Alberti representation
of $\mu_i$ supported on the closure of  $\im R_i$. Then, if $\albrep.$
satisfies \albcond, so does $\albrep i.$. The derivation
$D_i\in\wder\mu_i.$ given by
\begin{equation}
  \int
  gD_if\,d\mu_i=\frac{1}{P(F_i)}\int_{F_i}dP(\gamma)\int g\partial_\gamma f\,\tilde\nu_i(\gamma)
\end{equation} is naturally identified with an element of $\wder\mu.$
because $\mu_i\ll\mu$; in particular, $D_{\albrep.}=\sum_iD_i$. The
calculation
\begin{multline}
\frac{1}{P(F_i)}  \int_{F_i}dP(\tilde\gamma)\int_{\dom\tilde\gamma}\left(f\circ\tilde\gamma\right)'(t)\,
  \left(g\circ\tilde\gamma\right)(t)\,d\left(\mpush\tilde\gamma^{-1}.\tilde\nu_i(\tilde\gamma)\right)(t)\\
=
\int_{\frags(I; X)}dP_i(\gamma)\int_{\frags(X)}d\pi_\gamma(\tilde\gamma)
\int_{\dom\tilde\gamma}\left(f\circ\tilde\gamma\right)'(t)\,
\left(g\circ\tilde\gamma\right)(t)\,d\left(\mpush\tilde\gamma^{-1}.\tilde\nu_i(\tilde\gamma)\right)(t)\\
=
\int_{\frags(I; X)}dP_i(\gamma)\int_{\frags(X)}d\pi_\gamma(\tilde\gamma)
\int_{\dom\gamma}\left(f\circ\gamma\right)'(t)
\left(g\circ\gamma\right)(t)\,d\left(\mpush\gamma^{-1}.\tilde\nu_i(\tilde\gamma)\right)(t)\\
=
\int_{\frags(I; X)}dP_i(\gamma)
\int_{\dom\gamma}\left(f\circ\gamma\right)'(t)
\left(g\circ\gamma\right)(t)\,d\left(\mpush\gamma^{-1}.\nu_i(\gamma)\right)(t)\\
=\int gD_{\albrep i.}f\,d\mu_i
\end{multline} shows that $D_i=D_{\albrep i.}$. Let $P'$ the
probability measure on $\frags(I; X)$ given by
\begin{equation}
  P'=\sum_i2^{-i}P_i
\end{equation} and $\varphi_i$ a Borel representative of the
Radon-Nikodym derivative of $P_i$ with respect to $P'$; let 
\begin{equation}
  \begin{aligned}
    \nu':\frags(I; X)&\to M(X)\\
    \gamma&\mapsto\sum_i\nu_i(\gamma)\varphi_i(\gamma);
  \end{aligned}
\end{equation}
then $\albrep.'=(P',\nu')$ is an Alberti representation of
$\sum_i\mu_i=\mu$ and, if $\albrep.$ satisfies \albcond, so does
$\albrep.'$. Moreover, 
\begin{equation}
  D_{\albrep.'}=\sum_iD_{\albrep i.}=\sum_iD_i=D_{\albrep.}.
\end{equation}
\end{proof}
\subsection{From Derivations to Alberti
  representations}\label{subsec:derivations_to_alberti}
%\textcolor{red}
{The goal of this subsection is to show how information about
derivations translates into information about Alberti
representations. The first results, Theorem
\ref{derivation_alberti} and Corollaries \ref{der-alb} and
\ref{cor:mu_arb_cone}, give conditions under which the existence of
independent derivations translates into the existence of Alberti representations.
Theorem \ref{thm:weak*density} proves the weak* density of the
derivations associated to Alberti representations in the space of all
derivations. Finally, Theorem
\ref{thm:fin_gen_surjectivity} shows that if the module of derivations
is finitely generated any derivation must arise from an Alberti
representation.} Recall that we deal with 
\textbf{separable} metric spaces.
\begin{thm}\label{derivation_alberti}
  Let $X$ be a metric space and $\mu$ a Radon measure on $X$. Consider
  a Borel set $V\subset X$, derivations
  $\{D_1,\ldots,D_k\}\subset\wder\mu.$ and a Lipschitz function
  ${g\colon X\to\real^k}$ such that $D_ig_j=\delta_{i,j}\chi_V$. Then for each
  $\epsi>0$, unit vector $w\in{\mathbb S}^{k-1}$, angle $\alpha\in(0,\pi/2)$ and 
 speed parameter $\sigma\in(0,1)$, the measure $\mu\mrest V$ admits a
  ${(1,1+\epsi)}$\nobreakdash-\hspace{0pt}bi-Lipschitz Alberti representation in the
  $g$-direction of $\cone(w,\alpha)$ with 
  \begin{equation}\label{eq:speed_almost_optimal}
    {\text{$\langle w,g\rangle$\nobreakdash-\hspace{0pt}speed}}
    \ge\frac{\sigma}{\locnorm D_w,{\wder\mu\mrest V.}.+(1-\sigma)},
  \end{equation}
%\textcolor{red}
{where $D_w=\sum_{i=1}^kw_iD_i$.}
\end{thm}
The proof of Theorem \ref{derivation_alberti} relies on an
approximation scheme for Lipschitz functions, Theorem \ref{onedimapprox_multi}. We state the
relevant definitions and the approximation scheme here and defer the
proof to Subsection \ref{subsec:an-appr-scheme}. We define the
following classes of fragments:
\begin{defn}\label{def:classes_frags}
  For $\delta>0$ and $f:X\to\real$ Lipschitz we define:
  \begin{equation}
    \frags(X,f,\delta)=\left\{\gamma\in\frags(X):
    \text{$(f\circ\gamma)'(t)\ge\delta\metdiff\gamma(t)$ for $\lebmeas$-a.e.~$t\in\dom\gamma$}\right\};
  \end{equation}
  For $\delta>0$, $f:X\to\real^q$ Lipschitz,
  $w\in{\mathbb{S}}^{q-1}$ and $\alpha\in(0,\pi/2)$, we
  define:
  \begin{multline}
    \frags(X,f,\delta,w,\alpha)=\bigl\{\gamma\in\frags(X,\langle w,f\rangle,\delta):\text{$(f\circ\gamma)'(t)\in\cone(w,\alpha)$}
         \\ \text{for $\lebmeas$-a.e.~$t\in\dom\gamma$}\bigr\}.
       \end{multline}
     \end{defn}
\begin{defn}\label{def:locally_lipschitz}
  Let $f:X\to\real$ be a Lipschitz function, $S\subset X$ Borel and
  $\mu$ a Radon measure on $X$. We say that $f$ is \textbf{locally
  $\delta$-Lipschitz} on $S$ if for each $x\in S$ there is an $r>0$ such
  that the restriction $f|\ball x,r.$ is $\delta$-Lipschitz. Finally,
  we say that $f$ is \textbf{$\mu$-a.e.~locally $\delta$-Lipschitz} on $S$ if
  for $\mu$-a.e.~$x\in S$ there is an $r>0$ such that the restriction
  $f|\ball x,r.$ is $\delta$-Lipschitz.
\end{defn}
\begin{thm}\label{onedimapprox_multi}
  Let $X$ be a compact metric space, $f\colon X\to\real^q$ $L$-Lipschitz and
  $S\subset X$ compact. Let $\mu$ be a Radon measure on $X$. Assume
%\textcolor{red}
{  that the compact set
  $S$ is $\frags(X,f,\delta,w,\alpha)$-null.}
Denoting by $d_{\delta,\alpha}$, $\tilde d_{\delta,\alpha}$ the distances
  \begin{align}
    \label{eq:onedimapprox_multi_s1}
    d_{\delta,\alpha}(x,y)&=\delta\dist
    x,y.+\cot\alpha\|\pi_w^\perp
    f(x)-\pi_w^\perp f(y)\|_2,\\
    \label{eq:onedimapprox_multi_s2}
    \tilde
    d_{\delta,\alpha}&=\max(L\dist\cdot,\cdot.,d_{\delta,\alpha}),
\end{align} there are $(1,\tilde d_{\delta,\alpha})$-Lipschitz functions
  $g_k\xrightarrow{\text{w*}} \langle w,f\rangle$
  with $g_k$ $\mu$-a.e.~locally
  ${(1,d_{\delta,\alpha})}$\nobreakdash-\hspace{0pt}Lipschitz on $S$. 
\end{thm}
The proof of Theorem \ref{derivation_alberti} relies on the following
technical Lemmas \ref{lem:loc_estimate_dist} and
\ref{lem:loc_estimate_fnull}.
\begin{lem}\label{lem:loc_estimate_dist}
  Let $X$ be a separable metric space and $\mu$ a Radon measure on
  $X$. Consider a Borel set $S\subset X$, a Lipschitz function
  $f\colon X\to\real$, a derivation $D\in\wder\mu.$, and a (pseudo)-distance
  function $d'$ on $X$ satisfying $d'\le Cd$, where $d$ denotes the
  metric on $X$. Assume that $f|S$, regarded as a function defined on
  the metric space $(S,d)$, is $\mu\mrest S$-a.e.~locally
  $(1,d')$-Lipschitz and that, for some $c\ge0$,
  \begin{equation}
    \chi_S\left|Dd'(\cdot,p)\right|\le c\quad(\forall p\in S);
  \end{equation}
then \begin{equation}\label{eq:loc_estimate_dist}\chi_S|Df|\le c\quad(\text{$\mu$-a.e.})
\end{equation}
\end{lem}
\begin{proof}
As $\mu$ is Radon, it suffices to show \eqref{eq:loc_estimate_dist}
when $S$ is replaced by a compact subset $K$.
  Fix a ball $B\subset X$ centred on $K$ 
  such that  $f$ is 
  $(1,d')$-Lipschitz in $\bar B\cap K$. The idea of the proof is to
  replace $f$ by an approximation in terms of distance
  functions. Specifically, let  
$\mathscr{N}_\eta$ a finite $\eta$-net in the compact set
$\bar B\cap K$ and $f_\eta$ the McShane extension of $f|\mathscr{N}$ to
$X$:
\begin{equation}
  f_\eta(x)=\min_{p\in\mathscr{N_\eta}}\left(f(p)+d'(x,p)\right);
\end{equation}
note that $f_\eta$ is $(C,d)$-Lipschitz. Let 
\begin{equation}
  \tilde f_\eta(x)=\min\left(\max\left(f_\eta(x),-\sup_{\bar B}|f_\eta|\right),\sup_{\bar B}|f_\eta|\right),
\end{equation}
so that $\tilde f_\eta\in\lipalg X.$ and $\tilde f_\eta=f_\eta$ on
$\bar B$. Choose a sequence $\eta_n\searrow0$.  %\textcolor{red}
{Recall from
Subsection~\ref{subsec:der_modules} that as $X$ is separable,
the weak* topology on
$\lipalg X.$ is metrizable on bounded subsets of $\lipalg X.$; by
Banach-Alaoglu, the sequence $\tilde f_{\eta_n}$ has a convergent subsequence
$\tilde f_{\eta_n}\xrightarrow{\text{w*}} \tilde f$ in $\lipalg
X.$} (this argument is a version of Ascoli-Arzel\`a in
  disguise). Note that $\tilde f = f$ on $\bar B\cap K$.  For a fixed
$\eta$ there are finitely many closed sets $C_p\subset\bar B\cap K$
($p\in\mathscr{N}$) which cover $\bar B\cap K$ and such that, for each $
x\in C_p$,
\begin{equation}
  f_\eta(x)=f(p)+d'(x,p);
\end{equation}
this implies, by locality of derivations (Lemma~\ref{lem:locality_derivations}),
\begin{equation}
  \chi_{C_p\cap K}|D\tilde f_\eta|=\chi_{C_p\cap K}\left|Dd'(\cdot,p)\right|\le c;
\end{equation}
thus $\left\|\chi_K D\tilde f_\eta\right\|_{L^\infty(\mu\mrest\bar
  B)}\le c$. Using
lower semicontinuity of the norm under weak* convergence and Lemma \ref{lem:locality_derivations}, it follows
that \begin{equation}\left\|\chi_K D\tilde f\right\|_{L^\infty(\mu\mrest\bar
  B)}=\left\|\chi_K Df\right\|_{L^\infty(\mu\mrest\bar
  B)}\le c
\end{equation}
which implies \eqref{eq:loc_estimate_dist}.
\end{proof}
\begin{lem}\label{lem:loc_estimate_fnull}
  Let $X$ be a separable metric space and $\mu$ a Radon measure on
  $X$. Consider a Lipschitz function $g\colon X\to\real^k$ , a derivation
  $D\in\wder\mu.$ and a Borel set $S\subset X$. Assume that $S$ is
  $\frags(X,\allowbreak g,\allowbreak\delta,\allowbreak
  w,\alpha)$-null, where $w\in{\mathbb S}^{k-1}$ and $\alpha\in(0,\pi/2)$,
  and that, for each $u\in {\mathbb S}^{k-1}$ orthogonal to $w$, one
  has
  \begin{equation}\chi_S\left|D\langle u,g\rangle\right|\le\epsi\locnorm
  D,{\wder\mu.}.;
\end{equation}
then one has
  \begin{equation}
    \label{eq:loc_estimate_fnull1}
    \chi_S\left|D\langle w,g\rangle\right|\le\left(\delta+(k-1)\epsi\cot\alpha\right)\locnorm D,{\wder\mu.}..
  \end{equation}
\par In particular, if the derivations $\{D_1,\ldots,D_k\}\subset\wder\mu.$
satisfy
  $D_ig_j=\delta_{i,j}$ $\mu$-a.e., letting 
  $D_w=\sum_{i=1}^kw_iD_i$, one has
  \begin{equation}\label{eq:loc_estimate_fnull2}
    \chi_S\left|D_w\langle
    w,g\rangle\right|\le\delta\locnorm D_w,{\wder\mu\mrest S.}..
  \end{equation}
\end{lem}
\begin{proof}
  The second part, \eqref{eq:loc_estimate_fnull2}, follows from the
  first, \eqref{eq:loc_estimate_fnull1}, as $D_w$ annihilates $\langle
  u,g\rangle$ for $u$ orthogonal to $w$. As $\mu$ is Radon, it suffices
  to show \eqref{eq:loc_estimate_fnull1} for $S$ replaced by a compact
  subset $K$. 
  \par Consider the metric space $(K,d)$ and choose an
  orthonormal basis
  $\mathscr{B}$ of the hyperplane orthogonal to
  $w$; let
  \begin{equation}
    \begin{aligned}
      d'(x,y)&=\delta\dist x,y.+\cot\alpha\sum_{u\in\mathscr{B}}|\langle
      u,g(x)-g(y)\rangle|,\\
      \tilde d'(x,y)&=\max(\glip g.d(x,y),d'(x,y)).
    \end{aligned}
  \end{equation}
  As $K$ is $\frags(K,g,\delta,w,\alpha)$-null, by Theorem
  \ref{onedimapprox_multi} there are functions $f_n$
  \begin{enumerate}
  \item Which are $(1,\tilde d')$-Lipschitz and $\mu\mrest K$-a.e.~locally
    $(1,d')$-Lipschitz on $K$.
  \item%\textcolor{red}
{ Which converge weak* to $f=\langle w,g\rangle$ in $\lipalg K.$.}
  \end{enumerate}
  Reasoning as in the proof of Lemma \ref{lem:loc_estimate_dist} (see
  the construction of the functions $\tilde f_\eta$ and $\tilde f$), we can
  take extensions $\tilde f_n\in\lipalg X.$ of the $f_n$ and, after passing to a
  subsequence, arrange that $\tilde
  f_n\xrightarrow{\text{w*}} \tilde f$ where $\tilde f = f$ on $K$. %\textcolor{red}
{Let $p\in K$ and $q\in X$;
 on the open ball $B(q,1)$ one has $|d(\cdot,p)-d(p,q)|<1$; then as $d(\cdot,
  p)$ is $1$-Lipschitz and using locality of derivations
  (Lemma~\ref{lem:locality_derivations}), we have that on $B(q,1)$:
  \begin{equation}
    \label{eq:D_dist1}
    \begin{split}
      |Dd(\cdot,p)|&=|D(d(\cdot,p)-d(p,q))|\\
                   &\le\locnorm D,{\wder\mu.}.\cdot\max\left(\sup_{q'\in
                       B(q,1)}|d(q',p)-d(p,q)|,1\right)\\
                   &\le\locnorm D,{\wder\mu.}.;
    \end{split}
  \end{equation}
  as $q$ can be any point of $X$, and as $X$ is separable, we conclude
  that~(\ref{eq:D_dist1}) holds $\mu$-a.e.}
  On the other hand, by
  \cite[Lem.~7.2.2]{weaver_book99}\footnote{This is a consequence of
    weak* continuity and the classical approximation of functions in
    $\lipalg [0,1].$ by polynomials}
  \begin{equation}
    \label{eq:D_dist2}
    \left|\,D\left|\left\langle
          u,g(\cdot)-g(p)\right\rangle\right|\,\right|\le \left|D\left\langle
        u,g(\cdot)-g(p)\right\rangle\right|;
  \end{equation}
  but by hypothesis,
  \begin{equation}
    \label{eq:D_dist3}
    \chi_K\left|D\left\langle
        u,g(\cdot)-g(p)\right\rangle\right|\le\epsi\locnorm D,{\wder\mu.}.;
  \end{equation}
  putting together \eqref{eq:D_dist1}, \eqref{eq:D_dist2} and
  \eqref{eq:D_dist3}:
  \begin{equation}
    \chi_K\left|Dd'(\cdot,p)\right|\le\left(\delta+(k-1)\epsi\cot\alpha\right)\locnorm D,{\wder\mu.}..
  \end{equation}
  By Lemma \ref{lem:loc_estimate_dist} applied to the $\tilde f_n$,
  \begin{equation}
    \chi_K\left|D\tilde f_n\right|\le\left(\delta+(k-1)\epsi\cot\alpha\right)\locnorm D,{\wder\mu.}.;
  \end{equation}
  by lower semicontinuity of the norm under weak* convergence,
  \begin{equation}
    \chi_K\left|D\tilde f\right|\le\left(\delta+(k-1)\epsi\cot\alpha\right)\locnorm D,{\wder\mu.}.;
  \end{equation}
  finally by locality of derivations (Lemma~\ref{lem:locality_derivations}),
  \begin{equation}
    \chi_K\left|D f\right|\le\left(\delta+(k-1)\epsi\cot\alpha\right)\locnorm D,{\wder\mu.}..
  \end{equation}
\end{proof}
\begin{proof}[Proof of Theorem \ref{derivation_alberti}]
  By the gluing principle for Alberti representations, Theorem~\ref{alb_glue}, and by Lusin's Theorem we can replace $V$ by
  a compact subset $K$ on which, for some $\eta_0,\eta_1>0$,
  \begin{align}
    \|\chi_KD_w\|_{\wder\mu.}&\le\inf_{x\in K}\locnorm
    D_w,{\wder\mu.}.(x)+\eta_0,\\
    \sup_{x\in K}\frac{\sigma}{\locnorm
    D_w,{\wder\mu.}.(x)+(1-\sigma)}&\le\frac{1}{\inf_{x\in K}\locnorm
    D_w,{\wder\mu.}.(x)+\eta_0}-\eta_1;
  \end{align}
we let 
$\delta_{\eta_1}=\frac{1}{\|\chi_KD_w\|_{\wder\mu.}}-\eta_1$ and
$\alpha'\in(0,\alpha)$. We want to show that if a Borel set $S\subset K$
is $\frags(K,g,\delta_{\eta_1},w,\alpha')$-null, then it is
$\mu$-null. Taking $X=K$ in Lemma~\ref{lem:loc_estimate_fnull}, we get
\begin{equation}
\begin{split}
  \chi_S=\chi_SD_w\langle w,g\rangle&\le\delta_{\eta_1}\locnorm
    D_w,{\wder\mu.}.\\
    &\le\delta_{\eta_1}\|\chi_KD_w\|_{\wder\mu.}\\
    &\le1-\eta_1\|\chi_KD_w\|_{\wder\mu.}\\
    &<1,
\end{split}
\end{equation}
which implies $\mu(S)=0$. %\textcolor{red}
{Now by
  Theorem~\ref{alberti_rep_prod}, as
  $\frags(K,g,\delta_{\eta_1},w,\alpha')$-null sets are $\mu$-null, the
measure $\mu\mrest K$ admits an Alberti representation in the
$g$-direction of $\cone(v,\alpha)$ with
\begin{equation}
\begin{split}
\text{$\langle
  w,g\rangle$-speed}&\ge\frac{1}{\|\chi_KD_w\|_{\wder\mu.}}-\eta_1\\
&\ge\frac{1}{\inf_{x\in K}\locnorm
    D_w,{\wder\mu.}.(x)+\eta_0}-\eta_1\\
&\ge\sup_{x\in K}\frac{\sigma}{\locnorm
    D_w,{\wder\mu.}.(x)+(1-\sigma)}.
\end{split}
\end{equation}}
\end{proof}
\begin{cor}\label{der-alb}
  Let $V$ be a Borel set and assume that $\wder\mu\mrest V.$ contains
  $k$-inde\-pen\-dent de\-ri\-va\-tions. Then there is a Borel partition
  $V=\bigcup_{\alpha}V_\alpha$ and, for each $\alpha$, there is a
  $1$-Lipschitz function $f_\alpha:V\to\real^k$ such that,
  for each $\epsi>0$ and for all Borel maps $w:X\to{\mathbb S}^{k-1}$ and
  $\theta:X\to(0,\pi/2)$, the measure $\mu$ admits a
  $(1,1+\epsi)$-bi-Lipschitz Alberti representation $\albrep.$ with
  $\albrep.\mrest V_\alpha$ in the $f_\alpha$-direction of the cone field
  $\cone(w,\theta)$. 
  \par In particular, the measure $\mu\mrest V_\alpha$ admits
  Alberti representations in the $f_\alpha$\nobreakdash-\hspace{0pt}directions of independent
  cone fields $\cone_1,\ldots,\cone_k$.
\end{cor} 
\begin{proof}%[Proof of Corollary \ref{der-alb}]
Applying Corollary \ref{cor:pseudoduality} one finds disjoint Borel
sets $\{V_\alpha\}$ with $V=\bigcup_\alpha V_\alpha$ and such that, for each
$\alpha$, there are $1$-Lipschitz functions $f_\alpha:X\to\real^k$ and
derivations $\{D_{\alpha,1},\ldots,D_{\alpha,k}\}\subset\wder\mu.$
with
\begin{equation}
  D_{\alpha,i}f_{\alpha,j}=\chi_{V_\alpha}\delta_{i,j}.
\end{equation}
In the case in which $w$ and $\theta$ are constant, the result follows
applying Theorem~\ref{derivation_alberti}. If $w$ and $\theta$ are not
constant, one can find disjoint compact sets
$\{C_{\alpha,\beta}\}_{\beta}$:
\begin{enumerate}
\item Each $C_{\alpha,\beta}$ is a subset of $V_\alpha$.
\item We have the identity $\mu\left(V_\alpha\setminus\bigcup_\beta
  C_{\alpha,\beta}\right)=0$.
\item For each $\beta$ there are a unit vector $w_\beta\in{\mathbb
    S^{k-1}}$ and an angle $\theta_\beta\in(0,\pi/2)$ with
  $\cone(w_\beta,\theta_\beta)\subset \cone(w(x),\theta(x))$ for each $x\in
  C_{\alpha,\beta}$.
\end{enumerate}
One then applies Theorem \ref{alb_glue} to glue together the Alberti
representations of the measures $\mu\mrest C_{\alpha,\beta}$ in the
$f_\alpha$-directions of the cone fields $\cone(w_\beta,\theta_\beta)$.
\end{proof}
\begin{cor}\label{cor:mu_arb_cone}
  Let $f:X\to\real^k$ be Lipschitz and $\mu$ a Radon measure on $X$
  admitting Alberti representations $\albrep 1.,\ldots,\albrep k.$ in
  the $f$-directions of independent cone fields
  $\cone_1,\ldots,\cone_k$. Then, for each $\epsi>0$ and for all Borel
  maps $w:X\to{\mathbb S}^{k-1}$ and $\theta:X\to(0,\pi/2)$, the
  measure $\mu$ admits a $(1,1+\epsi)$-bi-Lipschitz Alberti
  representation in the $f$-direction of $\cone(w,\theta)$.
\end{cor}
\begin{proof}%[Proof of Corollary \ref{cor:mu_arb_cone}]
  By Theorem \ref{alberti_rep_prod} it is possible to assume that the
  Alberti representations are bi-Lipschitz. According to Theorem
  \ref{thm:directional_cone} the corresponding derivations $D_{\albrep
    i.}$ are independent.  By the gluing principle for Alberti
  representations, Theorem~\ref{alb_glue}, it suffices to prove the
  statement for $w$ and $\theta$ constant and to show that there is a
  Borel partition $\spt\mu=\bigcup_\alpha U_\alpha$ with each
  $\mu\mrest U_\alpha$ admitting an Alberti representation in the
  $f$-direction of $\cone(w,\theta)$. As the cone fields $\cone_i$ are
  independent, the matrix $\left(D_{\albrep i.}f_j\right)_{i,j=1}^k$
  is invertible $\mu$-a.e. This allows to pass to a Borel partition
  $\spt\mu=\bigcup_\alpha U_\alpha$ such that, for each $U_\alpha$,
  there are derivations $D_{\alpha,i}$ with
  $D_{\alpha,i}f_j=\delta_i^j$ \hbox{$\mu\mrest U_\alpha$-a.e.} This
  partition is necessary because the inverse of $\left(D_{\albrep
      i.}f_j\right)_{i,j=1}^k$ does not need to have its entries in
  $L^\infty(\mu)$. However, there are disjoint Borel sets $U_\alpha$
  with $\mu(\spt\mu\setminus\bigcup_\alpha U_\alpha)=0$ and such that on
  each $U_\alpha$ the inverse of $\left(D_{\albrep
      i.}f_j\right)_{i,j=1}^k$ has its entries in $L^\infty(\mu)$.  In
  fact, the entries of $\left(D_{\albrep i.}f_j\right)_{i,j=1}^k$ are
  in $L^\infty(\mu)$ and one needs a lower bound on its determinant on
  each $U_\alpha$, for example letting
  \begin{equation}
    U_n=\left\{x:\left|\det\left(D_{\albrep
      i.}f_j\right)_{i,j=1}^k\right|\in\left[
      \frac{1}{n+1},\frac{1}{n}\right)\right\}.
  \end{equation}
 The result now follows by applying
Theorem \ref{derivation_alberti}.
\end{proof}
We now present two technical {lemmas} that are needed to prove the
following two theorems:
\begin{thm}\label{thm:weak*density}
  The set $\Der(\Alb_{{\rm sub}}(\mu))$ is weak* dense in $\wder\mu.$.
\end{thm}
\begin{thm}\label{thm:fin_gen_surjectivity}
  Suppose that $\wder\mu.$ is finitely generated and $D\in\wder\mu.$; let
   \begin{equation}
     S=\left\{x\in X:\locnorm D,{\wder\mu.}.>0\right\};
   \end{equation}
   then $D\in\Der(\Alb(\mu\mrest S))$.
\end{thm}
In the first {lemma} we show how to produce from an Alberti
representation $\albrep.$ a new one $\albrep.'$ such that the
associated derivation satisfies $D_{\albrep.'}=\chi_U D_{\albrep.}$,
where $\chi_U$ is a characteristic function.
%\textcolor{red}
{\begin{lem}\label{lem:derivation_idempotent}
  Suppose $\albrep.\in\Alb(\mu\mrest S)$ and $U\subset X$ Borel. Then
  there is $\albrep.'\in\Alb(\mu\mrest S)$ with
  $D_{\albrep.'}=\chi_UD_{\albrep.}$. Moreover,
  \begin{enumerate}
  \item If $\albrep.$ is $C$-Lipschitz ($(C,D)$-bi-Lipschitz), so is $\albrep.'$.
    \item If $\albrep.$ has $f$-speed $\ge\delta$ or is in the
      $f$-direction of $\cone(w,\alpha)$, so does $\albrep.'\mrest U$.
  \end{enumerate}
\end{lem}}
\begin{proof}
  Without loss of generality we can assume that $S=X$, and so that
  $\albrep.$ is an Alberti representation in $\Alb(\mu)$ of the form
  $\albrep.=(P,\nu)$. Using
  Lemma \ref{lem:domain_reduction}, we find an Alberti representation
  $\albrep 1.=(P_1,\nu_1)\in\Alb(\mu\mrest U^c)$ with 
 \begin{equation}
   P_1\left(\frags(X,[0,1])\right)=1
 \end{equation}
and $D_{\albrep
    1.}=\chi_{U^c}D_{\albrep.}$. Moreover, if $\albrep.$ satisfies a
  set of conditions \albcond, so does $\albrep 1.$. Similarly,
  combinining Lemmas \ref{lem:domain_reduction} and \ref{lem:alberti_affine_action},
 we can find $\albrep 2.=(P_2,\nu_2)\in\Alb(\mu\mrest U^c)$, which is an Alberti
  representation of $\albrep.\mrest U^c$ with  \begin{equation}
   P_2\left(\frags(X,[2,3])\right)=1
 \end{equation}
and \begin{equation}D_{\albrep
    2.}=D_{\pullb\tau_{-1,0}.\albrep 1.}=-\chi_{U^c}D_{\albrep.};
\end{equation}
note that conditions on the Lipschitz or bi-Lipschitz
  constants satisfied by $\albrep.$ are also satisfied by $\albrep
  2.$. Using Lemma \ref{lem:domain_reduction} we find an Alberti
  representation $\albrep 3.=(P_3,\nu_3)\in\Alb(\mu\mrest U)$ with
 \begin{equation} P_3\left(\frags(X,[4,5])\right)=1
 \end{equation}
  and $D_{\albrep
    3.}=\chi_UD_{\albrep.}$; note that if $\albrep.$ satisfies
  \albcond, so does $\albrep 3.$. We now let
  \begin{align}
     P'&=\frac{1}{4}(P_1+P_2)+\frac{1}{2}P_3;\\
     \nu'(\gamma) &=
    \begin{cases}
      2\nu_1(\gamma) &\text{if $\gamma\in\frags(X,[0,1])$}\\
      2\nu_2(\gamma) &\text{if $\gamma\in\frags(X,[2,3])$}\\
      2\nu_3(\gamma) &\text{if $\gamma\in\frags(X,[4,5])$;}
    \end{cases}
  \end{align}
note that $\albrep.'=(P',\nu')\in\Alb(\mu)$ and 
\begin{equation}
  D_{\albrep.'}=\frac{1}{2}D_{\albrep 1.}+\frac{1}{2}D_{\albrep
    2.}+D_{\albrep 3.}=\chi_UD_{\albrep.}.
\end{equation}
\end{proof}
The second {lemma}, loosely speaking, allows to construct Alberti
representations whose associated derivations are linear combinations
of derivations associated with other Alberti representations. 
%\textcolor{red}
{\begin{lem}
  \label{lem:alberti_modules_comb}
  Let $\lambda\in L^{\infty}(\mu)$ be nonnegative (i.e. $\lambda\ge0$
  $\mu$-a.e.) and let $\albrep.\in\Alb(\mu)$ be a $C$-Lipschitz
  Alberti representation. Then for any $\varepsilon>0$ there exists a
  $((\|\lambda\|_\infty+\varepsilon)
  C)$-Lipschitz Alberti representation $\albrep.'\in\Alb(\mu)$ satisfying:
    \begin{equation}\label{eq:alberti_modules_comb_s1}
      D_{\albrep.'}=\lambda D_{\albrep.}.
    \end{equation}
Let  $\{\albrep j.\}_{j=1}^m\subset\Alb(\mu)$, then there is an
Alberti representation
    $\albrep.'\in\Alb(\mu)$ such that:
    \begin{equation}\label{eq:alberti_modules_comb_s2}
      D_{\albrep.'}=\sum_{j=1}^mD_{\albrep j.}.
    \end{equation}
\end{lem}}
\begin{proof}
  To show~(\ref{eq:alberti_modules_comb_s1}), let
  $M=\|\lambda\|_\infty+\varepsilon$ and let $2^{<\natural}$ denote the set of finite strings on
  $\{0,1\}$:
  \begin{equation}
    2^{<\natural}=\left\{ \alpha=(a_1,\ldots,a_m): m\in\natural, a_i\in\{0,1\} \right\};
  \end{equation}
  the length of $\alpha\in 2^{<\natural}$ will be denoted by $|\alpha|$,
  and the $k$-th entry by $\alpha(k)$. As $\lambda$ is well-defined up
  to null sets, we can assume that the range of
  $\lambda$ lies in $[0, M)$ and let
  \begin{equation}
    \begin{aligned}
      \Delta:2^{<\natural}&\to[0,M)\\
      \alpha&\mapsto\sum_{k=1}^{|\alpha|}\frac{\alpha(k)}{2^k}M;
    \end{aligned}
  \end{equation} 
  let $I(\alpha)$ denote the unique subinterval $[a,b)$ of the dyadic
  subdivision of $[0,M)$ in $2^{|\alpha|}$ intervals which contains $\alpha$.
  If we let $U_\alpha$ denote the Borel set
  \begin{equation}
    U_\alpha=\left\{x\in X: \lambda(x)\in I(\alpha)\right\},
  \end{equation}
  we can write 
  \begin{equation}
    \lambda=\sum_{n=1}^\infty\frac{M}{2^n}\sum_{|\alpha|=n}\alpha(n)\chi_{U_\alpha}.
  \end{equation} Using Lemma \ref{lem:derivation_idempotent} we
  can find a $C$-Lipschitz Alberti representation $\albrep n.'=(\tilde
  P_n,\tilde\nu_n)\in\Alb(\mu)$ with
  \begin{equation}
    D_{\albrep n.'}=\sum_{|\alpha|=n}\alpha(n)\chi_{U_\alpha}D_{\albrep .}
  \end{equation} and
  \begin{equation}
    \tilde P_n\left( \frags(X,[2(n-1),2(n-1)+1])\right)=1;
  \end{equation}
  if we reparametrize the fragments in the support of $\tilde P_n$
  using $\tau_{M,b_n}$ (where $b_n$ is chosen appropriately), 
  we obtain a $CM$-Lipschitz Alberti representation $\albrep
  n.=(P_n,\nu_n)\in\Alb(\mu)$ satisfying:
  \begin{equation}
    D_{\albrep n.}=M\sum_{|\alpha|=n}\alpha(n)\chi_{U_\alpha}D_{\albrep .}
  \end{equation}
and \begin{equation}
    P_n\left( \frags(X,[\frac{2(n-1)}{M},\frac{2(n-1)+1}{M}])\right)=1.
  \end{equation}
  We now let
  \begin{align}
    P'&=\sum_n2^{-n}P_n\\
    \nu'(\gamma)&=
    \begin{cases}
      \nu_n(\gamma)&\text{if $\gamma\in\frags(X,[2(n-1),2(n-1)+1])$}\\
      0&\text{otherwise},
    \end{cases}
  \end{align}
and conclude that 
  $\albrep.'=(P',\nu')\in\Alb(\mu)$ satisfies $D_{\albrep .'}=\lambda
  D_{\albrep.}$.
\par To show~(\ref{eq:alberti_modules_comb_s2}) we use Lemma
  \ref{lem:domain_reduction} to get Alberti representations
  $\{\albrep j.'=(P_j,\nu_j)\}_{j=1}^m$ with $D_{\albrep j.'}=D_{\albrep j.}$ and
  \begin{equation}
    P_j\left(\frags(X,[2(j-1),2(j-1)+1])\right)=1;
  \end{equation}
  if we let
  \begin{align}
    P_\oplus&=\frac{1}{m}\sum_{j=1}^mP_j\\
    \nu_\oplus(\gamma)&=
    \begin{cases}
      \nu_j(\gamma)&\text{if $\gamma\in\frags(X,[2(j-1),2(j-1)+1])$}\\
      0&\text{otherwise,}
    \end{cases}
  \end{align} we obtain $\albrep\oplus.=(P_\oplus,\nu_\oplus)\in\Alb(\mu)$ and
  \begin{equation}
    D_{\albrep\oplus.}=\frac{1}{m}\sum_{j=1}^mD_{\albrep j.};
  \end{equation} thus it suffices to take
  $\albrep.'=\pullb\tau_{m,0}.\albrep\oplus.$.
\end{proof}
The proof of Theorem \ref{thm:weak*density} relies on the following
technical lemma.
%\textcolor{red}
{\begin{lem}\label{lem:vector_alberti}
  Consider a $1$-Lipschitz function $F\colon X\to\real^N$, a derivation
  ${D_0\in\wder\mu.}$, a compact $K\subset X$ and a vector $V=c w$ where
  $c>0$ and $w\in{\mathbb S}^{N-1}$; suppose that for
  $0<s_1<s_2$ and $\epsi>0$
  \begin{equation}
    \begin{aligned}
      \locnorm D_0,{\wder\mu.}.(x)&\in(s_1,s_2)\quad(\forall x\in K)\\
      \sup_{x\in K}\left\|D_0F(x)-V\right\|_{2}&<\epsi,
    \end{aligned}
  \end{equation}
  then, for each $\alpha\in(0,\frac{\pi}{2})$, the measure $\mu\mrest
  K$ admits a $C$-Lipschitz Alberti representation $\albrep.'$ with
 \begin{equation}\label{eq:vector_alberti}
  \begin{aligned}
    \left\|D_{\albrep.'}\right\|_{\wder\mu\mrest K.}\le C&\le
    (1+\epsi)\left(\frac{s_2}{(1-2\varepsilon/c)-(N-1)\varepsilon s_2\cot(\alpha)/c}+\varepsilon\right)\\
    \left\|D_{\albrep.'}F(x)-V\right\|_2&\le c(\sin\alpha+1-\cos\alpha)\quad(\text{$\mu$-a.e.~$x\in K$}).
  \end{aligned}
\end{equation} 
\end{lem}}
\begin{proof}
%\textcolor{red}
{  Note that if $u$ is a unit vector orthogonal to $w$, one has}
  \begin{equation}
    \chi_K\left|D_0\langle u,F\rangle\right|=\chi_K\left|\langle u,D_0F-V\rangle\right|<\epsi;
  \end{equation}
supposing that $S\subset K$ is Borel and
$\frags(X,F,\delta,w,\alpha)$-null, Lemma
\ref{lem:loc_estimate_fnull} implies
\begin{equation}
  \chi_S\left|D_0\langle
    w,F\rangle\right|\le\left(\delta+(N-1){\epsi\cot\alpha}\right)\locnorm D_0,{\wder\mu.}.;
\end{equation}
on the other hand, if $\mu(S)>0$,
\begin{equation}
  \chi_S\left|D_0\langle w,F\rangle\right|\ge\chi_S\langle w,V\rangle-
    \chi_S\left|\langle
    w,D_0F-V\rangle\right|\ge\chi_S(c-\epsi).
\end{equation}
This implies
\begin{equation}\label{eq:lwxxx}
  \delta\ge\frac{c-\epsi}{s_2}-(N-1){\epsi\cot\alpha};
\end{equation}
%\textcolor{red}
{in particular, $\mu$ cannot give positive measure to a Borel $S\subset
K$ which is $\frags(X,F,\allowbreak\delta,w,\alpha)$-null if $\delta$ does not
satisfy the lower bound~(\ref{eq:lwxxx}). We therefore conclude from
Theorem~\ref{alberti_rep_prod} that the measure $\mu\mrest K$ admits a $(1,1+\epsi)$-bi-Lipschitz Alberti
representation $\albrep.=(P,\nu)$ in the $F$-direction
of $\cone(w,\alpha)$ with $\langle w,F\rangle$-speed $\ge\delta_0$
where
\begin{equation}
  \delta_0=\frac{c-2\epsi}{s_2}-(N-1){\epsi\cot\alpha}.
\end{equation}}
Thus for $P$-a.e.~$\gamma$ and for
$\mpush\gamma^{-1}.\nu_\gamma$-a.e.~$t$, one has
$(F\circ\gamma)'(t)\in\cone(w,\alpha)$. 
Thus we conclude from Theorem~\ref{thm:directional_cone} that for
$\mu$-a.e.~$x\in K$:
\begin{equation}
  \label{eq:ffx1}
  D_{\albrep.}F(x)\in\cone(w,\alpha).
\end{equation}
As the $\langle w,F\rangle$-speed of $\albrep.$ is $\ge\delta_0$ and
as $P$-a.e.~fragment is $(1,1+\varepsilon)$-bi-Lipschitz, from
Theorem~\ref{thm:directional_speed} we conclude that
$\chi_KD_{\albrep.}F\ge\delta_0\chi_K$ which implies, for \hbox{$\mu\mrest K$-a.e.~$x$}, that:
\begin{equation}
  \label{eq:ffx2}
  \|D_{\albrep.}F(x)\|_2\ge\delta_0;
\end{equation} on the other hand, as $F$ is $1$-Lipschitz,
\begin{equation}
  \label{eq:ffx3}
  \|D_{\albrep.}F(x)\|_2\le 1+\varepsilon.
\end{equation}
We now apply Lemma~\ref{lem:alberti_modules_comb} with
\begin{equation}
  \label{eq:ffx4}
  \lambda(x) = \frac{c}{\|D_{\albrep.}F(x)\|_2}\chi_K(x),
\end{equation}
in order to obtain an Alberti representation $\albrep.'$ which is
$(1+\varepsilon)(\frac{c}{\delta_0}+\varepsilon)$-Lipschitz and satisfies:
\begin{equation}
  \label{eq:ffx5}
  D_{\albrep.'}=\lambda D_{\albrep.}.
\end{equation}
Note that for $\mu$-a.e.~$x\in K$
%\textcolor{red}
{\begin{equation}
  \label{eq:ffx6}  
  \begin{split}
    \left\|D_{\albrep.'}F(x)-V\right\|_2&=
  c\left\|\frac{D_{\albrep.}F(x)}{\|{D_{\albrep.}F(x)\|_2}}-w\right\|_2\\
  &\le c(\sin\alpha+1-\cos\alpha);
  \end{split}
\end{equation}}
where in the last step we used 
Remark~\ref{rem:cone_diam} as 
$\frac{D_{\albrep.}F(x)}{\|{D_{\albrep.}F(x)\|_2}}$ lies in the cone
$C(w,\alpha)$. Now~(\ref{eq:ffx6}) establishes the second equation
in~(\ref{eq:vector_alberti}). The first equation
in~(\ref{eq:vector_alberti}) follows directly from the definition of $\delta_0$
as $\albrep.'$ is $(1+\varepsilon)(c/\delta_0+\varepsilon)$-Lipschitz.
% \begin{equation}
%   \left\|D_{\albrep.}F-w\right\|_2\le\epsi+\sin\alpha+1-\cos\alpha;
% \end{equation}
% on the other hand, from
% \begin{equation}
%   \delta_0\le\left\langle w,(F\circ\gamma)'(t)\right\rangle\le1+\epsi,
% \end{equation} one gets
% \begin{equation}
%   \lambda\le(1+\epsi)s_2+\epsi(N-1){s_2}\cot\alpha+2\epsi.
% \end{equation} If $\albrep.'=\pullb\tau_{\lambda,0}.\albrep.$ Lemma
% \ref{lem:alberti_affine_action} implies \eqref{eq:vector_alberti}.
\end{proof}
\begin{proof}[Proof of Theorem \ref{thm:weak*density}]
  It suffices to show that for each $\Omega(\xi_1,\ldots,\xi_k; D_0;
  \epsi)$ (see~(\ref{omegafxx})) there is an Alberti representation $\albrep.\in\Alb_{{\rm sub}}(\mu)$:
  \begin{equation}
    \label{eq:weak*density1}
    D_{\albrep.}\in\Omega(\xi_1,\ldots,\xi_k; D_0; \epsi).
  \end{equation}
  Note that as the $\xi_i$ satisfy $\sum_{f\in
    B_1}\|\xi_i(f)\|_{L^1(\mu)}<\infty$, there are
  $\left\{(f_i,g_i)\right\}_{i=1}^{N_1}\subset\lipalg X.\times L^1(\mu)$ and
  $\epsi_1>0$ such that if
    \begin{align}    \label{eq:weak*density2a}
    \|D\|_{\wder\mu.}&\le2\|D_0\|_{\wder\mu.},\\
    \label{eq:weak*density2b}
    \left|\int g_i(D-D_0)f_i\,d\mu\right|&<\epsi_1\quad(\forall i),
  \end{align}
 then $D\in\Omega(\xi_1,\ldots,\xi_k; D_0; \epsi)$.
Let $F=(f_i)_{i=1}^{N_1}$ and, for $D\in\wder\mu.$,
$DF=(Df_i)_{i=1}^{N_1}$; note that by possibly shrinking $\epsi_1$ we
can assume that $F:X\to\real^{N_1}$ is $1$-Lipschitz. Moreover, as we
need to produce an Alberti representation for the Borel subset
$S\subset X$ on which
 $\locnorm D_0,{\wder\mu.}.>0$ and $\|D_0F\|_2>0$, we can assume $S=X$.
Using Egorov's Theorem we find countably many
$\{(K_j,\epsi_{2,j})\}$ with $K_j\subset X$ compact and
 $\epsi_{2,j}\in(0,\infty)$ such that if
\eqref{eq:weak*density2a} holds and 
\begin{equation}
  \label{eq:weak*density3}
  \sup_{x\in
    K_j}\left\|DF(x_j)-D_0F(x_j)\right\|_{2}<\epsi_{2,j}\quad(\forall j),
\end{equation}  
 then \eqref{eq:weak*density2b} holds. Note that we can also assume
 that  
 \begin{equation}
   \begin{aligned}
     \inf_{x\in K_j}\locnorm D_0,{\wder\mu.}.(x)&>0\\
     \inf_{x\in K_j}\|D_0F(x)\|_2&>0.
   \end{aligned}
 \end{equation}
Using again Egorov's Theorem, we can further partition the $K_j$ so that
there are $(W_j,\epsi_{3,j},\alpha_j)\in
\left(\real^{N_1}\setminus\{0\}\right)\times(0,\frac{\epsi_{2,j}}{2})\times(0,\frac{\pi}{2})$ such that:
\begin{enumerate}
\item The supremum $\sup_{x\in
    K_j}\left\|D_0F(x_j)-W_j\right\|_2<\epsi_{3,j}$, and $\varepsilon_{3,j}<\varepsilon_{2,j}^2\|W_j\|_2$.
  \item For each $ x\in K_j$ $\locnorm
    D_0,{\wder\mu.}.\in(s_{1,j},s_{2,j})\subset(0,\infty)$.
    \item We have the inequality \begin{equation}
        (1+\epsi_{3,j})   \left(\frac{s_{2,j}} {
          (1-2\varepsilon_{3,j}/\|W_j\|_2)-(N-1)\varepsilon_{3,j}s_{2,j}\cot(\alpha)/\|W_2\|_j
          }+\varepsilon_{3,j}\right)\le 2\|D_0\|_{\wder\mu.}.
  \end{equation}
\item We have the inequality:
  \begin{equation}
   \|W_j\|_{2}(\sin\alpha_j+1-\cos\alpha_j)<\frac{\epsi_{2,j}}{2}.
  \end{equation}
\end{enumerate}
Using Lemma \ref{lem:vector_alberti} we obtain an Alberti
representation $\albrep j.$ of $\mu\mrest K_j$ such that
\eqref{eq:weak*density2a} and \eqref{eq:weak*density3} hold for
$D_{\albrep j.}$. Using the gluing principle Theorem~\ref{alb_glue} we
obtain $\albrep .\in\Alb_{{\rm sub}}(\mu)$ such that
(\ref{eq:weak*density1}) holds.
\end{proof}
% \begin{defn}
%   If $I\subset\real$ is a compact non-degenerate interval, let
%   \begin{equation}
%     \frags_I(X)=\left\{\gamma\in\frags(X): \dom\gamma\subset I\right\}.
%   \end{equation}
% \end{defn}
%The proof of Theorem \ref{thm:fin_gen_surjectivity} relies on the
%following technical Lemma.
We now turn to the proof of Theorem \ref{thm:fin_gen_surjectivity}.
\begin{proof}[Proof of Theorem \ref{thm:fin_gen_surjectivity}]
 %\textcolor{red}
{ As {$\chi_{X\setminus S}D=0$}, without loss of generality we can assume that $S=X$, and that $\wder\mu.$ is
  free on $\{D_{\albrep i.}\}_{i=1}^k$ where $\albrep i.\in\Alb(\mu)$
  and
  \begin{equation}
    D=\sum_{i=1}^k\lambda_iD_{\albrep i.}
  \end{equation} where the $\{\lambda_i\}_{i=1}^k\subset L^\infty(\mu)$
  are nonnegative. By~(\ref{eq:alberti_modules_comb_s1}) in
  Lemma~\ref{lem:alberti_modules_comb} we can find Alberti
  representations $\{\albrep i.'\}\subset\Alb(\mu)$ such that:
  \begin{equation}
    \label{eq:fin_gen_surjectivity_p1}
    D_{\albrep i.'}=\lambda_iD_{\albrep i.},
  \end{equation}
  so that 
  \begin{equation}
    \label{eq:fin_gen_surjectivity_p2}
    D=\sum_{i=1}^kD_{\albrep i.'}.
  \end{equation}
  By~(\ref{eq:alberti_modules_comb_s2}) in
  Lemma~\ref{lem:alberti_modules_comb} we can find an Alberti
  representation $\albrep.'\in\Alb(\mu)$ such that:
  \begin{equation}
    \label{eq:fin_gen_surjectivity_p3}
    D_{\albrep. '}=\sum_{i=1}^kD_{\albrep i.'}=D,
  \end{equation}
showing that $D\in\Der(\Alb(\mu\mrest S))$.}
\end{proof}
\subsection{Geometric characterization of $\locnorm\,\cdot\,,{\wform\mu.}.$}\label{subsec:weaver_norm}
In this subsection we prove Theorem~\ref{thm:weaver_fnorm_char}, which
gives an intrinsic and geometric characterization of the Weaver norm
on $\wder\mu.$. We also show that the inequality
$\|Df\|_{L^\infty(\mu)}\le\|D\|_{\wder\mu.}\glip f.$ localizes (Lemma
\ref{lem:local_lip_der}): the norms can be replaced by local norms, and
the global Lipschitz constant $\glip f.$ can be replaced by $\biglip f$.
\begin{thm}\label{thm:weaver_fnorm_char}
  Let $X$ be a compact metric space, $f\colon X\to\real$ Lipschitz,
  $S\subset X$ Borel, and $\mu$ a Radon measure on $X$. Then $\locnorm
  df,{\wform\mu\mrest S.}.\le\alpha$ if and only if for each $\epsi>0$
  there is an $S_\epsi\subset S$ which is Borel,
  $\frags(X,f,\alpha+\epsi)$-null, and satisfies $\mu(S\setminus
  S_\epsi)=0$. In particular,
  \begin{equation}
    \|df\|_{\wform\mu.}=\sup\left\{\|Df\|_{L^\infty(\mu)}:\text{$D\in\Der(\Alb_{{\rm
            sub}}(\mu))$ and $\|D\|_{\wder\mu.}\le1$}\right\}.
  \end{equation}
\end{thm}
\begin{proof}
  Necessity is proven by contrapositive. Assume that for some
  $\epsi>0$ the set $S$ does not contain a full $\mu$-measure subset which is
  $\frags(X,f,\alpha+\epsi)$-null. Inspection of
  the proof of Theorem~\ref{alberti_rep_prod} (see the discussion
  around~(\ref{iterxx})) shows
  that for each $\eta>0$
  \begin{equation}
    \mu=\mu\mrest G_\eta+\mu\mrest F_\eta,
  \end{equation}
where $G_\eta$, $F_\eta$ are complementary, $F_\eta$ being a
$\frags(X,f,\alpha+\epsi)$-null $F_{\sigma\delta}$ and $\mu\mrest
G_\eta$ admitting a $(1,1+\eta)$-bi-Lipschitz Alberti representation $\albrep\eta.=(P,\nu)$
with $f$-speed $\ge\alpha+\epsi$. By assumption, $F_\eta\cap S$ cannot be a
full $\mu$-measure subset of $S$, so $\mu(G_\eta\cap S)>0$.
 If $D_\eta$ denotes the derivation
associated to $\albrep\eta.$,
\begin{equation}
  \|D_\eta\|_{\wder\mu.}\le 1+\eta;
\end{equation}
%\textcolor{red}
{moreover, as for $P$-a.e.~$\gamma$ we have for
$\gamma^{-1}_{\#}\nu_\gamma$-a.e.~$t$ that
$(f\circ\gamma)'(t)\ge\alpha+\varepsilon$, we conclude that $\mu\mrest
G_\eta$-a.e.~one has:
\begin{equation}
  D_\eta f\ge\alpha+\epsi,
\end{equation}
and thus the following holds $\mu\mrest
G_\eta$-a.e.:
\begin{equation}
  \locnorm df,{\mu\mrest G_\eta}.\ge\frac{\alpha+\epsi}{1+\eta}>\alpha,
\end{equation} where $\eta$ is sufficiently small.}
\par Sufficiency is proven via the approximation scheme Theorem
\ref{onedimapprox}, reducing to the case in which $S$ is compact.
By replacing $S$ by $\bigcap_nS_{\frac{1}{n}}$, we can assume that
$S$ is, for each $n$,
$\frags(X,f,\alpha+{\frac{1}{n}})$-null; then there exist functions
$g_n$ which are
$\max\left(\glip f.,\alpha+{\frac{1}{n}}\right)$-Lipschitz and
$\mu$-a.e.~$(\alpha+{\frac{1}{n}})$\nobreakdash-\hspace{0pt}Lipschitz on $S$ with
$\|g_n-f\|_\infty\le{\frac{1}{n}}$. Thus, $g_n\xrightarrow{\text{w*}}
f$ so that $dg_n\xrightarrow{\text{w*}}df$ and by lower
semicontinuity,
\begin{equation}
  \|df\|_{L^\infty(\mu\mrest S)}\le\liminf_{n\to\infty}\|dg_{{n}}\|_{L^\infty(\mu\mrest S)}\le\alpha.
\end{equation}
\end{proof}
\par We also present a proof of the following useful result.
\begin{lem}\label{lem:local_lip_der}
    Let $X$ be a metric space, $\mu$ a Radon measure on $X$,
    $f\in\lipalg X.$ and $D\in\wder\mu.$; then
    \begin{equation}
      |Df|\le\locnorm D,{\wder\mu.}.\biglip f.
    \end{equation}
\end{lem}
\begin{proof}
  For all $\epsi_0,\epsi_1>0$, using Egorov and Lusin's Theorems,
  there are triples $(K_\alpha,\lambda_\alpha,r_\alpha)$ such that:
  \begin{enumerate}
  \item The $K_\alpha$ are disjoint compact with
    $\mu\left(X\setminus\bigcup_\alpha K_\alpha\right)=0$ and
    $\mu(K_\alpha)>0$.
\item The constants $\lambda_\alpha$ are nonnegative.
\item The constants $r_\alpha$ are positive and $K_\alpha$ has diameter
  less than $r_\alpha$.
\item For each $x\in K_\alpha$ and each $r\in(0,r_\alpha]$,
\begin{equation}
  \biglip f(x,r)\in(\lambda_\alpha-\epsi_0,\lambda_\alpha).
\end{equation}
\item The local norm $\locnorm D,{\wder\mu\mrest K_\alpha.}.$ lies in $(\|D\|_{\wder\mu\mrest K_\alpha.}-\epsi_1,\|D\|_{\wder\mu\mrest K_\alpha.}]$.
  \end{enumerate}
By point (3) $f$ is $(\lambda_\alpha)$-Lipschitz on
$K_\alpha$; let $F$ denote a McShane extension of $f|K_\alpha$ with
$\glip F.\le\lambda_\alpha$. Then
\begin{equation}
  \|Df\|_{L^\infty(\mu\mrest K_\alpha)}=\|DF\|_{L^\infty(\mu\mrest
    K_\alpha)}\le\lambda_\alpha\|D\|_{\wder\mu\mrest K_\alpha.};
\end{equation}
thus
\begin{equation}
\begin{split}
  \chi_{K_\alpha}|Df|&\le\lambda_\alpha\left(\locnorm D,{\wder\mu\mrest
    K_\alpha.}.+\epsi_1\right)\\
  &<(\biglip f+\epsi_0)\left(\locnorm D,{\wder\mu\mrest
    K_\alpha.}.+\epsi_1\right).
\end{split}
\end{equation}
As $\sum_\alpha\chi_{K_\alpha}=1$ in $L^\infty(\mu)$ (the convergence
of the series is in the weak* sense),
\begin{equation}
  |Df|<(\biglip f+\epsi_0)\left(\locnorm D,{\wder\mu.}.+\epsi_1\right).
\end{equation}
Letting $\epsi_0,\epsi_1\searrow0$ completes the proof.
\end{proof}
\section{Structure of differentiability spaces}\label{sec:app_diff_spaces}
\subsection{Differentiability spaces and derivations}\label{subsec:derivation_differentiability}
In this subsection we make a first application of derivations to study
differentiability spaces. The goal is to prove Lemma
\ref{partialderivatives} and Corollary \ref{cor:assouad_bound}. We
give a first proof of Lemma \ref{partialderivatives} which depends on
the results in \cite{bate-diff}; another proof, which relies on
Theorem \ref{thm:diff_char1} and \cite{derivdiff}, can be found at the
end of Subsection \ref{subsec:char_diff}.
\begin{lem}\label{partialderivatives}
   Let $(U,x)$ be a differentiability chart, with $U$ Borel, in a
   differentiability space $(X,\mu)$. Then the partial derivative operators
   $\frac{\partial}{\partial x^j}$ are derivations.
\end{lem}
\begin{proof}
  Let $n$ be the dimension of the chart $(U,x)$; we first show that
  there are disjoint Borel sets $\{U_\alpha\}$ with $U_\alpha\subset
  U$ and $\mu\left(U\setminus\bigcup_\alpha U_\alpha\right)=0$, and such that,
  for each $\alpha$, it is possible to find $g_{\alpha,k}^i\in
  L^\infty(\mu\mrest U_\alpha)$ ($i,k=1,\ldots,n)$ and derivations
  associated to Alberti representations $\{\albrep
  i.\}_{i=1}^n$, such that
  \begin{equation}\label{derrep}
    \sum_{i=1}^ng_{\alpha,k}^iD_{\albrep i.\mrest U_\alpha}x^j=\delta_k^j\chi_{U_\alpha};
  \end{equation} in fact, because of \cite[Lem.~6.2]{derivdiff}, in a
  differentiability space
  derivations are completely determined by their action on the chart
  functions and so \eqref{derrep} represents $\chi_{U_\alpha}\frac{\partial}{\partial
    x^k}$. Note that even though in \cite[Sec.~6]{derivdiff} the
  measure $\mu$ is assumed to be doubling, this is not restrictive by
  Theorem \ref{thm:bate_speight}. By \cite[Thm.~6.6]{bate-diff} we know that $\mu\mrest U$ admits
  Alberti representations in the $x$-directions of independent cone
  fields $\{\cone_l\}_{l=1}^n$. Using Corollary \ref{der-alb}, we obtain
  $(1,3/2)$-bi-Lipschitz Alberti representations $\{\albrep
  i.\}_{i=1}^n$ in the $x$-directions
  of cone fields $\cone(e_i,\theta)$ where $\{e_i\}_{i=1}^n$ is the standard
  Euclidean basis and $\theta$ is sufficiently small to ensure that
  the cone fields $\cone(e_i,\theta)$ are independent. In particular,
  letting $M=(D_{\albrep i.}x^j)_{i,j=1}^n$, %\textcolor{red}
{minding Theorem~\ref{thm:directional_cone}} and choosing Borel
  representatives for the entries of $M$, the bounded Borel function $\det M$ is different from zero on a $\mu$-full measure
  subset. In particular, one can choose disjoint Borel subsets
  $U_\alpha\subset U$ with $\mu\left(U\setminus\bigcup_\alpha
    U_\alpha\right)=0$, and such that, on each $U_\alpha$ the determinant $\det M$ is uniformly bounded
  from below by $\delta_\alpha>0$. Inverting the matrix $M$ one
  concludes that there are $g_{\alpha,k}^i\in L^\infty(\mu\mrest
  U_\alpha)$ such that \eqref{derrep} holds.
  \par The definition of chart implies that the partial derivative
  operators $\frac{\partial}{\partial x^j}$ are bounded linear maps
  from $\lipalg X.\to L^\infty(\mu\mrest U)$ that satisfy the product
  rule. In order to show weak* continuity, suppose that
  $f_t\xrightarrow{\text{w*}} f$ in $\lipalg X.$ and that $h\in
  L^1(\mu\mrest U)$. For each $\epsi>0$ there are finitely many
  $\{U_{\alpha_s}\}_{s=1}^{N(\epsi)}$ such that
  \begin{align}\label{eq:partialderivatives_p1}
    \sup_t\left|\int_{U\setminus\bigcup_{s=1}^{N(\epsi)}U_{\alpha_s}}h\frac{\partial{f_t}}{\partial
        x^j}\,d\mu\right|&\le\epsi; \\ \label{eq:partialderivatives_p2}
    \left|\int_{U\setminus\bigcup_{s=1}^{N(\epsi)}U_{\alpha_s}}h\frac{\partial{f}}{\partial
        x^j}\,d\mu\right|&\le\epsi.
  \end{align}
  Combining equations \eqref{eq:partialderivatives_p1} and
  \eqref{eq:partialderivatives_p2} with the fact that
  $\sum_{s=1}^{N(\epsi)}\chi_{U_{\alpha,s}}\frac{\partial}{\partial
    x^j}$ is a derivation, we obtain
  \begin{equation}
    \lim_{t\to\infty}\int_Uh\frac{\partial{f_t}}{\partial
        x^j}\,d\mu=\int_Uh\frac{\partial{f}}{\partial
        x^j}\,d\mu,
  \end{equation}
which shows that $\frac{\partial}{\partial x^j}$ is weak* continuous.
\end{proof}
\par We now prove the dimensional bound Corollary
\ref{cor:assouad_bound} which significantly strengthens previous
bounds on differentiability dimensions. In fact, we remove from
Theorems \ref{thm:cheeger}, \ref{thm:keith} and
\ref{thm:bate-diff_char} the dependence on $\tau$ from the bound on
the differentiability dimension.
\begin{cor}\label{cor:assouad_bound}
  Suppose $(X,\mu)$ is a $\sigma$-differentiability space with
  $\mu$ doubling, then $(X,\mu)$ is a differentiability space and
  the dimension of the mea\-surable differentiable structure is at most
  the Assouad dimension of $X$.
\end{cor}
\begin{proof}
%\textcolor{red}
{This result follows by applying Lemma \ref{derbound}, which shows that the
index of $\wder\mu.$ is locally bounded by the Assouad dimension of $X$, and Lemma
\ref{partialderivatives}.}
\end{proof}
\subsection{Characterization of differentiability spaces}\label{subsec:char_diff}
In this subsection we obtain a new characterization of
differentiability spaces. The goal is to prove Theorems
\ref{thm:diff_char1} and \ref{thm:diff_char2}. %\textcolor{red}
{Throughout this section
the metric space $X$ is assumed to be {\bf separable and complete}.}
\par We first introduce a family of sets used in the characterization
of differentiability.
\begin{defn}
  Given a complete metric space $X$ and a Radon measure $\mu$ on $X$,
  we define by $\Gap(\mu,X)$ the class of Borel subsets $V\subset X$
  such that there are $\alpha>\beta\ge0$ and a Lipschitz function
  $f:X\to\real$ with
  \begin{equation}
    \inf_{x\in V}\biglip f(x)\ge\alpha\quad\text{and}\quad\locnorm
    df,{\wform\mu\mrest V.}.\le\beta.
  \end{equation}
Note that the inequality involving $\locnorm
    df,{\wform\mu\mrest V.}.$ holds in $L^\infty(\mu\mrest V)$.
We define by $\Gap_0(\mu,X)$ the class of Borel subsets $V\subset X$
  such that there are $\alpha>0$ and a Lipschitz function $f:X\to\real$ with
  \begin{equation}
    \inf_{x\in V}\biglip f(x)\ge\alpha\quad\text{and}\quad\locnorm
    df,{\wform\mu\mrest V.}.=0.
  \end{equation}
\par Note that $\Gap_0(\mu,X)\subset\Gap(\mu,X)$ and if $Y\subset X$ is
Borel, applying McShane's Lemma and \eqref{eq:local_Lip},
% \textcolor{red}
{which states
that for a Lipschitz function $f$ and a point ${y\in Y}$ one has $\biglip
f(y)\ge\biglip_Yf(y)$, where the infinitesimal Lipschitz constant
$\biglip f(y)$ (resp.~$\biglip_Yf(y)$) is computed in $X$ (resp.~$Y$)}, we get:
\begin{align}
  \Gap(\mu\mrest Y,Y)&\subset\Gap(\mu,X)\\
  \Gap_0(\mu\mrest Y,Y)&\subset\Gap_0(\mu,X).
\end{align}
We finally let
\begin{align}
  \Gap(X)&=\bigcap_{\text{$\mu$ Radon}}\Gap(\mu,X)\\
  \Gap_0(X)&=\bigcap_{\text{$\mu$ Radon}}\Gap_0(\mu,X).
\end{align}
Note that $\Gap_0(X)\subset\Gap(X)$ and if $Y\subset X$ is
Borel,
\begin{align}
  \Gap(Y)&\subset\Gap(X)\\
\Gap_0(Y)&\subset\Gap_0(X).
\end{align}
\end{defn}
\begin{thm}\label{thm:diff_char1}
  The metric measure space  $(X,\mu)$ is a $\sigma$-differentiability
  space if and only if for each $S\in\Gap(X)$ one has:
  \begin{equation}\mu(S)=0.
  \end{equation}
\end{thm} 
The proofs of Theorems
\ref{thm:diff_char1} and \ref{thm:diff_char2} require some
preparation. We will use Theorem \ref{lip_ind}, the proof of which is
deferred to Subsection \ref{subsec:constr-indep-lipsch}.
\begin{thm}\label{lip_ind}
  Let $S\subset X$ be a Borel set and assume that for some $
  \delta_0\in(0,1]$ and for each $m\in\natural$ there is an $L$-Lipschitz
  function $f_m$ such that:
  \begin{enumerate}
  \item There is a $\rho_m\in(0,\frac{1}{m})$ such that if $B$ is a ball
    of radius $\rho_m$ centred at some point of $S$, the Lipschitz constant
    of the restriction $f|_B$ is at most $\frac{1}{m}$.
  \item For each $x\in S$ there is $y\in \ball x, \frac{1}{m}.$ such that
    \begin{equation}
      |f_m(x)-f_m(y)|\ge\delta_0 d(x,y)>0.
    \end{equation}
  \end{enumerate}
  \par Then for all $(M,\alpha)\in\natural\times(0,\min\left(
    \frac{1}{2}\sqrt{\frac{\delta_0
        L}{1+\delta_0}},\frac{1}{2}\delta_0\right))$ there are a
  Borel subset $S'\subset S$ and Lipschitz functions
  $\{\psi_0,\ldots,\psi_{M-1}\}$ such that: 
  \begin{itemize}
  \item The set $S\setminus S'$ is $\mu$-null.
  \item The $\psi_i$ have Lipschitz constant bounded by
    \begin{equation}
      3\left(L+\alpha+
        \frac{\alpha^2/L}{1-\alpha^2/L}(1+2^{-6}\alpha)\right).
    \end{equation}
  \item For each $x\in S'$:
    \begin{equation}
      \biglip\left(\sum_{i=0}^{M-1}\lambda_i\psi_i\right)(x)\ge
      \left(\delta_0-\frac{\alpha^2/L}{1-\alpha^2/L}-\alpha\right)\max_{i=0,\ldots,M-1}|\lambda_i|.
    \end{equation}
  \end{itemize}
\end{thm}
\par As a first step, we prove a lemma that produces
\emph{flat functions} as in the hypothesis of the preceding Theorem~\ref{lip_ind}.
\begin{lem}\label{lem:gap_flat_func}
  Let $X$ be a compact metric space with finite Assouad dimension, $f:X\to\real$ Lipschitz,
  $S\subset X$ Borel, and $\mu$ a Radon measure on $X$. Suppose that
  there are $\alpha$ and $\beta$ such that $\alpha>\beta\ge0$ and
  \begin{equation}
    \inf_{x\in S}\biglip f(x)\ge\alpha>\beta\ge\locnorm
    df,{\wform\mu\mrest S.}.;
  \end{equation}
then for each $(m,\delta,\gamma,\epsi)\in\natural\times(0,1)\times(0,\alpha-\beta)\times(0,1)$,
there is a triple $(K,r_m,h)$:
\begin{enumerate}
\item The set $K$ is compact with $\mu(S\setminus K)\le\epsi$.
\item The function $h$ is $3\glip f.$-Lipschitz and $r_m>0$.
\item For each $ x\in K$ there is $y$ such that
  \begin{equation}
    0<\dist x,y.\le\frac{1}{m}\quad\text{and}\quad
    \left|h(x)-h(y)\right|\ge\gamma\dist x,y..
  \end{equation}
\item For each $ x\in K$, $\glip h|{\ball x,r_m.}.\le\delta$.
\end{enumerate}
\par In the case in which $\beta=0$ the assumption on the finite Assouad
dimension is not needed.
\end{lem}
\begin{proof}
The rough idea of the proof is to first use Functional Analysis {to} produce a function $\hat
g_{n_0}$ such that $d\hat g_{n_0} - df$ has small local norm but such
that ${\biglip(\hat g_{n_0} - f)}$ is bounded away from zero. At this
point, one
can then apply the approximation scheme, Theorem~\ref{onedimapprox}, to the
function $\hat g_{n_0} - f$. 
\par Without loss of generality we can assume that $S$ is compact and,  for each $\eta_0>0$,
$\frags(X,f,\beta+\eta_0)$-null. Thus, by
  Theorem \ref{onedimapprox} there are functions $g_n\colon X\to\real$ which are 
  \begin{equation}
    L_0=\max\left(\glip f.,\beta+\eta_0\right)\text{-Lipschitz,}
  \end{equation}
and such that $g_n\xrightarrow{\text{w*}}f$ and $g_n$ is $\mu$-a.e.~locally
  $(\beta+\eta_0)$-Lipschitz on $S$. \par The part of the argument starting here is the only one that
requires that the Assouad dimension of the space is finite and is not
needed if $\beta=0$. Note that
$dg_n\xrightarrow{\text{w*}}df$ in $\wform\mu.$. By finiteness of the
Assouad dimension, by Lemma \ref{derbound} %\textcolor{red}
{(which shows that the
index of $\wder\mu.$ is locally bounded by the Assouad dimension of
$X$)} and Theorem
\ref{thm:free_dec}, there are finitely many disjoint Borel
$\{X_a\}$ with $\mu(X_a)>0$,
$\mu\left(X\setminus\bigcup_a X_a\right)=0$ and $\wder\mu\mrest
X_a.$ free of rank $N_a$.  Note that
\begin{equation}
  \wder\mu.=\bigoplus_a\wder\mu\mrest X_a.
\end{equation}
and Lemma \ref{lem:hahn-banach} implies that $\wform\mu\mrest
X_a.$ is also free of rank $N_a$. 
 In particular, we have
\begin{equation}
  \wform\mu.=\bigoplus_a\wform\mu\mrest X_a.
\end{equation}
and for each $a$ one can find a basis of $\wder\mu\mrest X_a.$,
$\{D_{a,i}\}_{i=1}^{N_a}$, and a constant $C_a>0$ such
that for each
$\omega\in\wform\mu\mrest X_a.$ one has:
\begin{equation}
  \locnorm\omega,{\wform\mu\mrest X_a.}.\le
  C_a\left(\sum_{i=1}^{N_a} \left|\langle D_{a,i},\omega\rangle\right|^2\right)^{1/2}.
\end{equation}
We let $N=\max_a N_a$, $C=\max_a C_a$ and
$\{e_i\}_{i=1}^{N_a}$ the standard basis of
$\real^{N_a}$. Consider the map
\begin{equation}
  \begin{aligned}
    \iota:\wform\mu.&\to L^2(\mu,\real^N)\\
    \omega&\mapsto\sum_a\sum_{i=1}^{N_a}\chi_{X_a}\langle
    D_{a,i},\omega\rangle e_i;
  \end{aligned}
\end{equation}
then
\begin{equation}
  \locnorm\omega,{\wform\mu.}.\le C|\iota(\omega)|_{l^2}\in L^\infty(\mu)
\end{equation}
and $\iota(dg_n)\xrightarrow{\text{w}-L^2}\iota(df)$ (the notation
$\xrightarrow{\text{w}-L^2}$
is used to denote weak convergence in
$L^2(\mu,\real^N)$). By Mazur's Lemma there are tail-convex combinations
\begin{equation}
  \hat g_n=\sum_{k=n}^{M(n)}t_kg_k
\end{equation}
with $\iota(d\hat g_n)\xrightarrow{L^2}\iota(df)$. Note that the $\hat
g_n$ are $L_0$-Lipschitz and $\mu$-a.e.~locally
  ${(\beta+\eta_0)}$\nobreakdash-\hspace{0pt}Lipschitz on $S$. By passing to a subsequence we
can assume that $\iota(\hat g_n)\to\iota(df)$ $\mu$-a.e. Using Egorov's
Theorem, for each $\eta_1,\eta_2>0$ there are a compact $K_0\subset S$ and
$n_0\in\natural$:
\begin{align}
  \mu\left(S\setminus K_0\right)&\le\eta_1,\\
  \sup_{x\in K_0}\left|\iota(d\hat g_{n_0})(x)-\iota(df)(x)\right|&\le\frac{\eta_2}{C},
\end{align}
which implies
\begin{equation}
  \locnorm d\hat g_{n_0}-df,{\wform\mu\mrest K_0.}.\le\eta_2.
\end{equation}
This is the end of the part of the argument that
requires that the Assouad dimension of the space is finite. 
\par As $\biglip$ behaves like a seminorm,
\begin{equation}
  \inf_{x\in S}\biglip\left(f-\hat g_{n_0}\right)(x)\ge\inf_{x\in S}\biglip
  f(x)-\sup_{x\in S}\biglip \hat g_{n_0}(x)\ge\alpha-\beta-\eta_0;
\end{equation}
using the Borel measurability of $\biglip\left(f-\hat g_{n_0}\right)$ and 
Egorov and Lusin's Theorems, for all positive $\eta_3$, $\eta_4$ and
for each natural $m$, there are a compact
$K_1\subset K_0$, and constants $\rho_0\ge\rho_1>0$ such that:
\begin{enumerate}
\item We have the inequalities:
\begin{align} \rho_1\le\rho_0&<\frac{1}{m},\\
\mu(K_0\setminus K_1)&\le\eta_3.
\end{align}
\item For each $ x\in K_1$ there is $ y\in X$: $\dist
  x,y.\in[\rho_1,\rho_0]$ and
  \begin{equation}\label{eq:low_var_bound}
    \left|\left(f-\hat g_{n_0}\right)(x)-\left(f-\hat
    g_{n_0}\right)(y)\right|\ge(\alpha-\beta-\eta_0-\eta_4)\dist x,y..
  \end{equation}
\end{enumerate}
Note that $f-\hat g_{n_0}$ is 
\begin{equation}
  L_1=\left(\glip f.+L_0\right)\text{-Lipschitz}.
\end{equation}
Note also that for each $\eta_5>0$ the compact $K_0$ is $\frags(X,f-\hat
g_{n_0},\eta_2+\eta_5)$-null. Applying again Theorem
\ref{onedimapprox} we can find functions $h_n:X\to\real$ which are 
  \begin{equation}
L_2=\max(L_1,\eta_2+\eta_5)\text{-Lipschitz,}
  \end{equation}
with $h_n\xrightarrow{\text{w*}}f-\hat g_{n_0}$ and which are $\mu$-a.e.~locally
  $(\eta_2+\eta_5)$-Lipschitz on $K_0$. Thus, for all $\eta_6,\eta_7>0$, there are
$n_1\in\natural$, a compact $K_2\subset K_1$ and constants
$\rho_1\ge\rho_2>0$:
\begin{align}
  \mu(K_1\setminus K_2)&\le\eta_6,\\
  \left\|h_{n_1}-\left(f-\hat g_{n_0}\right)\right\|_\infty&\le\eta_7\rho_1,
\end{align}
and for each $x\in K_2$, $\glip h|{\ball x,\rho_2.}.$ does not exceed
$\eta_2+\eta_5$.
Let $x\in K_2$; then there is $y\in X$ with $\dist
  x,y.\in[\rho_1,\rho_0]$ such that \eqref{eq:low_var_bound}
  holds. Therefore,
  \begin{equation}
    \begin{split}
      \left|h_{n_1}(x)-h_{n_1}(y)\right|&\ge\left|\left(f-\hat
      g_{n_0}\right)(x)-\left(f-\hat g_{n_0}\right)(y)\right|
      -2\left\|h_{n_1}-\left(f-\hat g_{n_0}\right)\right\|_\infty\\
&\ge(\alpha-\beta-\eta_0-\eta_4-2\eta_7)\dist x,y..
    \end{split}
  \end{equation}
We now require the constants $\eta_i$ to satisfy
\begin{align}
  \eta_1+\eta_3+\eta_6&\le\epsi\\
  \eta_2+\eta_5&\le\min(\delta,\glip f.)\\
  \gamma&\le(\alpha-\beta-\eta_0-\eta_4-2\eta_7)\\
  \beta+\eta_0&\le2\beta;
\end{align}
as a consequence, $L_2\le 3\glip f.$. We let $K=K_2$,
$h=h_{n_1}$ and $r_m=\rho_2$.
\end{proof}
To see how \textbf{porosity} interacts with differentiability we
will need to consider the pointwise Lipschitz constants of a function
with respect to a given subset.
\begin{defn}
  Let $Y\subset X$, $f:X\to\real$ Lipschitz. For $y\in Y$, the \textbf{pointwise
upper and lower Lipschitz constants of $f|_Y$ at $y$} will be denoted by $\biglip_Y
  f(y)$ and $\smllip_Y f(y)$. Note that
  \begin{align}
    \label{eq:local_Lip}
    \biglip_Y
  f(y)&\le\biglip f(y)\\
\label{eq:local_lip}
\smllip_Y f(y)&\le\smllip f(y).
  \end{align}
\end{defn}
We now introduce the classes of subsets used to characterize
differentiability. %We will assume that the metric space $X$ is
%compact. 
Even though we found it useful to work with $\Gap(\mu, X)$ sets, the
following lemma shows that one can just work with $\Gap(X)$ sets.
\begin{lem}
  \label{lem:gap_incl}
If $S\in\Gap(\mu,X)$ with $\mu(S)>0$, then there is a Borel $S'\subset S$ with
$\mu\left(S\setminus S'\right)=0$ and $S'\in\Gap(X)$. The same
conclusion holds replacing $\Gap$ with $\Gap_0$.
\end{lem} 
\begin{proof}
  Let $(f,\alpha,\beta)$ be as in the defining property of $\Gap$. By
  Theorem \ref{thm:weaver_fnorm_char}, for $n$ sufficiently large,
  there is a $\frags(X,f,\beta+\frac{1}{n})$-null Borel $S_n\subset S$
  such that $\mu(S\setminus S_n)=0$. Let $S'=\bigcap_n S_n$ so that
  $S'$ is a full $\mu$-measure Borel subset of $S$ which is
  $\frags(X,f,\beta+\frac{1}{n})$-null for each $n$. Thus, Theorem
  \ref{thm:weaver_fnorm_char} implies that
  \begin{equation}
    \locnorm df,{\wform\nu\mrest S'.}.\le\beta
  \end{equation} for all Radon measures $\nu$. As
  \begin{equation}
    \inf_{x\in S'}\biglip f(x)\ge\alpha,
  \end{equation} $S'\in\Gap(X)$.
\end{proof}
\par We now use Lemma \ref{lem:gap_flat_func} to construct independent
Lipschitz functions.
\begin{lem}
  \label{lem:ind_lip}
  Let $X$ be a compact metric space with finite Assouad dimension,
  $f\colon X\to\real$ Lipschitz, $S\in\Gap(\mu,X)$ with $\mu(S)>0$. Then
for all  $\epsi>0$ and $M\in\natural$, there is a Borel $S'\subset S$
  with $\mu\left(S\setminus S'\right)\le\epsi$ and there are $1$-Lipschitz
  functions $\left\{\psi_0,\ldots,\psi_{M-1}\right\}$ which are
  infinitesimally independent on $S'$. 
  \par In the case in which $S\in\Gap_0(\mu,X)$ the assumption on the
  Assouad dimension is not needed.
\end{lem}
\begin{proof}
  Let $(f,\alpha,\beta)$ be as in the defining property of $\Gap(X,\mu)$. By Lemma
  \ref{lem:gap_flat_func}, for each $m\in\natural$ and
  $\epsi_1^{(m)}>0$, there are $(K_m,\rho_m,h_m)$:
  \begin{enumerate}
  \item The set $K_m$ is compact with $\mu\left(S\setminus
    K_m\right)\le\epsi_1^{(m)}$.
\item The function $h_m$ is $3\glip f.$-Lipschitz and $\rho_m>0$.
%\textcolor{red}
{\item For each $ x\in K_m$ there is  $ y$:
  \begin{equation}
    0<\dist x,y.\le\frac{1}{m}\quad\text{and}\quad
    \left|h_m(x)-h_m(y)\right|\ge\frac{\alpha-\beta}{2}\dist x,y..
  \end{equation}}
\item For each $ x\in K_m$, $\glip h_m|{\ball x,\rho_m.}.\le\frac{1}{m}$.
  \end{enumerate}
Let $K=\bigcap_m K_m$ so that 
\begin{equation}
  \mu\left(S\setminus
    K\right)\le\sum_m\epsi_1^{(m)};
\end{equation}
choosing the $\epsi_1^{(m)}$ sufficiently small, we can ensure
$  \mu\left(S\setminus
    K\right)\le\epsi$. Applying Theorem \ref{lip_ind} 
     to $K$ we construct the functions $\left\{\psi_0,\ldots,\psi_{M-1}\right\}$.
\end{proof}
\par We now recall some facts about porosity and refer the reader to the
survey \cite{zajicek_por_surv} for more information.
\begin{defn}
  For $Y\subset X$ and $c>0$ we say that $Y$ is \textbf{$c$-porous at $y$} if
there is a sequence  $x_n\to y$ (convergence in $X$) with
  \begin{equation}
    0<c\dist x_n,y.<\setdist\{x_n\},Y..
  \end{equation}
%\textcolor{red}
{For $y\in Y$ we let
\begin{equation}
  W_c(y,Y)=\left\{y'\in X:  0<c\dist y' ,y.<\setdist\{y'\},Y.\right\};
\end{equation}}
thus $c$-porosity of $Y$ at $y$ is equivalent to:
\begin{equation}
  W_c(y,Y)\cap\ball y,r.\ne\emptyset\quad(\forall r>0).
\end{equation}
If $Y$ is $c$-porous at each $y\in Y$, $Y$ is called \textbf{$c$-porous}.
\end{defn} 
\begin{rem}Given a $c$-porous subset $Y\subset X$, one might wonder if there is
  also a Borel porous subset of $X$. By
  \cite[Lem.~1.4]{preiss_porosity} it follows that if $Y$ is
  $c$-porous, for each $\epsi>0$ there is $\tilde Y$, a $G_\delta$ set,
  with $Y\subset\tilde Y$ and $\tilde Y$ $(c-\epsi)$-porous.
\end{rem}
\begin{rem}\label{porosity_remark}
%\textcolor{red}
{Porosity is related to the relationship between $\biglip_Yf$ and
$\biglip f$. In the following this will be important in order to have
a well-defined notion of derivative at $y\in Y$. In fact, assume that
we have shown that $(Y,\mu\mrest Y)$ is a differentiability space and
found the derivative of $f$ in $Y$ at $y\in Y$ according
to~(\ref{diff_linearization}):
\begin{equation}\label{porxx1}
      \biglip_Y\left(f-\sum_{j=1}^n\frac{\partial
      f}{\partial x_\alpha^j}(y)x_\alpha^j\right)(y)=0;
\end{equation}
now we would like to conclude that $\frac{\partial
      f}{\partial x_\alpha^j}(y)$ gives also the derivative in $X$:
\begin{equation}\label{porxx2}
      \biglip\left(f-\sum_{j=1}^n\frac{\partial
      f}{\partial x_\alpha^j}(y)x_\alpha^j\right)(y)=0.
\end{equation}
For~(\ref{porxx1}) and (\ref{porxx2}) to be compatible we must have
$\biglip_Yf(y)=\biglip f(y)$; but this does not need to be the case.
For example, assume that $Y$ is $c$-porous at $y$ and let
$f=\setdist Y,\{\cdot\}.$.} Then
\begin{equation}\label{eq:por_Lip_lbound}
  \biglip f(y)\ge\limsup_{n\to\infty}\frac{\setdist\{y_n\},Y.}{\dist y_n,y.}>c;
\end{equation}
but $\biglip_Yf(y)=0$ as $f$ is identically zero on $Y$.
On the other hand, if for all $c\in(0,\frac{1}{2})$ the set $Y$ is not $c$-porous at $y\in Y$,
then, for each Lipschitz function $f$,
\begin{equation}
  \label{eq:npor_Lip_equal}
\biglip_Yf(y)=\biglip f(y).
\end{equation}
To see this, choose $x_n\to y$, with $\dist x_n,y.>0$ and
\begin{equation}
  \biglip f(y)=\lim_{n\to\infty}\frac{|f(y)-f(x_n)|}{\dist y,x_n.};
\end{equation}
%\textcolor{red}
{as $Y$ is not $c$-porous at $y$, there is a radius $r_c>0$ such that
\begin{equation}
  \sup_{z\in \ball y,r.}\setdist{\{z\}},Y.\le cr\quad(\forall r\in(0,r_c]).
\end{equation}
So for $n$ large enough we find $Y\ni y_n\ne y$ with
\begin{equation}
  \dist x_n,y_n.\le (c+2/n)\dist x_n,y.;
\end{equation}
this implies
\begin{equation}
  \begin{split}
    \biglip_Yf(y)&\ge\limsup_{n\to\infty}\frac{|f(y)-f(y_n)|}{\dist
      y,y_n.}\\
    &\ge\limsup_{n\to\infty}\frac{|f(y)-f(x_n)|-(c+2/n)\glip f.\dist
      y,x_n.}{(1+c+2/n)\dist x_n,y.}\\
    &\ge\frac{1}{1+c}\left(\biglip f(y)-c\glip f.\right).
  \end{split}
\end{equation}}
Letting $c\to0$ we obtain \eqref{eq:npor_Lip_equal}. 
 \end{rem}
We now show that porous sets are of class $\Gap_0(X)$.
\begin{lem}\label{lem:por_gap}
  If $S\subset X$ is $c$-porous and Borel, then $S\in\Gap_0(X)$.
\end{lem}
\begin{proof}
  Let $\nu$ be a Radon measure on $X$. Let $f=\setdist
  S,\{\cdot\}.$. From \eqref{eq:por_Lip_lbound} it follows that
  \begin{equation}
    \inf_{x\in S}\biglip f(x)\ge c.
  \end{equation}
If $\nu(S)=0$ we have, trivially, $\locnorm df,{\wform\nu\mrest
  S.}.=0$. If $\nu(S)>0$ we note that for each derivation $D\in\wder\nu.$, by
locality of derivations, as $f=0$ on $S$, $\chi_SDf=0$. This implies
\begin{equation}
  \locnorm df,{\wform\nu\mrest
  S.}.=0.
\end{equation}
\end{proof}
\par We  need the following technical lemma because our approximation
schemes are designed to work in compact spaces.
\begin{lem}\label{lem:porous_completion}
  If $X$ is a complete metric space and $K\subset X$ is compact and
  \hbox{$c$-porous}, then there is a compact $Y$ 
  with $K\subset
  Y\subset X$ and such that $K$ is $\frac{2c}{3}$-porous in $Y$.
\end{lem}
\begin{proof}
  For $m\in\natural$ define 
  \begin{equation}
    \begin{aligned}
      \psi_m:K&\to\left(0,\frac{1}{m}\right]\\
        x&\mapsto\sup\left\{\dist x,y.:y\in W_c(x,K)\cap\ball x,\frac{1}{m}.\right\};
    \end{aligned}
  \end{equation}
note that $\psi_m$ is lower-semicontinuous as $\{x\in K:\psi_m(x)>r\}$
is open; this can be seen as follows: choose $y\in W_c(x,K)$ with $\dist x,y.>r$; for $x'\in K$
sufficiently close to $x$, $\dist x',y.>r$ and $y\in W_c(x',K)$. 
\par Let
$r_m=\min_K\psi_m$. As $K$ is compact, we can choose a finite
$\frac{r_m}{3}$-net $N_m\subset K$. For $x\in N_m$ choose 
\begin{equation}
  w_m(x)\in W_c(x,K)\cap\ball x,\frac{1}{m}.
\end{equation}
with $\dist x,w_m(x).>\frac{2r_m}{3}$. Let $W_m=\bigcup_{x\in
  N_m}\{w_m(x)\}$. Let $x'\in K$; choose $x\in N_m$ with $\dist
x,x'.\le\frac{r_m}{3}$. Then
\begin{equation}
  \dist x',w_m(x).>\frac{r_m}{3}
\end{equation} and
\begin{equation}
\begin{split}
  \dist w_m(x),x'.&\le\dist x,w_m(x).+\dist x,x'.\\
  &<\dist x,w_m(x).+\frac{1}{2}\dist x,w_m(x).\\
  &=\frac{3}{2}\dist x,w_m(x)..
\end{split}
\end{equation}
This implies that
\begin{equation}
  \frac{2c}{3}\dist w_m(x),x'.<c\dist x,w_m(x).<\dist w_m(x),K.;
\end{equation}
thus $w_m(x)\in W_{\frac{2c}{3}}(x',K)$ and $K$ is $\frac{2c}{3}$-porous
in
\begin{equation}
  Z=K\cup\bigcup_mW_m.
\end{equation}
We let $Y$ be the closure of $Z$ in $X$ and show that $Y$ is compact
by showing that $Z$ is totally bounded. Let $\epsi>0$. By compactness
of $K$, finitely many balls $B_1,\cdots, B_{M_\epsi}$ cover $K$; As
$B_1\cup\ldots\cup B_{M_\epsi}$ is an open neighbhourhood of $K$ and
as $W_m$ lies in a $\frac{1}{m}$-neighbourhood of $K$, there is $m_0$ such that
$m\ge m_0$ implies
\begin{equation}
  W_m\subset B_1\cup\cdots\cup B_{M_\epsi};
\end{equation}
note that $\bigcup_{m<m_0}W_m$ is finite so finitely many balls of
radius $\epsi$ are needed to cover it.
\end{proof}
\par We now prove Theorem \ref{thm:diff_char1}.
\begin{proof}[Proof of Theorem \ref{thm:diff_char1}]
  We first prove necessity. The first step is to show that $\mu$ is
  asymptotically doubling. Specifically, suppose that $S$ is a Borel $c$-porous
  subset of $X$ with $\mu(S)>0$. Then there is a compact $K\subset S$ with
  $\mu(K)>0$ and, by Lemma~\ref{lem:porous_completion}, a compact
  $Y\supset K$ in which $K$ is $\frac{2c}{3}$-porous. By Lemma
  \ref{lem:por_gap} $K\in\Gap_0(Y)$ and so, by Lemma~\ref{lem:ind_lip} and as
  $(Y,\mu\mrest Y)$ is a $\sigma$-differentiability space, $\mu(K)=0$,
  yielding a contradiction. This implies that $\mu(S)=0$.  As observed
  in \cite{bate_speight}, \cite[Thm.~3.6]{preiss_porosity} implies that
  if $\mu$ annihilates porous sets, then $\mu$ is asymptotically
  doubling. We can therefore find disjoint compact sets $K_\alpha$
  with $\mu(K_\alpha)>0$, $\mu(X\setminus\bigcup_\alpha K_\alpha)=0$
  and $K_\alpha$ of finite Assouad dimension.
  \par We now show that
  on a full $\mu$-measure Borel subset of $K_\alpha$, the set $K_\alpha$ is
  not $c$-porous for any $c\in(0,\frac{1}{2})$. Let
  \begin{equation}
    P_\alpha=\left\{x\in K_\alpha:\exists
    c\in\left(0,\frac{1}{2}\right):\text{$K_\alpha$ is $c$-porous at $x$}\right\};
  \end{equation}
note that
\begin{equation}
  O_{c,\alpha}(r)=\left\{x\in X:\exists y\in\ball x,r.:\setdist
  K_\alpha,\{y\}.>c\dist x,y.>0\right\}
\end{equation}
is open; then the set
\begin{equation}
  P_\alpha=\bigcup_{c\in\rational\cap(0,\frac{1}{2})}\bigcap_{n\in\natural}\bigcup_{r\in(0,\frac{1}{n})}O_{c,\alpha}(r)\cap
  K_\alpha
\end{equation}
is a $G_{\delta\sigma}$-{set} because each $K_\alpha$, being closed, is a $G_\delta$;
moreover, $P_\alpha$ is a countable union of the porous sets
\begin{equation}
  \bigcap_{n\in\natural}\bigcup_{r\in(0,\frac{1}{n})}O_{c,\alpha}(r)\cap
  K_\alpha\quad\text{($c\in\rational\cap(0,\frac{1}{2})$).}
\end{equation}
 Thus $\mu(P_\alpha)=0$.
%\par\textcolor{red}
\par{Let $S\in\Gap(X)$, and let $(f,\alpha,\beta)$ be as in the defining
property of $\Gap$, where $f$ is a real-valued Lipschitz function. }
Then \eqref{eq:npor_Lip_equal} implies that if
$S\cap(K_\alpha\setminus P_\alpha)\ne\emptyset$, then
$S\cap(K_\alpha\setminus P_\alpha)\in\Gap(K_\alpha)$. But by Lemma
\ref{lem:ind_lip}, $(K_\alpha,\mu\mrest K_\alpha)$ being a
\hbox{$\sigma$-diffe}\-rentiability space, $\mu(S\cap K_\alpha)=0$ so that
$\mu(S)=0$.
\par We now prove sufficiency. If $S$ is Borel and $c$-porous,
then $S\in\Gap_0(X)$ by Lemma~\ref{lem:por_gap}, so $S$ is $\mu$-null by
hypothesis. Then, as in the necessity argument (see the beginning of
this proof), $\mu$ is
$\sigma$-asymptotically doubling and we find disjoint compact sets $K_\alpha$ with $\mu(K_\alpha)>0$,
  $\mu(X\setminus\bigcup_\alpha K_\alpha)=0$ and $K_\alpha$ doubling,
  and hence of finite
  Assouad dimension; by Corollary~\ref{derbound} (which shows
  that a bound on the Assouad dimension implies that the module of
  Weaver derivations is finitely generated)  and Theorem~\ref{thm:free_dec} we can partition the $K_\alpha$ and
  assume that the module $\wder\mu\mrest K_\alpha.$ is free of rank
  $N_\alpha$. As in the necessity argument, $P_\alpha$
  is $\mu$-null so it suffices to show, by \eqref{eq:npor_Lip_equal},
  that each $(K_\alpha, \mu\mrest K_\alpha)$ is a differentiability space
  in order to conclude that $(X,\mu)$ is a $\sigma$-differentiability
  space.
  \par Let $S\in\Gap(\mu\mrest K_\alpha,K_\alpha)\subset\Gap(\mu,X)$; by Lemma
  \ref{lem:gap_incl} it contains a $\mu$-full measure Borel subset of
  type $\Gap(X)$, so $\mu\mrest K_\alpha(S)=0$. By definition of
  $\Gap$ sets, this implies that for every $f\colon K_\alpha\to\real$
  \begin{equation}\label{eq:diff=Lip_loc}
    \locnorm df,{\wform\mu\mrest K_\alpha.}.(x)=\biglip
    f(x)\quad\text{(for $\mu\mrest K_\alpha$-a.e.~$x$);}
  \end{equation}
by Corollary \ref{cor:pseudoduality} up to furthering
partitioning the $K_\alpha$, we can assume that there is a basis
$\{D_{\alpha,i}\}_{i=1}^{N_\alpha}$ of $\wder\mu\mrest K_\alpha.$ and there
are $1$-Lipschitz functions $\{g_{\alpha,j}\}_{j=1}^{N_\alpha}$ with
\begin{equation}
  D_{\alpha,i}g_{\alpha,j}=\delta_{ij}\chi_{K_\alpha}.
\end{equation}
Because of \eqref{eq:diff=Lip_loc}, for each $\epsi>0$
there are disjoint Borel $\{U_{\alpha,\beta}\}$ which are subsets of
$K_\alpha$, which satisfy
$\mu\left(K_\alpha\setminus\bigcup_\beta U_{\alpha,\beta}\right)=0$,
and such that there are derivations $\{D_{\alpha,\beta}\}$ with
\begin{align}
  \locnorm D_{\alpha,\beta},{\wder\mu\mrest U_{\alpha,\beta}.}.&=1\\
  D_{\alpha,\beta} f&\ge\biglip f-\epsi\quad\text{on $U_{\alpha,\beta}$;}
\end{align}
moreover, there are
$\{\lambda_{i,\alpha,\beta}\}_{i=1}^{N_\alpha}\subset\bborel
K_\alpha.$ with
\begin{equation}
  D_{\alpha,\beta}=\sum_{i=1}^{N_\alpha}\lambda_{i,\alpha,\beta}D_{\alpha,i};
\end{equation}
evaluating on the $g_{\alpha,j}$ gives 
\begin{equation}
  \|\lambda_{i,\alpha,\beta}\|_{L^\infty(\mu\mrest K_\alpha)}\le1,
\end{equation} so that
\begin{equation}
  |D_{\alpha,\beta}f|\le N_\alpha \max_{i=1,\ldots,N_\alpha}|D_{\alpha,i}f|.
\end{equation}
Thus 
\begin{equation}\label{eq:reverproof}
  \biglip f(x)\le N_\alpha\max_{i=1,\ldots,N_\alpha}|D_{\alpha,i}f|(x)\quad\text{(for $\mu\mrest K_\alpha$-a.e.~$x$),}
\end{equation}
and \cite[Thm.~5.9]{derivdiff} shows that $(K_\alpha, \mu\mrest
K_\alpha)$ is a differentiability space.
\end{proof}
\par We now prove Theorem \ref{thm:diff_char2}.
\begin{proof}[Proof of Theorem \ref{thm:diff_char2}]
{%\color{red}
{%\color{red}
We show that $(X,\mu)$ being a $\sigma$-differentiability space
   implies (1). In fact, assume that  $\locnorm
   df,{\wform\mu.}.\ne\biglip f$ on a set of positive measure. By
   Lemma~\ref{lem:local_lip_der} we have $\locnorm
   df,{\wform\mu.}.\le\biglip f$ $\mu$-a.e., and so $\mu$ would give
 positive measure to the set:
 \begin{equation}
   \label{eq:pf_diff_ch2_1}
   J=\left\{x: \locnorm df,{\wform\mu.}.(x) < \biglip f(x)\right\},
 \end{equation}
where with slight abuse of notation we use $\locnorm df,{\wform\mu.}.$ to
denote a Borel representative of the local norm of $df$. But we can
decompose $J$ as a countable union
\begin{equation}
  \label{eq:pf_diff_ch2_1_1}
  J = \bigcup_{\substack{0\le\beta<\alpha \\ \alpha,\beta\in\rational}}J_{\alpha,\beta}
\end{equation}
where
\begin{equation}
  \label{eq:pf_diff_ch2_1_2}
  J_{\alpha,\beta}=\left\{
    x:\locnorm df,{\wform\mu.}(x).\le\beta < \alpha < \biglip f(x)
  \right\}.
\end{equation}
Note that $J_{\alpha,\beta}\in\Gap(\mu,X)$ from the definition of the
class of sets $\Gap(\mu,X)$, and then, as $\mu(J)>0$, there
must be a pair $(\alpha_0,\beta_0)$ with
$\mu(J_{\alpha_0,\beta_0})>0$. By Lemma~\ref{lem:gap_incl} there is a
Borel set $J'_{\alpha_0,\beta_0}\subset J_{\alpha_0,\beta_0}$ with
$\mu(J_{\alpha_0,\beta_0} \setminus J_{\alpha_0,\beta_0}')=0$ and
$J_{\alpha_0,\beta_0}'\in\Gap(X)$. Thus we would have $\mu(J_{\alpha_0,\beta_0}')>0$,
 which would contradict
 Theorem~\ref{thm:diff_char1}. 
}
\par We now show that (1) implies (2). Recall from
Theorem~\ref{thm:weaver_fnorm_char} that $\locnorm df,{\wform\mu.}.$
is completely determined by the speed of $f$ on Alberti
representations. 
{%\color{red}
In particular, if on a set $A_\alpha$ we have:
\begin{equation}
  \label{eq:pf_diff_ch2_2}
  \begin{aligned}
    \locnorm df,{\wform\mu.}.&\ge\alpha\\
    (1+\varepsilon)\alpha&\ge\biglip f>0,
  \end{aligned}
\end{equation}
then the measure $\mu\mrest A_\alpha$ admits a $(1,1+\varepsilon)$-bi-Lipschitz
Alberti representation $\albrep\alpha,\sigma',\varepsilon.$ with 
\begin{equation}
  \label{eq:pf_diff_ch2_3}
  \text{$f$-speed}\ge\frac{\sigma'\alpha(1-\varepsilon)}{1+\varepsilon},
\end{equation}
for any $\sigma'\in(0,1)$. In fact, by
Theorem~\ref{thm:weaver_fnorm_char} $\mu\mrest A_\alpha$ cannot give
positive measure to any $\frags(X,f,\sigma'\alpha(1-\varepsilon/2))$-null
set, because otherwise we would have \begin{equation}\locnorm
df,{\wform\mu.}.\le\sigma'\alpha\end{equation} on a positive measure subset of $A_\alpha$. Thus the existence of the desired Alberti representation for
$\mu\mrest A_\alpha$ follows from Theorem~\ref{alberti_rep_prod}. But then such a
representation $\albrep\alpha,\sigma',\varepsilon.$ has 
\begin{equation}
  \label{eq:pf_diff_ch2_4}
\text{$f$-speed}  \ge\chi_{A_\alpha}\frac{\sigma'(1-\varepsilon)}{(1+\varepsilon)^2}\biglip f,
\end{equation}
and (2) follows by choosing $\sigma'$ and $\varepsilon$ so that
$\frac{\sigma'(1-\varepsilon)}{(1+\varepsilon)^2}\ge\sigma$, and applying the gluing
principle for Alberti representations (Theorem \ref{alb_glue}).}
\par We show that (2) implies (3). Consider a
$(1,1+\varepsilon)$-bi-Lipschitz Alberti representation
$\albrep.=(P,\nu)$ with $f$-speed $\ge\sigma\biglip f$. Then for
$P$-a.e.~$\gamma$ at $\gamma_{\#}^{-1}\nu_{\gamma}$-a.e.~$t$ we have:
\begin{equation}
  \label{eq:pf_diff_ch2_5}
  (f\circ\gamma)'(t) \ge \sigma\biglip f(\gamma(t)).
\end{equation}
For $\gamma_{\#}^{-1}\nu_{\gamma}$-a.e.~$t$ the point $t$ is a
Lebesgue density point of $\dom\gamma$, and thus we conclude from
the definition of the small Lipschitz constant $\smllip f$ that:
\begin{equation}
  \label{eq:pf_diff_ch2_6}
  (f\circ\gamma)'(t) \le (1+\varepsilon)\smllip f(\gamma(t)).
\end{equation}
We thus obtain:
\begin{equation}
  \label{eq:pf_diff_ch2_7}
  \sigma\biglip f(\gamma(t)) \le (1+\varepsilon)\smllip f(\gamma(t));
\end{equation}
because of the existence of $\albrep.$, the set of points $x$ of
the form $x=\gamma(t)$ such that~(\ref{eq:pf_diff_ch2_5}),
(\ref{eq:pf_diff_ch2_6}) and (\ref{eq:pf_diff_ch2_7}) hold has full
$\mu$-measure; we thus conclude that for $\mu$-a.e.~$x$:
\begin{equation}
  \label{eq:pf_diff_ch2_8}
  \sigma\biglip f(x)\le (1+\varepsilon)\smllip f(x),
\end{equation}
and the result follows letting $\sigma\nearrow 1$ and
$\varepsilon\searrow0$.
\par We now show that (3) implies that $(X,\mu)$ is a
$\sigma$-differentiability space. The first observation is that $\mu$
must be asymptotically doubling. Suppose the contrary, i.e.~that $\mu$
gave positive measure to a porous set $P$. Letting $f=d(\cdot,P)$
(i.e.~the distance function from $P$), we would find (using
Lemma~\ref{lem:por_gap}) an $\alpha>0$
such that:
\begin{equation}
  \label{eq:pf_diff_ch2_9}
  \biglip f\ge\alpha>0=\locnorm df,{\wform\mu.}.\quad(\text{$\mu\mrest P$-a.e.}).
\end{equation}
Applying Lemma~\ref{lem:gap_flat_func} (as $\locnorm
df,{\wform\mu.}.=0$ the assumption on the finite Assouad dimension is
not needed) we are in a position to apply Theorem~\ref{lip_ind}. As
the proof of Theorem~\ref{lip_ind} is quite technical, it has been 
deferred to Subsection \ref{subsec:constr-indep-lipsch}. Soon after
that proof, in
Remark~\ref{rem:liplipvio} we show how that argument actually produces
a Lipschitz function $\psi$ where $\biglip\psi\ne\smllip\psi$ on a set of
positive measure. This violates (3), and so we conclude that $\mu$ is
asymptotically doubling. We can now apply Keith's
Theorem~\ref{thm:keith} to conclude that $(X,\mu)$ is a
$\sigma$-differentiability space: note that even though
Keith
worked with doubling measures\footnote{Keith also introduced a notion of ``chunky measure'',
  but later it became clear that asymptotically doubling measures provide
  the right setting}, his arguments easily
generalize to asymptotically doubling measures (as usual, by taking
a countable partition such that on each set one has control on the doubling constant).}
\end{proof}
\par We now give an alternative proof of Lemma~\ref{partialderivatives} which relies on Theorem~\ref{thm:diff_char1}.
\begin{proof}[Alternative proof of Lemma~\ref{partialderivatives}]
  Let $(U,x)$ be an $n$-dimensional chart for the differentiability space $(X,\mu)$;
  then \eqref{eq:reverproof} shows that it is possible to find
  derivations $\{D_i\}_{i=1}^n\subset\wder\mu\mrest U.$ such that, for
  each $f\in\lipalg X.$,
  \begin{equation}
    \label{eq:altproof}
    \biglip f(x)\le n\max_{i=1,\ldots,n}|D_{i}f|(x)
  \end{equation} holds for $\mu\mrest U$-a.e.~$x$. We can then apply
  \cite[Lem.~6.20]{derivdiff} to conclude that the partial derivative
  operators $\frac{\partial}{\partial x^j}$  are derivations; note
  that even though in \cite{derivdiff} the measure $\mu$ is assumed
  doubling, this property of $\mu$ is used only to ensure that the
  Lebesgue's Differentiation Theorem holds for the measure
  $\mu$. However, the Lebesgue Differentiation Theorem holds for
  asymptotically doubling measures and so it is possible to apply the
  results in \cite{derivdiff}.
\end{proof}
\section{Technical tools}\label{sec:technical-tools}
\subsection{An approximation scheme}\label{subsec:an-appr-scheme}
The goal of this subsection is to sketch a combinatorical proof of Theorem
\ref{onedimapprox_multi}. The first step in this argument is the construction of a
cylinder out of the graph of a real-valued Lipschitz
function. %\textcolor{red}
{For
this we would like to say that if a subset $S\subset X$ meets each
fragment $\gamma$ belonging to a family ${\mathcal G}$ in a set {whose
$1$-dimensional Hausdorff measure is $0$}, then this property of $S$ is
intrinsic, i.e.~it still holds if we embedd $S$ isometrically into
another space $W$. Unfortunately, this is just wishful thinking as in
general ${\mathcal G}$ is a family of fragments which lie in $X$, not
in $W$. However, in the cases of interest the family ${\mathcal G}$ is
completely specified by conditions which involve Lipschitz
functions. Therefore our wishful thinking can be made precise in
Lemma~\ref{nullembedd_multi} relying on Definition~\ref{def:classes_of_fragments_approx}. }
\begin{defn}\label{def:classes_of_fragments_approx}
    For $\delta>0$, metric spaces $X$ and $W$, Lipschitz functions
    ${G\colon X\to W}$ and ${f\colon X\to\real^q}$, we define
\begin{multline}
  \frags(X,f,G,\delta,w,\alpha)=\Biggl\{\gamma\in\frags(X):
    \text{$(\langle w,f\rangle\circ\gamma)'(t)\ge\delta\metdiff
      G\circ\gamma(t)$}\\
    \text{and 
      $(f\circ\gamma)'(t)\in\bar\cone(w,\alpha)$
for
      $\lebmeas$-a.e. $t\in \dom\gamma$}\Biggr\}.
\end{multline}
\end{defn}
\begin{lem}\label{nullembedd_multi}
  Assume that $\psi:X\to Y$ is an isometric embedding and
  $S\subset X$. Then
   $S$ is $\frags(X,f,G,\delta,w,\alpha)$-null if and only if 
\begin{equation}  
\text{$\psi(S)$ is $\frags(Y,\tilde f,\tilde G,\delta,w,\alpha)$-null}
\end{equation}
where $\tilde f$
  and $\tilde G$ are any Lipschitz
  extensions of $f\circ\psi^{-1}$, $G\circ\psi^{-1}$.
\end{lem}
\begin{proof}
  Necessity is proven by contrapositive, assuming that there is a
  \begin{equation}
\gamma\in\frags(Y,\tilde f,\tilde G,\delta,w,\alpha)
  \end{equation}
  with $\hmeas
  1.(\gamma\cap\psi(S))>0$. It is then possible to find a compact
  $K'\subset\dom\gamma$ with $\gamma(K')\subset \psi(S)$ and $\hmeas
  1.\left((\gamma|K')\cap\psi(S)\right)>0$. Let
  $\tilde\gamma=\psi^{-1}\circ (\gamma|K')$. Then 
  \begin{align}
    (f\circ\tilde\gamma)'(t)&=(\tilde f\circ\gamma)'(t)\\
    \metdiff\tilde\gamma(t)&=\metdiff\gamma(t)\\
    \metdiff\tilde G\circ\gamma(t)&=\metdiff G\circ\tilde\gamma(t)
  \end{align}
at any $t\in K'$ which is a Lebesgue density point for $K'$. In particular,
\begin{equation}
\tilde\gamma\in\frags(X,f,G,\delta,w,\alpha)
\end{equation}
and $\hmeas 1.(\tilde\gamma\cap
S)>0$.
\par Sufficiency is proven by observing that the previous part of the
argument allows to identify $\frags(X,f,G,\delta,w,\alpha)$ with
$\frags(\psi(X),\tilde f,\tilde G,\delta,w,\alpha)$ via
$\gamma\mapsto\psi\circ\gamma$ because the metric differential and the
derivative of a Lipschitz function along a fragment are determined, at
a point $t$, by the behaviour of the fragment on a subset for which
$t$ is a Lebesgue density point. Sufficiency then follows because
$\frags(X,f,G,\delta,w,\alpha)\subset\frags(Y,\tilde f,\tilde
G,\delta,w,\alpha)$.
\end{proof}
\par We now introduce the definition of cylinder.
\begin{defn}
  Let $X$ be a compact metric space and $M>0$. The \textbf{cylinder}
  $\Cyl(X,M)$ is the compact metric space $X\times[0,M]$ with metric
  \begin{equation}
    \dist {(x_1,t_1)},{(x_2,t_2)}.=\max\left(\dist x_1,x_2.,|t_1-t_2|\right).
  \end{equation} The projection on the base $X$ will be denoted by
  $\pi$ and the projection on the axis $[0,M]$ by $\tau$. Note that if $X$ is geodesic, $\Cyl(X,M)$ is geodesic.
\end{defn}
We now reduce the general approximation problem to the
case in which $X$ is a cylinder and $\langle w,f\rangle$ is the height
function. 
\begin{lem}\label{con_red_multi}
  If $f\colon X\to\real^q$ is Lipschitz and $S\subset X$ is Borel and 
  \begin{equation}\frags(X,f,\delta,w,\alpha)\text{-null},
  \end{equation} %\textcolor{red}
{after embedding $X$ in a larger space and after shifting,
  rescaling and extending $f$}, we can assume that:
  \begin{enumerate}
  \item  The space $X$ is the cylinder $\Cyl(Y,M)$ for
    $Y=Z\times Q$, where $Z$ is a compact geodesic metric space, $Q$ is a compact
    rectangle in $\real^{q-1}$ and $M\le\diam X$.
  \item If $\tilde\tau\colon Y\to Q$ is the projection, 
    $(\pi_w^\perp\circ f,\langle w,f\rangle)=(\tilde\tau,\tau)$.
  \item If $\pi_Z\colon Y\to Z$ is the projection,
    $S$ is
    $\frags(\Cyl(Y,M),(\tilde\tau,\tau),\pi_Z,\delta,e_q,\alpha)$\nobreakdash-\hspace{0pt}null.
      \end{enumerate}
\end{lem}
\par Points of $\Cyl(Y,M)$ will be denoted by $(z,v,t)$ where $z\in
Z$, $v\in Q$ and $t\in[0,M]$.
\begin{proof}
  We can postcompose $f$ with an isometry of $\real^q$ so that the
  minimum of each component $f_i$ is $0$ and $w=e_q$, where
  $\{e_1,\ldots,e_q\}$ is the standard basis of $\real^q$. We rescale
  $f$ to be $1$-Lipschitz: this accounts for the factor $L$ in (\ref{eq:onedimapprox_multi_s2}).
  Considering a Kuratowski embedding $\Xi$ of $X$ in $l^\infty$ we obtain
  the isometric embedding:
  \begin{equation}
    \label{eq:con_red_multi_p1}
    \graph f\xrightarrow{\Xi\times f}l^\infty\oplus_{\infty}\real^q;
  \end{equation}
  we let $\tilde Y$ denote the closed convex
  hull of $X$ in $l^\infty\oplus_\infty\real^q$, which is compact
  (\cite[Thm.~3.25]{rudin-functional}). The set $Z$ is given by
  $\tilde Y\cap (l^\infty\times\{0\})$ and the set $Q$ by
  \begin{equation}
    Q=[0,\max f_1]\times\cdots\times[0,\max f_{N-1}]\subset\real^{q-1};
  \end{equation}
  the functions $(\tilde\tau,\tau)$ are induced by the projection
  $l^\infty\oplus_\infty\real^q\to\real^q$, and the set $S$ is
  replaced by the graph of $\Xi\times f$ restricted to $S$. The fact
  that this graph is $\frags(\Cyl(Y,M),
  (\tilde\tau,\tau),\pi_Z,\delta,e_q,\alpha)$-null follows from Lemma \ref{nullembedd_multi}.
\end{proof}
\par The next step is to cover $S$ by \emph{thin} strips. 
\begin{defn}
  Given a Lipschitz function $f\colon Y\to[0,M]$ and $h>0$ we define the \textbf{open
  strip of width $h$ above $f$} by
  \begin{equation}
    \strip f,h.=\left\{(y,t)\in\Cyl(Y,M): t\in(f(y),f(y)+h)\right\};
  \end{equation}
note that $\strip f,h.$ is an open set. The \textbf{lower and upper
hypersurfaces bounding $\strip f,h.$} are the closed sets:
\begin{align}
  \partial_-\strip f,h.&=\left\{(y,t)\in\Cyl(Y,M): t=f(y)\right\};\\
  \partial_+\strip f,h.&=\left\{(y,t)\in\Cyl(Y,M): t=f(y)+h\right\}.
\end{align}
\end{defn}
\begin{lem}\label{strip_cov_multi}
  Assume that the compact set $S\subset\Cyl(Y,M)$ is
$$\text{$\frags(\Cyl(Y,M),\allowbreak(\tilde\tau,\allowbreak\tau),\allowbreak\pi_Z,
\allowbreak\delta,\allowbreak e_q,\allowbreak\alpha)$-null;}$$
 then for each $n\in\natural$
the set  $S$ can be covered by $M(n)$ open strips 
  \begin{equation}\left\{\strip
  f_i,2\frac{\delta+\cot\alpha+1}{n}.\right\}_{i=1}^{M(n)}
\end{equation}
where the $f_i$ are
  $1$-Lipschitz with respect to the distance 
    \begin{equation}
    d_{\delta,\alpha}((z_1,v_1,t_1),(z_2,v_2,t_2))=\delta\max\left(\dist
    z_1,z_2.,|t_1-t_2|\right)+\|v_1-
    v_2\|_2\cot\alpha,
  \end{equation}and 
  \begin{equation}
    \lim_{n\to\infty}\frac{M(n)}{n}=0.
  \end{equation}
\end{lem}
\begin{proof}
    Recall that a \textbf{$\frac{1}{n}$-net} in a metric space is a maximal set
  of points which are separated by a distance
  $\ge\frac{1}{n}$. As $Y$ and $[0,M]$ are compact, we obtain finite
  $\frac{1}{n}$-nets ${\mathscr{N}}_Y\subset Y$ 
  and ${\mathscr{N}}_{[0,M]}\subset[0,M]$. On the set
  \begin{equation}
    {\mathscr{N}}_{\Cyl(Y,M)}^S=\left\{(y,t)\in{\mathscr{N}}_{\Cyl(Y,M)}:
    \ball {(y,t)},\frac{1}{n}.\cap S\ne\emptyset\right\}
  \end{equation}
  we define the partial order
$\preceq$:
\begin{multline}
  (z_1,v_1,t_1)\preceq(z_2,v_2,t_2)\Longleftrightarrow
  t_2-t_1\ge\delta\dist z_1,z_2.\\\quad\text{and}\quad(t_2-t_1)\tan\alpha\ge\|v_1-v_2\|_2,
\end{multline}
and let $M(n)$ denote the length of a maximal chain in
$({\mathscr{N}}_{\Cyl(Y,M)}^S,\preceq)$ (note that the nets depend on
$n$). We prove that $M(n)=o(n)$ arguing by contradiction, that is,
assuming that for some  $C>0$ there are naturals $n\to\infty$ with
$M(n)\ge Cn$. We denote by
$\{(z_i^{(n)},v_i^{(n)},t_i^{(n)})\}_{i=1}^{M(n)}$ a maximal chain with the
$t_i^{(n)}$ in increasing order. We construct a bi-Lipschitz
path $\gamma_n:[0,M]\to\Cyl(Y,M)$ as follows:
\begin{itemize}
\item The restriction of $\gamma_n$ to $[t_i^{(n)},t_{i+1}^{(n)}]$,
is a constant speed geodesic joining
$(z_i^{(n)},\allowbreak v_i^{(n)},\allowbreak t_i^{(n)})$ to
$(z_{i+1}^{(n)},v_{i+1}^{(n)},t_{i+1}^{(n)})$.
\item The restrictions of $\gamma_n$ to
$[0,t_1^{(n)}]$ and $[t_{M(n)}^{(n)},M]$ are, respectively, constant speed geodesics
jointing $(z_1^{(n)},v_1^{(n)},0)$ to
${(z_1^{(n)},v_1^{(n)},t_1^{(n)})}$, and ${(z_{M(n)}^{(n)},\allowbreak v_{M(n)}^{(n)},\allowbreak t_{M(n)}^{(n)})}$ to
$(z_{M(n)}^{(n)},v_{M(n)}^{(n)},M)$. 
\end{itemize}
The path $\gamma_n$ is
$[1,\max\left(\frac{1}{\delta},\tan\alpha\right)]$-bi-Lipschitz and
belongs to
%\textcolor{red}
{$\frags(\Cyl(Y,M),(\tilde\tau,\tau),\pi_Z,\delta,e_q,\alpha)$.}
Moreover, letting
\begin{multline}
  \label{eq:strip_cov_multi_p1}
  K_n=\left[\max\left(0,t_1^{(n)}-\frac{\min(\delta,\cot\alpha)}{4n}\right),t_1^{(n)}+\frac{\min(\delta,\cot\alpha)}{4n}\right]\\
  \cup\bigcup_{2\le i\le
    M(n)-1}\left[t_i^{(n)}-\frac{\min(\delta,\cot\alpha)}{4n},t_i^{(n)}+\frac{\min(\delta,\cot\alpha)}{4n}\right]\\
  \cup\left[t_{M(n)}^{(n)}-\frac{\min(\delta,\cot\alpha)}{4n},\min\left(M,t_{M(n)}^{(n)}+\frac{\min(\delta,\cot\alpha)}{4n}\right)\right],
\end{multline}
we note that $\gamma_n(K_n)$ lies within distance $\frac{3}{2n}$ from
$S$. We can pass to converging subsequences
$\gamma_n\to \gamma$ and $K_n\to K$, so that we have the lower bound
\begin{equation}
  \lebmeas(K)\ge\limsup_{n\to\infty}\lebmeas(K_n)\ge \left(C-\frac{1}{n}\right)\frac{\min(\delta,\cot\alpha)}{2}
\end{equation}
and $\gamma(K)$ lies in $S$, 
contradicting that $S$ is
$\frags(\Cyl(Y,M),(\tilde\tau,\tau),\pi_Z,\delta,e_q,\alpha)$-null.
\par 
By Mirsky's Lemma (dual to Dilworth's Lemma, \cite{mirsky_dil}),
there are $M(n)$ antichains
${\mathscr{A}}_1,\ldots,{\mathscr{A}}_{M(n)}$ covering
${\mathscr{N}}_{\Cyl(Y,M)}^S$. As in an antichain no two
elements are comparable with respect to the order,
each ${\mathscr{A}}_p$ can be regarded as the graph of a function
$g_p\colon\pi_Z\times\tilde\tau({\mathscr{A}}_p)\to\real$ which is
$1$-Lipschitz with respect to the distance $d_{\delta,\alpha}$
(defined in (\ref{eq:onedimapprox_multi_s2})); the functions can be
extended to $1$-Lipschitz functions
$g_p:Y\to\real$ and from the definition of
${\mathscr{N}}_{\Cyl(Y,M)}^S$ we conclude that
\begin{equation}
  S\subset\bigcup_{p=1}^{M(n)}\strip g_p-\frac{\delta+\cot\alpha+1}{n},2\frac{\delta+\cot\alpha+1}{n}..
\end{equation}
\end{proof}
\par We now replace the strips given by Lemma \ref{strip_cov_multi} by
disjoint ones which come with an order structure: the cost to pay is
to make the strips slightly bigger and allow to cover a full measure
subset of $S$, but not necessarily the whole of $S$.
\begin{lem}\label{strip_cov_disj_multi}
  Assume that $S\subset\bigcup_{i=1}^N\strip f_i,h.$ ($h>0$) where the
  functions $f_i$ are $(1,d_{\delta,\alpha})$-Lipschitz. Then there are
  $(1,d_{\delta,\alpha})$-Lipschitz functions $\{g_i\}_{i=1}^N$
  and
  ${\lambda_i\in(1,\frac{3}{2})}$ such that  
  \begin{equation}\label{eq:strip_cov_disj_multi_s1}
    g_1\le g_2\le\cdots\le g_N,
  \end{equation}
\begin{equation}\label{eq:strip_cov_disj_multi_s2}
  \strip g_i,\lambda_i h.\cap\strip g_j,\lambda_j
  h.=\emptyset\quad\text{(for $i\ne j$),}
\end{equation}
and
\begin{equation}\label{eq:strip_cov_disj_multi_s3}
  \mu\left(S\setminus\bigcup_{i=1}^N\strip g_i,\lambda_i h.\right)=0.
\end{equation}
\end{lem}
\begin{proof}
  We first show that there are $(1,d_{\delta,\alpha})$-Lipschitz
  functions $\{\tilde f_i\}_{i=1}^N$ which satisfy:
  \begin{align}
    \label{eq:strip_cov_disj_multi_p1}
    \tilde f_1\le\tilde f_2\le\cdots\le\tilde f_N,\\
    \label{eq:strip_cov_disj_multi_p2}
      \bigcup_{i=1}^N\strip f_i,h. =\bigcup_{i=1}^N\strip \tilde f_i,h..
    \end{align}
    The value of $\tilde f_i$ at $x$ is that $i$-th of the values
    $\{f_j(x)\}_{j=1}^N$ after they have been put in nondecreasing
    order; specifically, if $\Omega_N(i)$ denotes the set of subsets
    of $\{1,\ldots,N\}$ of cardinality $i$, we have:
    \begin{equation}
      \label{eq:strip_cov_disj_multi_p3}
      \tilde f_i(x)=\min_{S\in\Omega_N(i)}\max_{j\in S}f_j(x);
    \end{equation}
    from (\ref{eq:strip_cov_disj_multi_p3}) it follows immediately
    that the $\{f_j(x)\}_{j=1}^N$ are
    $(1,d_{\delta,\alpha})$-Lipschitz and that
    (\ref{eq:strip_cov_disj_multi_p1}) holds.
    For each $x$, $\tilde f_i$ equals $f_j(x)$ for some $j$ and,
    vice versa, $f_i(x)$ equals some $\tilde f_j(x)$ for some $j$: this
    implies (\ref{eq:strip_cov_disj_multi_p2}).
    \par We let $g_1=\tilde f_1$ and choose $\lambda_1\in(1,\frac{3}{2})$ such that
  \begin{equation}\label{eq:strip_cov_disj_multi_p4}
    \mu\left(\partial_+\strip f_1,\lambda_1 h.\right)=0,
  \end{equation}  this being possible since $\mu$ is a finite measure and $(1,\frac{3}{2})$
is uncountable. For ${j\in\{2,\ldots,N\}}$ we let
$g_j=\max(g_{j-1}+\lambda_{j-1}h,\tilde f_j)$ and choose
$\lambda_j\in(1,\frac{3}{2})$ such that 
  \begin{equation}
    \mu\left(\partial_+\strip g_j,\lambda_j h.\right)=0.
  \end{equation}
  From the definition of the $\{g_j\}_{j=1}^N$ we get that
  (\ref{eq:strip_cov_disj_multi_s1}) and
  (\ref{eq:strip_cov_disj_multi_s2}) hold; moreover, for each $j$ we
  have:
\begin{equation}\label{eq:strip_cov_disj_multi_p5} 
  \strip f_j,h.\subset\bigcup_{i=1}^j\strip
  g_i,\lambda_ih.\cup\bigcup_{i=1}^{j-1}
\partial_+\strip
  g_i,\lambda_ih.,
\end{equation}
from which (\ref{eq:strip_cov_disj_multi_s3}) follows.
\end{proof}
\par We can now prove Theorem \ref{onedimapprox_multi}.
\begin{proof}[Proof of Theorem \ref{onedimapprox_multi}]
  By Lemma~\ref{con_red_multi} we can replace $X$ by a cylinder and approximate
  the Lipschitz function
  $\tau$, which is defined at point (2) of Lemma~\ref{con_red_multi}. By Lemma \ref{strip_cov_disj_multi} we can cover $S$, up to
  a $\mu$-null set, by
$M(n)$ disjoint strips: \begin{equation}\left\{\strip
g_i,2\lambda_i\frac{\delta+\cot\alpha+1}{n}.\right\}_{i=1}^{M(n)},
\end{equation}
and one can define a total order relation between the
hypersurfaces
bounding these strips:
\begin{equation}
\begin{split}
  &\partial_-\strip g_i,2\lambda_i\frac{\delta+\cot\alpha+1}{n}.\prec
  \partial_+\strip g_i,2\lambda_i\frac{\delta+\cot\alpha+1}{n}.\\
 \prec&\partial_-\strip g_{i+1},2\lambda_{i+1}\frac{\delta+\cot\alpha+1}{n}.\prec
  \partial_+\strip g_{i+1},2\lambda_{i+1}\frac{\delta+\cot\alpha+1}{n}..
\end{split}
\end{equation}
We let $\mathcal{T}_n$ denote the union of the strips and define
\begin{equation}
  \tau_n(z,v,t)=\int_0^t\chi_{{\mathcal{T}}_n^c}(z,v,s)\,ds,
\end{equation}
which satisfies
\begin{equation}
  \|\tau-\tau_n\|_\infty\le
  3(1+\delta+\cot\alpha)\frac{M(n)}{n}=O(1/n).
\end{equation}
\par We need to show that $\tau_n$ is globally $1$-Lipschitz with
  respect to the distance:
  \begin{equation}
    D\left((z_1,v_1,t_1),(z_2,v_2,t_2)\right)=\max\left(|t_1-t_2|,d_{\delta,\alpha}\left((z_1,v_1),(z_2,v_2)\right)\right),
  \end{equation}
  and that it is, in each strip, $1$-Lipschitz with respect to $d_{\delta,\alpha}$.
There are different cases to consider depending on the relative
position of two points $(z_1,v_1,t_1),(z_2,v_2,t_2)$: we just consider
the case in which they are separated by a hypersurface. We can choose
the separating hypersurface to be minimal with respect to the order
$\prec$, and assume that it is  $\partial_-\strip
  g_j,2\lambda_j\frac{\delta+\cot\alpha+1}{n}.$. We let $
  \eta_j=\sum_{i=1}^{j-1}2\lambda_i\frac{\delta+\cot\alpha+1}{n}$ and
  assume that $t_1\le g_j(z_1,v_1)$ and $t_2\ge g_j(z_2,v_2)$; we have:
\begin{align}
  \tau_n(z_1,v_1,t_1)&=t_1-\eta_{j-1};\\
  \tau_n(z_2,v_2,t_2)&\in[g_j(z_2,v_2)-\eta_{j-1},t_2-\eta_{j-1}],
\end{align}
which implies
\begin{equation}
  \tau_n(z_2,v_2,t_2)\le t_2-\eta_{j-1}=t_2-t_1+t_1-\eta_{j-1}\le|t_2-t_1|+\tau_n(z_1,v_1,t_1);
\end{equation}
moreover,
\begin{equation}
  \begin{split}
    \tau_n(z_1,v_1,t_1)=t_1-\eta_{j-1}&\le
    g_j(z_1,v_1)-\eta_{j-1}\\&=g_j(z_2,v_2)-\eta_{j-1}+g_j(z_1,v_1)-g_j(z_2,v_2)
    \\&\le\tau_n(z_2,v_2,t_2)+d_{\delta,\alpha}(z_1,z_2),
  \end{split}
\end{equation}
which implies that
\begin{equation}
  \left|\tau_n(z_1,v_1,t_1)-\tau_n(z_2,v_2,t_2)\right|\le D\left((z_1,v_1,t_1),(z_2,v_2,t_2)\right).
\end{equation}
Finally, inside each strip $\strip
  g_j,2\lambda_j\frac{\delta+\cot\alpha+1}{n}.$ the function $\tau_n$
  differs from $g_j\circ\pi_Y$ by a constant, and is hence $(1,d_{\delta,\alpha})$-Lipschitz.
\end{proof}
\subsection{Dimensional bounds and tangent
  cones}\label{subsec:dimens-bounds-tang}
%\textcolor{red}
{The goal of this subsection is to prove Theorem~\ref{alberti_blow_up}
and Corollary~\ref{derbound}. The proof of Theorem~\ref{alberti_blow_up} shows how
to use Alberti representations in the $f$-direction of $N$ independent cones to produce lines in blow-ups of
doubling metric spaces; along these lines the blow-up $g$ of the Lipschitz
function $f$ has constant derivative, and one can then claim surjectivity
of $g$ (whose target is $\real^N$). Corollary~\ref{derbound}
essentially shows
that the Assouad dimension bounds the number of independent Weaver
derivations. }
\par We first recall the definition of
blow-ups of Lipschitz functions on metric spaces, following the approach
of isometrically embedding all the pointed metric spaces in a common
proper metric space \cite[Sec.~2.2]{keith-modulus}. 
\par \begin{defn}\label{defn_blow_up} Recall that $\frac{1}{t}X$
  denotes the metric space $X$ where the metric has been rescaled by
  the factor $\frac{1}{t}$.
A \textbf{blow-up of a metric
    space $X$ at a point $p$} is a complete pointed metric space
  $(Y,q)$ such that there is a sequence $t_n\to0^+$ and
  $\left(\frac{1}{t_n}X,p\right)\to(Y,q)$ in the Gromov-Hausdorff
  sense.  The class of blow-ups at $p$ is denoted by
  $\tang(X,p)$.
\end{defn}
\par To show the existence of blow-ups the following notion of finite
dimensionality for metric spaces is useful.
\begin{defn}\label{defn:assouad_dim}
  A metric space $X$ is \textbf{doubling} if there is a constant $C$ such that
 every set of diameter $\le N$ can be covered by at most $C$ sets of
  diameter $\le N/2$. By induction it follows that $X$ admits a covering
  function $C(\epsi)=C\epsi^{-D}$ where any set of diameter $\le N$ can be covered by
  at most $C(\epsi)$ sets of diameter $\le\epsi N$. The minimal
  exponent $D$ is called the \textbf{Assouad dimension} of $X$.
\end{defn}
\begin{rem}\label{rem:selection}
If $X$ is doubling, given a sequence with
  $t_n\to0^+$, it is possible to find by a standard compactness
  argument (see~\cite{tyson_mackay_conf} or \cite[Sec.~5]{keith04}) a subsequence $(t_{n_k})$ such that the sequence
$\left(\frac{1}{t_{n_k}}X,p\right)$ converges in the Gromov-Hausdorff sense. In
this case the spaces
$\left(\frac{1}{t_{n_k}}X,p\right)$ and $(Y,q)$ can be isometrically
embedded\footnote{Mapping basepoints to basepoints} into a proper metric
space $(Z,z)$  so that, for each $R>0$,
\begin{align}
  \lim_{k\to\infty}
\sup_{y\in \ball
    z,R.\cap Y}\setdist
   \frac{1}{t_{n_k}}X,\{y\}.&=0\\
  \lim_{k\to\infty}\sup_{x\in \ball
    z,R.\cap  \frac{1}{t_{n_k}}X}\setdist 
   Y,\{x\}.&=0.
\end{align}
\par In particular,
any point $q'\in Y$ can be \textbf{approximated} by a
sequence $p'_{n_k}\in\frac{1}{t_{n_k}}X$ such that $p'_{n_k}\to q'$ in $Z$.
\end{rem}
\par We now define blow-ups of Lipschitz functions.
\begin{defn}\label{defn:blow_up_lip}
A \textbf{blow-up of a Lipschitz function $f\colon X\to\real^Q$ at a point $p$} is a
triple $(Y,q,g)$ where:
\begin{enumerate}
\item The space $(Y,q)$ is in $\tang(X,p)$ and is a limit realized by a sequence $(t_n)_n$
  of scaling factors.
\item The function $g:Y\to\real^Q$ is Lipschitz with $g(q)=0$.
\item If $p'_n$ approximates $q'$, 
\begin{equation}
  \lim_{n\to\infty}\frac{f(p'_n)-f(p)}{t_n}=g(q').
\end{equation}
\end{enumerate}
The class of all blow-ups of $f$ at $p$ will be denoted by
$\tang(X,p,f)$. By an Ascoli-Arzel\`a argument, if $X$ is doubling, $\tang(X,p,f)\ne\emptyset$.
\end{defn}
We now prove Theorem \ref{alberti_blow_up}: the assumption on the
completeness of $\mu$ is required to ensure that Suslin sets are
$\mu$-measurable. Note that any Radon measure can be extended so that
sets in the $\sigma$-algebra generated by Suslin sets are measurable.
\begin{thm}\label{alberti_blow_up}
Let $\mu$ be a complete Radon measure on a metric space $X$ which has
 finite Assouad dimension $D$. Consider a Lipschitz function $f\colon X\to \real^N$, points $\{v_i\}_{i=1}^N\subset{\mathbb S}^{N-1}$,
and constants $\{\alpha_i\}_{i=1}^N\subset(0,\pi/2)$ and $\delta>0$.
Suppose that $\mu$ admits Lipschitz Alberti representations $\{\albrep
i.\}_{i=1}^N$ such that:
\begin{enumerate}
\item The Alberti representation $\albrep i.$ is in the $f$-direction of $\cone(v_i,\alpha_i)$
  with $\langle v_i, f\rangle$\nobreakdash-\hspace{0pt}spe\-ed $\ge\delta$.
\item For some $\theta>0$ the cone fields $\cone(v_i,\alpha_i+\theta)$ are independent.
\end{enumerate}
Then there is a $\mu$-full measure Borel subset
$U\subset X$ such that, for each $p\in U$ and for each blow-up $(Y,q,g)\in\tang (X,p,
f)$, the function $g:Y\to\real^N$ is surjective.
\end{thm}
\begin{proof}
%\textcolor{red}
{The idea of the proof is to ``blow-up'' fragments in $X$ in order to
produce lines in $Y$ along which $g$ has constant derivative. With a
bit of additional care one can argue that starting at $q$ and fixing traveling
times $(\tau_1,\ldots,\tau_N)$ ($\tau_i<0$ implies that one moves
``backwards'') it is possible to move in succession
along $N$ distinct lines, for a time $\tau_i$ along the $i$-th line, and
where  the 
derivative of $g$ equals a constant vector $w_i$ along the $i$-th line. Moreover,
one shows that the vectors
$\{w_i\}_{i=1}^N$ are independent, and this immediately implies that $g$ is
surjective. }
\par  We assume that the Alberti representations $\{\albrep i.\}_{i=1}^N$
  are $C$-Lipschitz.  Given a $C$-Lipschitz Alberti representation $\albrep.$ in the
  $f$-direction of a cone field $\cone(v,\alpha)$ with $\langle
  v,f\rangle$-speed $\ge\delta$,  we define the set
  $\frags(p,R,\epsi,\albrep.)$ of fragments $\gamma$ that satisfy the
  following conditions at $p$:
  \begin{enumerate}
  \item The fragment $\gamma$ is a $C$-Lipschitz path fragment in the $f$-direction
    of $\cone(v,\alpha)$ with $f$-speed $\ge\delta$, and such that
    $\lebmeas(\dom\gamma)>0$, and $0$ is a density point of $\dom
    \gamma$.
  \item The metric differential $\metdiff\gamma(0)$ and
    the derivative
    $(f\circ\gamma)'(0)\in\cone(v,\alpha)$ exist, $\gamma(0)=p$ and
  \begin{equation}
    \left|(f\circ\gamma)'(0)\right|\ge\delta\metdiff\gamma(0).
  \end{equation}
\item For each  $ r\in(0,R)$ one has $\lebmeas(\dom\gamma\cap\ball 0,r.)\ge
2r(1-\epsi)$.
\item For each $ t\in\dom\gamma\cap\ball 0,r.$  one has
\begin{equation}
  \left|f(\gamma(t))-f(p)-(f\circ\gamma)'(0)t\right|\le\epsi |t|.
\end{equation}
\item For all $ t,s\in\dom\gamma\cap\ball 0,r.$ one has\footnote{This is 
the approximate continuity of the metric differential (\cite[Thm.~3.3]{ambrosio-rectifiability})} 
\begin{equation}
  \left|\dist\gamma(t),\gamma(s).-\metdiff\gamma(0)|t-s|\right|\le
\epsi(|t|+|s|).
\end{equation}
  \end{enumerate}
Let $F(R_1,\epsi_1,\albrep 1.)=\{p\in U:\frags(p,R_1,\epsi_1,\albrep
1.) \ne\emptyset\}$; this set is Suslin and, as $\mu$ admits the
Alberti representation $\albrep 1.$, as $R_1\searrow0$ the sets
\{$F(R_1,\epsi_1,\albrep 1.)\}$ increase to a full
measure subset of $U$. By the Jankoff measurable selection
principle \cite[Thm.~6.9.1]{bogachev_measure}
 there is
a selection function $\Gamma(R_1,\epsi_1,\albrep 1.)$ associating to
$p\in F(R_1,\epsi_1,\albrep 1.)$ a fragment satisfying the conditions
(1)--(5) above. This choice is
measurable in the sense that
\begin{align}
  \varphi(R_1,\epsi_1,\albrep 1.)(p)&=(f\circ \Gamma(R_1,\epsi_1,\albrep 1.))'(0)\\
  \psi(R_1,\epsi_1,\albrep 1.)(p)&=\metdiff \Gamma(R_1,\epsi_1,\albrep 1.)(0)
\end{align}
are measurable functions\footnote{With respect to the $\sigma$-algebra
generated by Suslin sets}. By
Lusin's Theorem~\cite[Thm.~7.1.13]{bogachev_measure} there are compact sets 
\begin{equation}
F_c(R_1,\epsi_1,\albrep
1.;\tau_1)\subset F(R_1,\epsi_1,\albrep 1.)
\end{equation}
such that:
\begin{enumerate}
\item for all $p,q\in F_c(R_1,\epsi_1,\albrep
1.;\tau_1)$ with $\dist p,q.\le C\tau_1$ one has
\begin{align}
\left|\varphi(R_1,\epsi_1,\albrep 1.)(p)-\varphi(R_1,\epsi_1,\albrep
1.)(q)\right|&<\epsi_1;\\
\left|\psi(R_1,\epsi_1,\albrep 1.)(p)-\psi(R_1,\epsi_1,\albrep
1.)(q)\right|&<\epsi_1;
\end{align}
\item as $\tau_1\to 0$ one has
\begin{equation}
\mu(F(R_1,\epsi_1,\albrep 1.)\setminus F_c(R_1,\epsi_1,\albrep
1.;\tau_1))\to0.
\end{equation}
\end{enumerate}
The construction of the sets $F(\cdots)$, $F_c(\cdots)$ and of the selection functions
$\Gamma(\cdots)$ proceeds inductively in the
following way: assuming for $k<N$ that
\begin{equation}F(R_1,\epsi_1,\albrep
1.;\tau_1;\cdots;R_k,\epsi_k,\albrep k.),
\end{equation}
\begin{equation}
\Gamma(R_1,\epsi_1,\albrep
1.;\tau_1;\cdots;R_k,\epsi_k,\albrep k.),
\end{equation}
 and \begin{equation}
F_c(R_1,\epsi_1,\albrep
1.;\tau_1;\cdots;R_k,\epsi_k,\albrep k.;\tau_k)
 \end{equation}
 have been constructed,
let 
\begin{equation}
F(R_1,\epsi_1,\albrep
1.;\tau_1;\cdots;R_k,\epsi_k,\albrep
k.;\tau_k;R_{k+1},\epsi_{k+1},\albrep k+1.)
\end{equation}
 be the set of those 
\begin{equation}
p\in F_c(R_1,\epsi_1,\albrep
1.;\tau_1;\cdots;R_k,\epsi_k,\albrep k.;\tau_k)\cap
F(R_{k+1},\epsi_{k+1},\albrep k+1.)
\end{equation}
such that there is a fragment
$\gamma\in\frags(p,R_{k+1},\epsi_{k+1},\albrep k+1.)$ satisfying: 
\begin{multline}\label{rec_dens}
\lebmeas(\gamma^{-1}(F_c(R_1,\epsi_1,\albrep
1.;\tau_1;\cdots;R_k,\epsi_k,\albrep k.;\tau_k))\cap\ball 0,r.)\\\ge
2r(1-\epsi_{k+1})   \quad(\forall r\le R_{k+1}).
\end{multline}
These sets are Suslin measurable and, as $R_{k+1}\to0$, increase to a full
measure subset of $F_c(R_1,\epsi_1,\albrep
1.;\tau_1;\cdots;R_k,\epsi_k,\albrep k.;\tau_k)$. By Jankoff's measurable
selection principle there is a selection function
$\Gamma(R_{k+1},\epsi_{k+1},\albrep k+1.)$ associating to
\begin{equation}
p\in
F(R_1,\epsi_1,\albrep 1.;\tau_1;\cdots;R_k,\epsi_k,\albrep
k.;\tau_k;R_{k+1},\epsi_{k+1},\albrep k+1.)
\end{equation}
 a fragment satisfying the conditions
 above. This choice is measurable in the sense that
\begin{align}
  \varphi(R_{k+1},\epsi_{k+1},\albrep k+1.)(p)&=(f\circ
  \Gamma(R_{k+1},\epsi_{k+1},\albrep k+1.))'(0)\\ \psi(R_{k+1},\epsi_{k+1},\albrep
  k+1.)(p)&=\metdiff \Gamma(R_{k+1},\epsi_{k+1},\albrep k+1.)(0)
\end{align}
are measurable functions. By
Lusin's theorem there are compact sets 
\begin{multline}
F_c(R_1,\epsi_1,\albrep
1.;\tau_1;\cdots;R_k,\epsi_k,\albrep
k.;\tau_k;R_{k+1},\epsi_{k+1},\albrep k+1.;\tau_{k+1})\\\subset   
F(R_1,\epsi_1,\albrep
1.;\tau_1;\cdots;R_k,\epsi_k,\albrep k.;\tau_k;R_{k+1},\epsi_{k+1},\albrep
k+1.)
\end{multline}
such that, whenever $p$, $q$ satisfy
$\dist p,q.\le C\tau_{k+1}$, one has:
\begin{align}
\left|\varphi(R_{k+1},\epsi_{k+1},\albrep {k+1}.)(p)-\varphi(R_{k+1},\epsi_{k+1},\albrep
{k+1}.)(q)\right|&<\epsi_{k+1};\\
\left|\psi(R_{k+1},\epsi_{k+1},\albrep {k+1}.)(p)-\psi(R_{k+1},\epsi_{k+1},\albrep
{k+1}.)(q)\right|&<\epsi_{k+1};
\end{align} as $\tau_{k+1}\to0$ we can assume that
\begin{equation}
F_c(R_1,\epsi_1,\albrep
1.;\tau_1;\cdots;R_k,\epsi_k,\albrep
k.;\tau_k;R_{k+1},\epsi_{k+1},\albrep k+1.;\tau_{k+1})
\end{equation}
increases to a
full measure subset of 
\begin{equation}
F(R_1,\epsi_1,\albrep
1.;\tau_1;\cdots;R_k,\epsi_k,\albrep k.;\tau_k;R_{k+1},\epsi_{k+1},\albrep
k+1.).
\end{equation}
Having fixed $c>0$, it is possible to choose $(R^{(i)}_h), (\tau^{(i)}_h)$ such
that:
\begin{itemize}
\item The sequences $(R^{(i)}_h), (\tau^{(i)}_h)$ are decreasing in
  both $h\in\natural$ and ${i\in\{1,\allowbreak\ldots,\allowbreak N\}}$.
\item One has that $\lim_{h\to\infty}R^{(i)}_h=\lim_{h\to\infty}\tau^{(i)}_h=0$.
\item Letting \begin{equation}
  F_h=F_c(R^{(1)}_h,\frac{1}{h},\albrep 1.;\tau^{(1)}_h;\cdots;
R^{(N)}_h,\frac{N}{h},\albrep N.;\tau^{(N)}_h),
\end{equation} the Borel set $V=\bigcap_hF_h$ satisfies
$\mu\left(X\setminus V\right)<c$.
\end{itemize}
\par We will simplify the notation for selectors and derivatives
writing $\Gamma^{(i)}_h$, $\varphi^{(i)}_h$, and $\psi^{(i)}_h$.
Let $p\in V$ and \begin{equation}
  \left(\frac{1}{t_h}X,p,\frac{f-f(p)}{t_h}\right)\to(Y,q,g);
\end{equation} consider a subsequence $t_h$ such that
$\lim_{h\to\infty}\frac{\tau^{(N)}_h}{t_h}=\infty$. By passing to
further subsequences we can assume that:
\begin{itemize}
\item The $\varphi^{(i)}_h(p)$ converge to
  $w_i\in\cone(v_i,\alpha_i+\theta)\setminus\ball 0,\delta.$, implying
  that the $\{w_i\}_{i=1}^N$ are independent.
\item The $\psi^{(i)}_h(p)$ converge to $\delta_i\in [\delta, C]$.
\end{itemize} In particular, the functions
\begin{equation}
  \Gamma^{(N)}_h(p):\dom \Gamma^{(N)}_h(p)\to X
\end{equation} are $C$-Lipschitz, and we define 
\begin{equation}
  \tilde\Gamma^{(N)}_h(p):\frac {1}{t_h}\dom\Gamma^{(N)}_h(p)\to\frac{1}{t_h} X
\end{equation} by
\begin{equation}
\tilde\Gamma^{(N)}_h(p)(s)=\Gamma^{(N)}_h(p)(t_h\cdot s).
\end{equation}
From property (3) we obtain
\begin{equation}
  \lebmeas(\dom\tilde\Gamma^{(N)}_h \cap\ball 0,\frac{\tau^{(N)}_h}{t_h}.)\ge
2\frac{\tau^{(N)}_h}{t_h}(1-\frac{1}{h}). 
\end{equation}
 The functions $\tilde\Gamma^{(N)}_h(p)$ are still  $C$-Lipschitz with
 respect to the rescaled metrics, and $\tilde\Gamma^{(N)}_h(p)(0)=p$
 and  by
a variant of Ascoli--Arzel\`a we get a $C$-Lipschitz
\begin{equation}
  \tilde\Gamma^{(N)}_\infty:\real\to Y
\end{equation} with $\tilde\Gamma^{(N)}_\infty(0)=p$. For $s\in\real$
we choose $s_h\in\dom \tilde\Gamma^{(N)}_h(p)$ converging to $s$ and
observe that by (4) 
\begin{equation}\label{der_conv}
  \left|\frac{f\circ\tilde\Gamma^{(N)}_h(p)(s_h)-f(p)}{t_h}-\varphi^{(N)}_h(p)\cdot
  s_h\right|\le\frac{|s_h|}{h},
\end{equation} which implies 
\begin{equation}
  g\circ\tilde\Gamma^{(N)}_\infty(s)=w_N\cdot s.
\end{equation}
A similar argument involving the metric derivative and
$\psi^{(N)}_h(p)$ shows that
\begin{equation}
  \dist\tilde\Gamma^{(N)}_\infty(s),\tilde\Gamma^{(N)}_\infty(s').=\delta_N|s-s'|.
\end{equation} If $s\in\real$, because of \eqref{rec_dens}, there are
$s_h\in \dom\tilde\Gamma^{(N)}_h$ converging to $s$ such that
$\Gamma^{(N-1)}_h\left(\Gamma^{(N)}_h(p)(s_h)\right)$ is defined. Let
\begin{equation}
    \tilde\Gamma^{(N-1)}_h:\frac {1}{t_h}\dom
\Gamma^{(N-1)}_h\left(\Gamma^{(N)}_h(p)(s_h)\right)
\to\frac{1}{t_h} X
\end{equation}
be defined by
\begin{equation}
  \tilde\Gamma^{(N-1)}_h(\sigma)=\Gamma^{(N-1)}_h\left(\Gamma^{(N)}_h(p)(s_h)\right)(t_h\cdot \sigma).
\end{equation} These functions are $C$-Lipschitz with respect to the
rescaled metrics and a variant of Ascoli--Arzel\`a yields a
$C$-Lipschitz 
\begin{equation}
    \tilde\Gamma^{(N-1)}_\infty:\real\to Y
\end{equation} with
$\tilde\Gamma^{(N-1)}_\infty(0)=\tilde\Gamma^{(N)}_\infty(s)$. There
is an analogue of \eqref{der_conv} where $\varphi^{(N)}_h(p)$ is
replaced by $\varphi^{(N-1)}_h\left(\Gamma^{(N)}_h(p)(s_h)\right)$; as
for $h$ sufficiently large we have:
\begin{equation}
  \left|\varphi^{(N-1)}_h\left(\Gamma^{(N)}_h(p)(s_h)\right)-
\varphi^{(N-1)}_h(p)\right|<\frac{1}{h},
\end{equation} we conclude that:
\begin{equation}
  g\circ \tilde\Gamma^{(N-1)}_\infty(\sigma)-g\circ
  \tilde\Gamma^{(N-1)}_\infty(0)=w_{N-1}\cdot\sigma.
\end{equation} A similar argument involving the metric derivative
shows that
\begin{equation}
  \dist \tilde\Gamma^{(N-1)}_\infty(s),\tilde\Gamma^{(N-1)}_\infty(s').=\delta_{N-1}\left|s-s'\right|.
\end{equation} Continuing inductively we conclude that there is a map
$\tilde\Gamma\colon\real^N\to Y$ such that $\tilde\Gamma(0)=q$ and,
for each $(0\le k\le N-1)$ and each $s\in\real^k$, the map
$\tilde\gamma(t)=\tilde\Gamma((0,t,s))$ is a $\delta_{N-k}$-constant
speed geodesic satisfying:
\begin{equation}
  g\circ\gamma(t)-g\circ\gamma(0)=w_{N-k}\cdot t.
\end{equation} As the vectors $w_i$ are independent, the map $g$ is surjective.
\end{proof}
\par We need a Lemma relating the Assouad dimension of a space to that of a
blow-up \cite[Prop.~6.1.5]{tyson_mackay_conf}:
\begin{lem}\label{assouad_tang}
  If $X$ has Assouad dimension $\le D$ and if $(Y,q)\in\tang(X,p)$,
  then $Y$ has Assouad dimension $\le D$.
\end{lem}
\par The following {lemma} provides a lower bound on the Assouad
dimension.
\begin{lem}\label{assouad_surj}
  If $Y$ has Assouad dimension $\le D$ (or Hausdorff dimension $\le D$) and if there is a surjective Lipschitz
  map $g:Y\to\real^N$, then $D\ge N$.
\end{lem}
\begin{proof}
  The argument in \cite[Subsec.~8.7]{heinonen_analysis} shows that the
  Assouad dimension of $Y$ is at least its Hausdorff dimension. Now
  Lipschitz maps do not increase the Hausdorff dimension so the
  Hausdorff dimension of $Y$ is at least the Hausdorff dimension of $g(Y)=\real^N$.
\end{proof}
\par We can now prove Corollary \ref{derbound}.
\begin{cor}\label{derbound}
  If $X$ is a metric space with Assouad dimension $D$ and if $\mu$
  is a Radon measure, then $\wder\mu.$ has index locally bounded by
  $D$.
\end{cor}
\begin{proof}
  We can reduce to the hypothesis of Theorem \ref{alberti_blow_up}
  because of Corollary \ref{der-alb} and then apply Lemmas
  \ref{assouad_tang}, \ref{assouad_surj}.
\end{proof}
\begin{rem}
  Note that using ultralimits one can replace in Corollary
  \ref{derbound} the assumption on the Assouad dimension with a uniform
  upper bound $D$ on the Hausdorff dimension of the blow-ups of $X$.
\end{rem}
\subsection{Construction of independent Lipschitz
  functions}\label{subsec:constr-indep-lipsch}
The goal of this subsection is to prove Theorem \ref{lip_ind}. We will
use the following Truncation Lemma (\cite[Lem.~4.1]{bate-diff}).
\begin{lem}\label{lip_trunc}
  Let $\epsi\in(0,h/4)$ and assume that $S\subset X$ is Borel and
  $f\colon X\to\real$ is  
$L$-Lipschitz. Then there is an $L$-Lipschitz function $g$, constructed from
$f$, such that:
  \begin{enumerate}
  \item The function $g$ satisfies $0\le g\le h$.
    \item The function $g$ is supported in $\ball S,\frac{2h}{L}.$,
      which is the set of points at distance $<\frac{2h}{L}$ from $S$.
      \item If $x,y\in\ball S,\frac{h}{L}.$,
        \begin{equation}
          |f(x)-f(y)|\ge|g(x)-g(y)|.
        \end{equation}
        \item There is a Borel subset $S'\subset S$ with
          \begin{equation}
            \mu(S')\ge\left(1-\frac{4\epsi}{h}\right)\mu(S)
          \end{equation}
such that if $x\in S$ and $\dist x,y.\le\epsi/L$, then
\begin{equation}
  |f(x)-f(y)|=|g(x)-g(y)|.
\end{equation}
  \end{enumerate}
\end{lem}
\begin{proof}[Proof of Theorem \ref{lip_ind}]
The strategy of the proof is to construct Lipschitz functions that at
some scales experience a definite lower bound on their variation, but
at other scales have a smaller Lipschitz constant. To construct such
functions one uses as building blocks Lipschitz functions which have
small Lipschitz constants below some scale, and then combines them
in a series.
\par Choose $m_1$ such that $\frac{1}{m_1}<\frac{\alpha^2}{2^5L}$ and
let $h_1=\frac{\alpha^2}{4}$ and $\epsi_1=\frac{L}{m_1}$. We use Lemma~\ref{lip_trunc} to find $g_1$ and a %\textcolor{red}
{Borel} set $S_1\subset S$ such that the following
holds:
\begin{enumerate}
\item We have the inequalities $\mu(S_1)\ge\left(1-\frac{4\epsi_1}{L}\right)\mu(S)\ge\frac{1}{2}\mu(S)$.
\item The function $g_1$ is supported in $\ball S,\frac{2h_1}{L}.=\ball S,\frac{\alpha^2}{
2L}.$ and for each ${x,y\in\ball S,\frac{h_1}{L}.}$
  \begin{equation}
    |f_{m_1}(x)-f_{m_1}(y)|\ge|g_1(x)-g_1(y)|.
  \end{equation}
\item If $B$ is a ball of radius $\rho_{m_1}$ centred at
some point of $S$, the Lipschitz constant of $g_1|B$ is at most $\frac{1}{m_1}$.
\item As $\frac{\epsi_1}{L}=\frac{1}{m_1}$, for each $x\in S_1$ there is
$x_1\in \ball x,\frac{1}{m_1}.$ with
  \begin{equation}
    \left|g_1(x)-g_1(x_1)\right|\ge\delta_0\dist x,x_1.>0.
  \end{equation}
\end{enumerate}
We define $g_{k+1}$ inductively in the following way. Having chosen $m_{k+1}$
such that
\begin{equation}
  \frac{1}{m_{k+1}}<\frac{\alpha^2}{2^{(k+1)+4}L}\rho_{m_k},
\end{equation}
we let $h_{k+1}=\frac{\alpha^2}{4}\rho_{m_k}$ and $\epsi_{k+1}=\frac{L}{m_{k+1}
}$. We use Lemma \ref{lip_trunc} to find $g_{k+1}$ and
% \textcolor{red}
{a Borel $S_{k+1}\subset S$ }
such that the following holds:
\begin{enumerate}
\item We have the inequalities $\mu(S_{k+1})\ge\left(1-\frac{4\epsi_{k+1}}{L}\right)\mu(S)\ge
\left(1-\frac{\rho_{m_k}}{2^{k+1}}\right)\mu(S)$.
\item The function $g_{k+1}$ is supported in $\ball S,\frac{2h_{k+1}}{
L}.=\ball S,\frac{\alpha^2\rho_{m_k}}{2L}.$ and for each $x,y\in \ball S,\frac{h_{k+1}}{L}.$, 
  \begin{equation}
    |f_{m_{k+1}}(x)-f_{m_{k+1}}(y)|\ge|g_{k+1}(x)-g_{k+1}(y)|.
  \end{equation}
\item  If $B$ is a ball of radius $\rho_{m_{k+1}}$ centred at
some point of $S$, the Lipschitz constant of $g_{k+1}|B$ is at most $\frac{1}{m_{k+1}}$.
\item As $\frac{\epsi_{k+1}}{L}=\frac{1}{m_{k+1}}$, for each $x\in S_{k+1}$ there is
$x_{k+1}\in \ball x,\frac{1}{m_{k+1}}.$ with
  \begin{equation}
    \left|g_{k+1}(x)-g_{k+1}(x_{k+1})\right|\ge\delta_0\dist x,x_{k+1}.>0.
  \end{equation}
\end{enumerate}
Note that from the definition of $m_k$ one can verify by induction that
\begin{equation}
  \frac{1}{m_k}\le\left(\frac{\alpha^2}{L}\right)^k 2^{-\left(\frac{k(k+1)}
{2}+4k\right)};
\end{equation}
note also that
\begin{align}
  h_{s+1}&=\frac{\alpha^2}{4}\rho_{m_s}\le\frac{\alpha^2}{4}\frac{1}{m_s},\\
  h_{s+k}&=\frac{\alpha^2}{4}\left(\frac{\alpha^2}{L}\right)^{k-1}
  2^{-(s+5)}\rho_{m_s}\text{ (here $k>1$)};
\end{align}
so that
\begin{equation}\label{intexx}
  \sum_{k\ge s+1}h_k\le\frac{\alpha^2}{4}\left(
1+2^{-(s+5)}\frac{\alpha^2/L}{1-\alpha^2/L}\right)\rho_{m_s}.
\end{equation}
This implies that if $n_k\nearrow\infty$, then $\varphi=\sum_k g_{n_k}<\infty$ 
defines a continuous function. We show that $\varphi$ is Lipschitz with
Lipschitz constant 
\begin{equation}
  3\left(L+\alpha+
  \frac{\alpha^2/L}{1-\alpha^2/L}(1+2^{-6}\alpha)\right).
\end{equation}
We want to bound $|\varphi(x)-\varphi(y)|$, and we will consider two
distinct cases.
In the first case, we assume $x\in S$. If $y\ne x$ assume there is some $s$
such that $\dist x,y.\in(\frac{\alpha}{2}\rho_{m_s},\rho_{m_s}]$. If
$t\le s$ we have
$y\in\ball x,\rho_{m_t}.$ and thus we obtain
  \begin{equation}
    \left|g_{t}(x)-g_{t}(x_{t})\right|\le\frac{1}{m_t}\dist x,y.;
  \end{equation}
in particular, from the estimate for $\frac{1}{m_k}$ we deduce that
\begin{equation}\label{interxx}
  \sum_{k}\frac{1}{m_k}\le\frac{\alpha^2/L}{1-\alpha^2/L},
\end{equation}
so that
  \begin{equation}\label{interfxx}
    \sum_{t\le s}\left|g_{t}(x)-g_{t}(x_{t})\right|\le
\frac{\alpha^2/L}{1-\alpha^2/L}\dist x,y. .
  \end{equation}
For the second estimate we use the bound~(\ref{intexx}):
\begin{equation}\label{tail_bound}
  \begin{split}
    \sum_{t\ge s+1} \left|g_{t}(x)-g_{t}(x_{t})\right|&\le
    2\sum_{t\ge s+1}h_t\\&
    \le\frac{\alpha^2}{2}\left(
1+2^{-(s+5)}\frac{\alpha^2/L}{1-\alpha^2/L}\right)\rho_{m_s}\\
&\le\alpha\left(
1+2^{-(s+5)}\frac{\alpha^2/L}{1-\alpha^2/L}\right)\dist x,y..
\end{split}
\end{equation}
We therefore get the bound
\begin{equation}
  \left|\varphi(x)-\varphi(y)\right|\le\left(\alpha+
  \frac{\alpha^2/L}{1-\alpha^2/L}(1+2^{-6}\alpha)\right)\dist x,y..
\end{equation}
The other possibility is that $\dist x,y.\in(\rho_{m_{s+1}},\frac{\alpha}{2}
\rho_{m_s}]$ for some $s$ or $\dist x,y.>\rho_{m_1}$. In this case we can
estimate as above except that for $g_{s+1}$ we must use the full Lipschitz
constant $L$. We therefore have:
\begin{equation}
  \left|\varphi(x)-\varphi(y)\right|\le\left(L+\alpha+
  \frac{\alpha^2/L}{1-\alpha^2/L}(1+2^{-6}\alpha)\right)\dist x,y..
\end{equation}
In the second case $x,y\not\in S$. We use that the functions $g_s$ are 
supported on a neighbourhood of $S$. Without loss of generality we can
assume $\dist x,S.\le\dist y,S.$. Now, for some $s$ we have
$\dist x,S.\in[\frac{\alpha^2}{2L}\rho_{m_{s+1}},\frac{\alpha^2}{2L}\rho_{
m_s})$ (for $s=0$ we let $\rho_{m_0}=1$)
 or $\dist x,S.\ge\frac{\alpha^2}{2L}$. So for $t\ge s+2$ we have that $g_t(x)=g_t(y)=0$ and in the last case $\varphi(x)=\varphi(y)=0$ so there is nothing
to prove. We first assume that $\dist x,y.<\frac{\alpha^2}{2L}\rho_{m_s}$. Then
the hypothesis on $\alpha$ implies that $\alpha^2/L<1$ so that
$x,y$ belong to a ball centred on $S$ of radius at most $\rho_{m_s}$. For
$t\le s$ the Lipschitz constant of $g_t$ is then at most
$\frac{1}{m_t}$ and thus~(\ref{interxx})
implies:
\begin{equation}
  \sum_{t\le s}\left|g_t(x)-g_t(y)\right|\le\frac{\alpha^2/L}{1-\alpha^2/L}
\dist x,y..
\end{equation}
We therefore have:
\begin{equation}
    \left|\varphi(x)-\varphi(y)\right|\le\left(L+\frac{\alpha^2/L}{1-\alpha^2/L}\right)\dist x,y..
\end{equation}
We now assume $\dist x,y.\ge\frac{\alpha^2}{2L}\rho_{m_s}$. We can find
$\tilde x\in S$ such that $\dist x,\tilde x.\le\frac{\alpha^2}{2L}\rho_{m_s}$
so that
\begin{align}
    \left|\varphi(x)-\varphi(\tilde x)\right|&\le\left(L+\alpha+
  \frac{\alpha^2/L}{1-\alpha^2/L}(1+2^{-6}\alpha)\right)\dist x,\tilde x.\\
\left|\varphi(y)-\varphi(\tilde x)\right|&\le\left(L+\alpha+
  \frac{\alpha^2/L}{1-\alpha^2/L}(1+2^{-6}\alpha)\right)\dist y,\tilde x..
\end{align}
Note now that $\dist \tilde x,x.
\le\dist x,y.$ and $\dist \tilde x,y.\le2\dist x,y.$. Therefore:
\begin{equation}
  \left|\varphi(x)-\varphi(y)\right|\le 3\left(L+\alpha+
  \frac{\alpha^2/L}{1-\alpha^2/L}(1+2^{-6}\alpha)\right)\dist x,y..
\end{equation}
We now pass to the construction of $M$ independent functions. We let
\begin{equation}
  \psi_j=\sum_{\text{$k\equiv j\bmod{M}$}}g_k
\end{equation}
and 
\begin{equation}
  \tilde S_j=\bigcap_{k}\bigcup_{\substack{s\ge k\\ s\equiv j\bmod{M}}}
S_s,
\end{equation}
which is a full measure Borel subset of $S$. In particular $S'=\bigcap_j 
\tilde S_j$ is a full measure Borel subset of $S$. Let $\tilde x\in S'$
and assume that $|\lambda_0|\ge\max_{0\le i\le M-1}|\lambda_i|$.
For each $k$ there is an $n_k\ge k$ with $x\in S_{n_k}$ and $n_k\equiv0
\bmod{M}$. We can then find a point $x_{n_k}$ such that
\begin{equation}
  \left| g_{n_k}(x)-g_{n_k}(x_{n_k})\right|\ge\delta_0\dist x,x_{n_k}.>0.
\end{equation}
Then
\begin{equation}
  \left|\sum_{j=0}^{M-1}\lambda_j\psi_j(x)-\lambda_j\psi_j(x_{n_k})
\right|\ge|\lambda_0|\left(\delta_0\dist x,x_{n_k}.-\sum_{t\ne n_k}
\left|g_t(x)-g_t(x_{n_k})\right|\right).
\end{equation}
The terms for $t>n_k$ can be bounded using \eqref{tail_bound} to get:
\begin{equation}
  \sum_{t> n_k}
\left|g_t(x)-g_t(x_{n_k})\right|
\le\alpha\left(
1+2^{-(n_k+5)}\frac{\alpha^2/L}{1-\alpha^2/L}\right)\dist x,x_{n_k}..
\end{equation}
Note that for $t<n_k$, $\dist x,x_k.<\rho_{m_t}$ so
by~(\ref{interfxx}) we have the following
bound on the terms for $t<n_k$:
\begin{equation}
    \sum_{t< n_k}
\left|g_t(x)-g_t(x_{n_k})\right|\le\frac{\alpha^2/L}{1-\alpha^2/L}
\dist x,x_{n_k}..
\end{equation}
Letting $k\nearrow\infty$ we conclude that
\begin{equation}
    \biglip\left(\sum_{i=0}^{M-1}\lambda_i\psi_i\right)(x)\ge
  \left(\delta_0-\frac{\alpha^2/L}{1-\alpha^2/L}-\alpha
\right)|\lambda_0|.
\end{equation}
\end{proof}
\begin{rem}\label{rem:liplipvio} %\textcolor{red}
{Let $\varphi_\alpha$ denote the
  function $\varphi$ that we constructed in the previous proof for a
  given value of $\alpha$. We observe that by varying $\alpha$ one can
  violate the Lip-lip inequality. Specifically, consider what happens} to $\varphi_\alpha$ as $\alpha\searrow0$:
we get a full measure Borel subset $S'\subset S$ with
\begin{equation}
  \inf_{x\in S'}\biglip\varphi_\alpha(x)\ge  \left(\delta_0-\frac{\alpha^2/L}{1-\alpha^2/L}-\alpha
\right)
\end{equation}
and if $x\in S'$ and $r\in(\frac{\alpha}{2}\rho_{m_s},\rho_{m_s}]$,
\begin{equation}
  \varlip\varphi_\alpha(x,r)\le \left(\alpha+
  \frac{\alpha^2/L}{1-\alpha^2/L}(1+2^{-6}\alpha)\right);
\end{equation}
in particular, on a full measure subset of $S$ the Lip-lip inequality
\eqref{eq_Lip_lip_ineq} is 
violated.
\end{rem}
\bibliographystyle{alpha}
\bibliography{der_alb_biblio}
				% note that in
                                % concatenating files for the
                                % bibliography, white spaces are AVOIDED
\end{document}